\documentclass[10pt, reqno]{amsart}
 \usepackage{lipsum}
\usepackage{a4wide}
 \usepackage{amssymb} 
\usepackage{amsmath}
\usepackage{amsthm} 
\usepackage{tikz} 
\usepackage{amscd}
\usepackage{hyperref}
\usepackage{enumitem}
\hypersetup{colorlinks,linkcolor={blue},citecolor={blue},urlcolor={red}}
\theoremstyle{plain}
\usepackage[margin=2.6cm]{geometry}
 \numberwithin{equation}{section}

\newcommand{\id}{\operatorname{id}}

\newcommand{\sch}[1]{\operatorname{{\bf #1}}}

\newcommand{\ind}{\operatorname{ind}}

\newcommand{\res}{\operatorname{res}}
\newcommand{\ho}{\operatorname{Hom}}

\newcommand{\diag}{\operatorname{diag}}
\newcommand{\tr}{\operatorname{tr}}
\newcommand{\End}{\operatorname{End}}

\newtheorem{theorem}{Theorem}[subsection]

\newtheorem{lemma}[theorem]{Lemma}
\newtheorem{remark}[theorem]{Remark}
\newtheorem{proposition}[theorem]{Proposition}

\newtheorem{claim}{Claim}

 \makeatletter
\newcommand{\proofpart}[2]{\par
  \noindent\emph{Part #1: #2}\par\nobreak
  \@afterheading } \makeatother 
\title{Generic cuspidal representations of $U(2,1)$}
 \author{\large{Santosh Nadimpalli }} 
 \date{\today}
\begin{document}
\begin{abstract}
  Let $F$ be a non-Archimedean local field and let $\sigma$ be a
  non-trivial Galois involution with fixed field $F_0$. When the
  residue characteristic of $F_0$ is odd, using the construction of
  cuspidal representations of classical groups by Stevens, we classify
  generic cuspidal representations of $U(2,1)(F/F_0)$.
\end{abstract}
\maketitle
\section{Introduction}
Let $F$ be a non-Archimedean local field and let $\sigma$ be a
non-trivial Galois involution with fixed field $F_0$. Let $V$ be a
three dimensional $F$-vector space, and let $h:V\times V\rightarrow F$
be a non-degenerate hermitian form with
\begin{equation}\label{Hermitian_form_def}
  h(xv, yw)=x\sigma(y)\sigma(h(w, v))
  \ \text{for all}\ v, w\in V\ \text{and}
  \ x, y\in F.
\end{equation}
Let $\sigma_h$ be the adjoint anti-involution induced by $h$ on
$\End_F(V)$. Let $\sch{G}$ be the unitary group associated to
$(V, h)$, and let $G$ be the group of $F_0$-rational points of
$\sch{G}$. In this article we obtain a classification of generic
cuspidal representations of $G$ using the underlying skew semisimple
strata, in the construction of cuspidal representations of $G$.

Let $\sch{H}$ be a quasi-split reductive algebraic group defined over
$F_0$.  Let $\sch{B}$ be a Borel subgroup of $\sch{H}$, defined over
$F_0$, and let $\sch{U}$ be the unipotent radical of $\sch{B}$. Let
$\sch{T}$ be a maximal $F_0$-split torus of $\sch{H}$ contained in
$\sch{B}$. For any algebraic group $\sch{R}$, defined over $F_0$, we
denote by $R$ the group of $F_0$-rational points of $\sch{R}$. A
character $\psi$ of $U$ is said to be non-degenerate if $\psi$ is
non-trivial on each simple root group of $T$ in $U$.  A representation
$\pi$ of $H$ is called generic if $\ho_{U}(\pi, \Psi)\neq 0$, for some
non-degenerate character $\Psi$ of $U$.

A cuspidal representation $(\pi, V)$ of $H$ is a generic
representation if and only if there exists a non-zero linear
functional $l:V\rightarrow \mathbb{C}$ and a character $\Psi$ of $U$
such that
$$l(\pi(u)v)=\Psi(u)l(v),\ \text{for all}\ u\in U, v\in V.$$
If such a linear functional exists, for $(\pi, V)$, then the character
$\Psi$ necessarily satisfies the non-degenerate condition on the
character $\Psi$. The functional $l$ is called a Whittaker linear
functional.  Moreover, when $H$ equals $G$, genericity of an
irreducible smooth representation does not depend on the choice of the
pair $(U, \Psi)$. A Whittaker linear functional on an irreducible
smooth representation, if it exists, is unique up to scalars (see
\cite{shalika_mul_one} and \cite{rodier_whit_model}).

Whittaker functionals are first used by Jacquet and Langlands to
define the local $L, \epsilon$-factors for ${\rm GL_2}(F_0)$
(see \cite[Theorem 2.18]{auto_forms_gl_2}). These methods and their
generalisations have played a fundamental role in the Langlands
program, and especially in the theory of automorphic $L$-functions
(see \cite{analytic_prop_l_fns_sha_gel} for a survey).  The local
Langlands correspondence gives a natural partition of irreducible
smooth representations of $H$ into finite sets called the
$L$-packets. It is conjectured by Shahidi that there exists a unique
generic representation inside an $L$-packet consisting of irreducible
tempered representations (see \cite[Conjecture
9.4]{shahidi_planc_lang_conj}). When the characteristic of $F$ is
zero, the local Langlands correspondence for the group $G$ is
established by Rogawski in the book \cite{llc_u3}, and Shahidi's
conjecture, for the group $G$, is proved in the paper
\cite{gen_unit_3_var}.  We hope that the results of this article are
useful in understanding an explicit version of the local Langlands
correspondence for $G$, relating the $L$-packets of cuspidal
representations to their inducing data.

Every cuspidal representation of ${\rm GL}_n(F_0)$ is generic (see
\cite[Chapter 3, 5.18]{Bernstein_zelevinsky_0}), but this is no longer
true for classical groups.  The classification of generic positive
depth cuspidal representations from the inducing data is expected to
depend on some subtle arithmetic aspects of the inducing data. When
the characteristic of $F_0$ is zero, $F/F_0$ is unramified, and the
cardinality of the residue field of $F_0$ is odd, Murnaghan classified
the generic cuspidal representations of $G$ in the article
\cite[Theorem 7.13]{murnaghan_char_emp_u_3}.  The methods used in
\cite{murnaghan_char_emp_u_3} are based on character formulas for
cuspidal representations--the Murnaghan--Kirillov formula--and using a
local character expansion to relate with Shalika germs. DeBacker and
Reeder also studied genericity of very cuspidal representations,
arising from an unramified torus, of an unramified $p$-adic group (see
\cite{reeder_debacker_very_cusp_gen}). Blondel and Stevens, using
different techniques from Murnaghan, have classified generic cuspidal
representations of ${\rm Sp}_4(F_0)$, for a non-Archimedean local
field $F_0$ with odd residue characteristic (see
\cite{sp_4_genericity}). The methods of this article are inspired by
the work of Blondel and Stevens on ${\rm Sp}_4(F_0)$.

The explicit construction of cuspidal representations of $G$, when
$F/F_0$ is an unramified extension of $p$-adic fields, goes back to
the work of Moy and Jabon in the articles \cite{moy_u3} and
\cite{jabon_u3_gsp_4} respectively. Later, Blasco, in the article
\cite{blasco_u3}, gave an explicit construction of cuspidal
representations in the line of Bushnell--Kutzko's work on the
admissible dual of ${\rm GL}_n$. In this article, we use the
generalisation of Bushnell--Kutzko construction of cuspidal
representations to classical groups by Stevens, culminating in the
paper \cite{cusps_classical}.

We now describe the results of this article using the language of
strata from the theory of types (see Section \ref{sec_cusp_reps} and
references in {\it loc.cit}).  Let
$\mathfrak{x}=[\Lambda, n,0, \beta]$ be any skew semisimple stratum in
$\End_{F}(V)$, in particular, $\Lambda$ is a lattice sequence, $n$ is
a non-negative integer, $\beta\in \End_F(V)$ with
$\sigma_h(\beta)=-\beta$, and the $G$-stabilizer of $\beta$--for the
adjoint action of $G$ on its Lie algebra--is isomorphic to a product
of unitary groups.  Let $\Pi_\mathfrak{x}$ be the set of all cuspidal
representations containing a type, in the sense of Bushnell--Kutzko,
constructed from the stratum $\mathfrak{x}$. Let $\psi$ be a fixed
non-trivial additive character of $F$, and let $\psi_\beta$ be the
function sending $X\in \End_F(V)$ to $\psi(\tr(\beta(\id_V-X)))$. Let
$\mathfrak{X}_\beta(F_0)$ be the set of all $F_0$-rational Borel
subgroups $\sch{B}$ of $\sch{G}$ such that $\psi_\beta$ defines a
character on the group of $F_0$-rational points of the unipotent
radical of $\sch{B}$. The main result of this article is the following
theorem:
\begin{theorem}\label{intro_main_thm}
  Let $F$ be a non-Archimedean local field with odd residue
  characteristic.  Let $\mathfrak{x}=[\Lambda, n, 0, \beta]$ be any
  skew semisimple stratum with $n>0$. The cuspidal representations in
  the set $\Pi_\mathfrak{x}$ are either all generic or all
  non-generic. If $\mathfrak{X}_\beta(F_0)$ is empty, then every
  cuspidal representation in the set $\Pi_\mathfrak{x}$ is
  non-generic. Except when $\beta$ has a non-degenerate two
  dimensional eigenspace, a cuspidal representation in the set
  $\Pi_\mathfrak{x}$ is generic if and only if
  $\mathfrak{X}_\beta(F_0)$ is non-empty. If $\beta$ has a
  non-degenerate two dimensional eigenspace, $V_2$, then the set
  $\mathfrak{X}_\beta(F_0)$ is non-empty if and only if $(V_2, h)$ is
  isotropic. However, every cuspidal representation in the set
  $\Pi_\mathfrak{x}$ is non-generic.
\end{theorem}
The set $\mathfrak{X}_\beta(F_0)$ is the set of $F_0$-rational points
of a closed subvariety $\mathfrak{X}_\beta$ of the variety of Borel
subgroups of $\sch{G}$. We determine necessary and sufficient
conditions on $\beta$ for the non-emptiness of
$\mathfrak{X}_\beta(F_0)$.  Hence, we obtain a more explicit form of
Theorem \ref{intro_main_thm}, and for this, we refer to Theorem
\ref{summary_main_theorem}. The genericity of depth-zero cuspidal
representations $G$ is well understood (see \cite[Section
6]{depth-zero_debacker_reeder}). However, we recall these results for
giving a complete analysis of genericity of cuspidal representations
of $G$, especially, for those results not stated in the literature,
for instance, when $F/F_0$ is ramified.

In general, the proofs use the explicit construction of cuspidal
representations for classical groups by Stevens in the articles
\cite{semisimple_char_classical} and \cite{cusps_classical} and
Mackey-decomposition to understand the restriction of a cuspidal
representation to a maximal unipotent subgroup of $G$.  Blondel and
Stevens, in the article \cite{sp_4_genericity}, related the set
$\mathfrak{X}_\beta(F_0)$ with the problem of genericity of cuspidal
representations of ${\rm Sp}_4(F_0)$.  We use the approach in
\cite{sp_4_genericity} to classify generic cuspidal representations of
$G$; however, there are significant differences from the case of
${\rm Sp}_4(F_0)$. In the case of ${\rm Sp}_4(F_0)$, the variety
$\mathfrak{X}_\beta$ is a $\mathbb{P}^1$-fibre space over a quadratic
hypersurface--in a $3$-dimensional projective space over $F_0$. Hence,
the problem of finding rational points on $\mathfrak{X}_\beta$ is
reduced to that of a quadratic hypersurface.  For the unitary group in
$3$-variables, the variety $\mathfrak{X}_\beta$ is the intersection of
two quadratic hypersurfaces in a $5$-dimensional projective space over
$F_0$. This poses some arithmetic difficulties in understanding the
$F_0$-rational points on $\mathfrak{X}_\beta$. The other significant
difference between ${\rm Sp}_4(F_0)$ and $G$ is that the isomorphism
class of a non-degenerate subspace of $(V, h)$ is not determined by
its dimension. For instance, a $2$-dimensional non-degenerate subspace
of $(V, h)$ can be either isotropic or anisotropic. The problem of
genericity depends on these differences.

We briefly sketch the contents of each section. The
algebra $F[\beta]$ is a direct sum of fields, say
$$F[\beta]=F[\beta_1]\oplus F\beta_2]\oplus\cdots\oplus F[\beta_k]$$
with $\beta=\sum_{i=1}^k\beta_i$ and $\sigma_h(\beta_i)=-\beta_i$, for
$1\leq i\leq k$. This decomposition of $F[\beta]$ corresponds to a
maximal orthogonal decomposition of $V=\perp_{i=1}^k V_i$, for the
property that $F[\beta]$ acts on $V_i$ via its projection onto
$F[\beta_i]$, for $1\leq i\leq k$. The above decomposition is unique
and is determined by $\beta$. In Section \ref{prelims}, we set up some
preliminaries to prove non-genericity results. In Section
\ref{sec_cusp_reps}, we review some useful results from the
construction of cuspidal representations of $G$.  In Section
\ref{the-simple_case}, we consider the case where $F[\beta]$ is a
field. When the characteristic of $F_0$ is zero, the $L$-packet
containing a cuspidal representation from the set $\Pi_\mathfrak{x}$
has cardinality $1$. Hence, any representation in the set
$\Pi_\mathfrak{x}$ is expected to be generic. We first prove that
$\mathfrak{X}_\beta(F_0)$ is non-empty, and using this, we will show
that every representation in the set $\Pi_\mathfrak{x}$ is generic.

In Section \ref{type_B_sec}, we consider the case where $F[\beta]$ is a
$3$-dimensional algebra, $F[\beta]=F[\beta_2]\oplus F[\beta_1]$, and
$\beta=\beta_2+\beta_1$ such that $[F[\beta_2]:F]=2$. We will
completely determine when $\mathfrak{X}_\beta(F_0)$ is non-empty and
this depends only on the valuation of $\beta_i$, in the field
$F[\beta_i]$, and on the isomorphism class of the hermitian space
$(V_2, h)$. Then we will use these results to show that a
representation in the set $\Pi_\mathfrak{x}$ is generic if and
only if the set $\mathfrak{X}_\beta(F_0)$ is non-empty. In this case,
we sometimes have to find a nice integral model of
$\mathfrak{X}_\beta$ and lift points from its special fibre.

In Sections \ref{type_C_sec} and \ref{type_D_sec}, we treat the cases
where $F[\beta]$ is a direct sum of two copies of $F$ and three copies
of $F$, respectively. The strategy is similar to that of the previous
sections. But, in Section \ref{type_C_sec}, we will see examples when
$\mathfrak{X}_\beta(F_0)$ is non-empty, and yet every representation in
the set $\Pi_\mathfrak{x}$ is non-generic. We note that $\beta$ is not
a regular semisimple element in this case. 

In Section \ref{type_D_sec}, we have $\beta=\beta_1+\beta_2+\beta_2$,
with $\beta_i\in F$ and $\sigma(\beta_i)=-\beta_i$. When $F/F_0$ is
unramified, the non-emptiness of the set $\mathfrak{X}_\beta(F_0)$
depends only on the valuations of $\beta_i$. However, when $F/F_0$ is
ramified the information on the valuations of $\beta_i$, for
$1\leq i\leq 3$, is not enough to determine whether the set
$\mathfrak{X}_\beta(F_0)$ is empty or not. Although it is fairly easy
to determine the conditions on $\beta$ for the non-emptiness of
$\mathfrak{X}_\beta(F_0)$, these conditions do not involve the natural
invariants of the stratum $\mathfrak{x}$; hence, we did not make them
explicit.  Nonetheless, we will show that a representation in the set
$\Pi_\mathfrak{x}$ is generic if and only if $\mathfrak{X}_\beta(F_0)$
is non-empty.

{\bf Acknowledgements\ } I want to thank Maarten Solleveld for many
useful discussions and clarifications during the course of this
work. The author is supported by the NWO Vidi grant ``A Hecke algebra
approach to the local Langlands correspondence'' (nr. 639.032.528). I
want to thank Shaun Stevens for answering some questions on his paper,
for his interest, and other clarifications. I want to thank Peter
Badea for suggesting some useful references. I want to thank Kam-Fai
Geo Tam for helpful discussions. 
\section{Preliminaries}\label{prelims}
All representations in this article are defined over $\mathbb{C}$-vector
spaces. Let $G, H$ be two groups with $H\subset G$, and let $\rho$ be a
representation of $H$. We denote by $\rho^g$ the representation of
$g^{-1}Hg$ sending $h\in g^{-1}Hg$ to $\rho(ghg^{-1})$. The group
$g^{-1}Hg$ is denoted by $H^g$.

For any real number $x$, we denote by $\lfloor x\rfloor$ the greatest
integer less than or equal to $x$. Let $\lceil x\rceil$ be the
smallest integer greater than or equal to $x$. Let $x+$ be the
smallest integer strictly bigger than $x$ and $x-$ be the greatest
integer strictly smaller than $x$.

Let $(W,h)$ be a pair consisting of an $F$-vector space and a
non-degenerate hermitian form $h$ on $W$. Let $(W', h)$ be a
non-degenerate subspace of $(W, h)$. Then let ${\bf 1}_{W'}$ be the
projection of $W$ onto $W'$ with kernel $(W')^\perp$.
\subsection{}\label{local_fields}
For a non-Archimedean local field $K$, let $\mathfrak{o}_K$ be the
ring of integers of $K$, let $\mathfrak{p}_K$ be the maximal ideal of
$\mathfrak{o}_K$, let $k_K$ be the residue field
$\mathfrak{o}_K/\mathfrak{p}_K$, and $q_K$ denotes the cardinality of
the residue field $k_K$. Let $\nu_K$ be the normalised valuation of
$K$. From now we assume that $q_K$ is odd.

Let $F$ be a non-Archimedean local field with a Galois involution
$\sigma$. Let $F_0$ be the fixed field of $\sigma$. Let $\varpi$ be a
uniformizer of $F$ such that $\sigma(\varpi)=(-1)^{e(F|F_0)-1}\varpi$.
When $F/F_0$ is ramified, we set $\varpi_0$ to be the element
${\rm Nr}_{F/F_0}(\varpi)$, and when $F/F_0$ is unramified, we set
$\varpi_0=\varpi$. The element $\varpi_0$ is a uniformizer of $F_0$.
Let $\psi_0$ be a fixed additive character of $F_0$ with conductor
$\mathfrak{p}_{F_0}$.  The character $\psi_0\circ{\rm tr}_{F/F_0}$ is
denoted by $\psi_F$.  Let $F=F_0[\delta]$, where
$\sigma(\delta)=-\delta$ and $\nu_F(\delta)=e(F|F_0)-1$. Let
$\nu_{F/F_0}$ be the valuation of $F$ extending the normalised
valuation of $F_0$.

For any $F_0$-scheme $\sch{X}$, we denote by $X$ the set of
$F_0$-rational points of $\sch{X}$. If $\sch{H}$ is any linear
algebraic group over $F_0$, then the group $H$ is considered as a
topological group whose topology is induced from the non-Archimedean
metric on $F$.
\subsection{}
Let $V$ be a three dimensional $F$-vector space and let $h$ be a
non-degenerate hermitian form on $V$, as defined in
\eqref{Hermitian_form_def}. Let $\sigma_h$ be the adjoint
anti-involution on $\End_F(V)$ induced by the hermitian form $h$. The
hermitian space $(V,h)$ is isotropic, and we assume that the
determinant of $(V, h)$ is the trivial class in
$F_0^\times/{\rm Nr}_{F/F_0}(F^\times)$.  Let $\sch{G}$ be the unitary
$F_0$-group scheme associated with the pair $(V, h)$. We identify the
Lie algebra, $\mathfrak{g}$, of $\sch{G}$ with
$$\{X\in \End_F(V): \sigma_h(X)=-X\}.$$
From now the algebra $\End_F(V)$ is denoted by $A$.
\subsection{}
A basis, $(e_1, e_0, e_{-1})$, of $(V, h)$ is called a {\it
  Witt-basis} if $h(e_1, e_1)=h(e_{-1}, e_{-1})=0$,
$h(e_1, e_{-1})=1$, and $e_0\in \langle e_1, e_{-1}\rangle^\perp$ with
$h(e_0, e_0)=1$. A basis, $(e_1, e_{-1})$, for a two dimensional
hermitian space $(V', h')$, is called a {\it Witt-basis} if
$h(e_1, e_1)=h(e_{-1}, e_{-1})=0$, and $h(e_1, e_{-1})=1$. Let
$\sch{B}$ be any $F_0$-rational Borel subgroup of $\sch{G}$, and let
$\sch{U}$ be the unipotent radical of $\sch{B}$. Let $\sch{T}$ be a
maximal $F_0$-split torus of $\sch{G}$ contained in $\sch{B}$. Let
$\bar{\sch{U}}$ be the unipotent radical of the opposite Borel
subgroup, $\bar{\sch{B}}$, of $\sch{B}$ with respect to $\sch{T}$. Let
$\sch{Z}$ and $\sch{N}$ be the centraliser and the normaliser of
$\sch{T}$ respectively. We denote by $W_G$ the Weyl group
$\sch{N}/\sch{Z}$.

There exists a Witt-basis $(e_1, e_0, e_{-1})$ of $V$--giving an
embedding of $G$ in ${\rm GL}_3(F)$--such that $B$ stabilises the line
$\langle e_1\rangle$. The groups $T$ and $Z$ are identified with the
groups
$$\{\diag(t, 1, t^{-1}): t\in F_0^\times\}\  \text{and}\ 
\{\diag(z, z', \sigma(z)^{-1}): z, z'\in F^\times,\
z'\sigma(z')=1\}$$ respectively. The groups $U$ and $\bar{U}$ are
identified with the groups
\begin{align*}&\left\{u(c,d):=\begin{pmatrix}
      1&c&d\\0&1&-\sigma(c)\\
      0&0&1\end{pmatrix}: c, d\in F, \
    c\sigma(c)+d+\sigma(d)=0\right\}, \\
  &\left\{\bar{u}(c,d):=\begin{pmatrix}
      1&0&0\\c&1&0\\
      d&-\sigma(c)&1\end{pmatrix}: c, d\in F, \
    c\sigma(c)+d+\sigma(d)=0\right\}
\end{align*}
respectively. The derived groups of $U$ and $\bar{U}$, denoted by
$U_{\text{der}}$ and $\bar{U}_{\text{der}}$ respectively, and they are
identified with the groups $\{u(0, d):d\in F, d+\sigma(d)=0\}$ and
$\{\bar{u}(0, d): d\in F, d+\sigma(d)=0\}$ respectively. Let
$\{U_{\rm der}(r):r \in \mathbb{Z}\}$ be a filtration of compact
subgroups of $U_{\rm der}$ defined as follows:
\begin{equation}
  U_{\rm der}(r):=\{u(0, y): y\in \delta\mathfrak{p}_{F_0}^r\}. 
\end{equation}
Similarly, we set $\bar{U}_{\rm der}(r)$ to be the group $\{\bar{u}(0,
y): y\in \delta\mathfrak{p}_{F_0}^r\}$, for $r\in \mathbb{Z}$.
\subsection{}\label{subsection_generic}
Let $\sch{U}$ be the unipotent radical of an $F_0$-rational Borel subgroup
$\sch{B}$ of $\sch{G}$. An irreducible smooth representation $(\pi,
W)$ of $G$ is called a {\it generic representation} if and only if
there exists a non-zero linear functional $l:W\rightarrow \mathbb{C}$ and a
non-trivial character $\Psi$ of $U$ such that
\begin{equation}
  l(\pi(u)w)=\Psi(u)l(w)\ \text{for all}\ u\in U, w\in W.
\end{equation}
The group $Z$ acts transitively on the set of non-trivial characters
of $U$, and hence the genericity of an irreducible smooth
representation $(\pi, W)$ of $G$ does not depend on a choice of the
pair $(U, \Psi)$. The linear functional $l:W\rightarrow \mathbb{C}$ is
called a {\it Whittaker linear functional}.  We have
$$\dim_\mathbb{C}\ho_{U}(\pi, \Psi)\leq 1,$$
for any irreducible smooth representation $\pi$ of $G$, and a non-trivial 
character $\Psi$ of $U$ (see \cite{shalika_mul_one} and
\cite{rodier_whit_model}).
\subsection{}\label{possible_whit_models}
Let $\beta$ be an element in the algebra $A$. Let $\psi_\beta$ be the
function on $A$ given by
$$\psi_\beta(X)=\psi_F(\text{tr}(\beta(\id_V-X)))\
\text{for all}\ X\in A.$$
Let $V_1\subset V_2\subset V$ be a complete flag of $F$-vector spaces,
and let $P$ be the Borel subgroup of ${\rm GL}_F(V)$ fixing this
flag. Let $Y$ be the unipotent radical of $P$.  If $V_2=V_1^\perp$,
then $Y\cap G$ is the unipotent radical of the Borel subgroup
$P\cap G$ of $G$. The function $\psi_\beta$ is a character of $Y$ if
and only if
\begin{equation}\label{flag_to_character}
\beta V_1\subset V_2.
\end{equation}  
Let $\mathfrak{B}$ be the variety of Borel subgroups of $\sch{G}$. Let
$\mathfrak{X}_\beta(F_0)$ be the following subset of
$\mathfrak{B}(F_0)$:
\begin{equation}
  \mathfrak{X}_\beta(F_0)=\{\sch{B}\in \mathfrak{B}(F_0):
  \psi_\beta\ \text{is a character of}\ 
  {\rm R}_u(\sch{B})(F_0)\}.
\end{equation}
Here, ${\rm R}_u(\sch{B})$ is the unipotent radical of a Borel
subgroup $\sch{B}$ of $\sch{G}$. Note that the set
$\mathfrak{X}_\beta(F_0)$ is the set of $F_0$-rational points of a
closed sub-variety of $\mathfrak{B}$, to be denoted by
$\mathfrak{X}_\beta$.
\subsection{}
The following lemma is frequently used in proving certain cuspidal
representations are non-generic. Let $(e_1, e_0, e_{-1})$ be a
Witt-basis for $(V, h)$ and let $B$ be the Borel subgroup of $G$
fixing the line $\langle e_1\rangle$. Let $U$ be the unipotent radical
of $B$. Using the basis $(e_1, e_0, e_{-1})$, we identify $G$ as a
subgroup of ${\rm GL}_3(F)$.
\begin{lemma}\label{basic_inequality}
  Let $g$ be an element of $G$, and let $r$ be an integer. The
  character $\psi_\beta^g$ of $U_{\rm der}(r)$ is non-trivial if and
  only if
$$\nu_{F_0}(\delta h(ge_{1}, \beta
ge_{1}))\leq -r.$$ Similarly, the character $\psi_\beta^g$ of the
group $\bar{U}_{\rm der}(r)$, is non-trivial if and only if
$$\nu_{F_0}(\delta h(ge_{-1}, \beta ge_{-1}))\leq -r.$$ 
\end{lemma}
\begin{proof}
  We prove the lemma for $U_{\rm der}(r)$, and the other case is
  similar.  Let $X_{\rm der}$ be the $3\times 3$ matrix
$$\begin{pmatrix}
0&0&1\\0&0&0\\0&0&0
\end{pmatrix}.$$ We have the following equality:
$$\psi_\beta(gu(0, \delta d)g^{-1})=\psi_0(d({\rm
  tr}_{F/F_0}(\delta\tr(\beta gX_{\text{der}}g^{-1}))).$$ Note that
$\tr(\beta gX_{\text{der}}g^{-1})$ is equal to
$\tr(g^{-1}\beta gX_{\text{der}})$, and
$\tr(g^{-1}\beta gX_{\text{der}})$ is equal to
$h(ge_{1}, \beta ge_{1})$. Since $\sigma_h(\beta)=-\beta$, we get that
$${\rm tr}_{F/F_0}(\delta h(ge_{1}, \beta ge_{1}))
=2\delta h(ge_{1}, \beta g e_{1}).$$ Hence, the character
$\psi_\beta^g$ is trivial on $U_{\text{der}}(r)$ if and only if the
character $d'\mapsto \psi_0(d'\delta h(ge_{1}, \beta ge_{1}))$, for
$d'\in \mathfrak{p}_{F_0}^r$, is trivial. Since the conductor of
$\psi_0$ is equal to $\mathfrak{p}_{F_0}$, we get the required
inequality.
\end{proof}
\section{Strata and cuspidal representations}
\label{sec_cusp_reps}
In this section, we recall the construction of cuspidal
representations of $G$, via Bushnell--Kutzko's theory of types. We
refer to the articles \cite{semisimple_char_classical},
\cite{cusps_classical} and \cite{miyauchi_stevens} for more details.
\subsection{}\label{lattice_sequences}
An $\mathfrak{o}_F$-{\it lattice sequence}, $\Lambda$, on $V$ is a
function from $\mathbb{Z}$ to the set of $\mathfrak{o}_F$-lattices in
$V$ satisfying the following conditions:
\begin{enumerate}
\item $\Lambda(n+1)\subseteq \Lambda(n)$, for all $n\in \mathbb{Z}$, 
\item there exists a positive integer $e(\Lambda)$ such that
  $\Lambda(n+e(\Lambda))=\mathfrak{p}_F\Lambda(n)$, for any
  $n\in \mathbb{Z}$.
\end{enumerate}
Given any lattice $\mathcal{L}\subset V$, let $\mathcal{L}^\#$ be the
lattice $\{v\in V:\ h(v, \mathcal{L})\subset \mathfrak{p}_F\}$.  For
any lattice sequence $\Lambda$, let $\Lambda^\#$ be the lattice
sequence defined as:
$$\Lambda^\#(n)=\Lambda(-n)^\#,\ \text{for all}\  n\in \mathbb{Z}.$$
A lattice sequence $\Lambda$ is said to be {\it self-dual} if there
exists an integer $d$ such that $\Lambda^\#(n)=\Lambda(n+d)$, for all
$n\in \mathbb{Z}$. Since we only use $\mathfrak{o}_F$-lattice
sequences, we call them directly as lattice sequences.

Let $W$ be a subspace of the vector space $V$, and let $\Lambda$ be a
lattice sequence on $V$. We denote by $\Lambda\cap W$ the lattice
sequence on $W$ sending $n$ to $\Lambda(n)\cap W$.

Given any lattice sequence $\Lambda$ and integers $a, b\in
\mathbb{Z}$, the lattice sequence $a\Lambda+b$ is defined by setting
$$(a\Lambda+b)(n)=\Lambda(\lceil (n-b)/a\rceil),
\ \text{for all}\ n\in \mathbb{Z}.$$ The set of lattice sequences
$\{a\Lambda+b:\ a, b\in \mathbb{Z}\}$ is called the {\it affine
  class} of $\Lambda$. For any self-dual lattice-sequence $\Lambda$,
we can find a lattice sequence $\Lambda'$ in the affine class of
$\Lambda$ such that $e(\Lambda')$ is an even integer, and
$(\Lambda')^\#=\Lambda'-1$. Henceforth, we assume that all self-dual
lattice sequences satisfy these conditions.
\subsection{}
Given any lattice sequence $\Lambda$, and an integer $n$, let
$\tilde{\mathfrak{a}}_{n}(\Lambda)$ be the following sublattice of
$\End_F(V)$:
$$\tilde{\mathfrak{a}}_n(\Lambda)=
\left\{T\in \End_F(V): T\Lambda(i)\subset \Lambda(i+n)\ \forall\ i\in
\mathbb{Z}\right\}.$$ The decreasing sequence of lattices
$\{\tilde{\mathfrak{a}}_{n}(\Lambda):n\geq 0\}$ has trivial
intersection. Given any element $T\in \End_F(V)$, we denote by
$\nu_\Lambda(T)$ the unique integer $k$ such that $T\in
\tilde{\mathfrak{a}}_k(\Lambda)$ and $T\not\in
\tilde{\mathfrak{a}}_{k+1}(\Lambda)$. Let $\tilde{P}_0(\Lambda)$ be
the units in the ring ${\tilde{\mathfrak{a}}}_0(\Lambda)$. For any
positive integer $n$, let $\tilde{P}_n(\Lambda)$ be the compact open
subgroup $\id_V+{\tilde{\mathfrak{a}}}_n(\Lambda)$ of ${\rm
  GL}_F(V)$. For any self-dual lattice sequence $\Lambda$, the
lattices $\tilde{\mathfrak{a}}_{n}(\Lambda)$ are stable under
$\sigma_h$. For any non-negative integer $n$, let $P_n(\Lambda)$ be
the compact open subgroup $\tilde{P}(\Lambda)\cap G$ of $G$.

The group $P_0(\Lambda)/P_1(\Lambda)$ is the set of $k_{F_0}$-rational
points of a (not necessarily connected) reductive algebraic group over
$k_{F_0}$, and let $P^0(\Lambda)$ be the inverse image of the
$k_{F_0}$-rational points of its connected component. The compact
subgroup $P^0(\Lambda)$ is called the parahoric subgroup associated to
$\Lambda$. If $F/F_0$ is unramified, then $P^0(\Lambda)$ is equal to
$P_0(\Lambda)$ and has index $2$ in $P_0(\Lambda)$ otherwise.

A {\it stratum} in $\End_F(V)$ is the data, $[\Lambda, n, r, \beta]$,
consisting of a lattice sequence $\Lambda$ on $V$, integers $n\geq
r\geq 0$, and an element $\beta\in \End_F(V)$ such that $\beta\in
\tilde{a}_{-n}(\Lambda)$. Two strata $[\Lambda, n, r, \beta_1]$ and
$[\Lambda, n, r, \beta_2]$ are said to be equivalent if 
$\beta_2-\beta_1\in \tilde{a}_{-r}(\Lambda)$. A stratum $[\Lambda, n,
r, \beta]$ is called a {\it zero stratum} if $n=r$ and
$\beta=0$. For $n\geq r\geq n/2>0$, the set of equivalence classes
of strata are in bijection with the characters of the group
$\tilde{P}_{r+1}(\Lambda)/\tilde{P}_{n+1}(\Lambda)$. The character
corresponding to the equivalence class containing the stratum
$[\Lambda, n, r, \beta]$ is given by
$$\psi_\beta:\id_V+X\mapsto  
\psi_F({\rm tr}\beta X),\ \text{for}\ X\in \tilde{a}_n(\Lambda).$$

A stratum is called {\it skew} if the lattice sequence $\Lambda$ is
self-dual and $\beta\in \mathfrak{g}$. We have the same notion of
equivalence on skew strata. For $n\geq r\geq n/2> 0$, an equivalence
class of skew strata corresponds to a character on the group
$P_{r+1}(\Lambda)/P_{n+1}(\Lambda)$, given by
$\res_{P_{r+1}(\Lambda)}\psi_\beta$.
\subsection{}\label{strata}
Recall that a stratum $[\Lambda, n, r, \beta]$ is called a {\it simple
  stratum} if it satisfies the
following conditions:
\begin{enumerate}
\item We have $n\geq r\geq 0$,
\item The valuation of $\beta$ with respect to $\Lambda$, denoted by
  $\nu_\Lambda(\beta)$, is equal to $-n$.
\item The algebra $F[\beta]$ is a field and it normalises the lattice
  sequence $\Lambda$.
\item We have $r< -k_0(\Lambda, \beta)$, where $k_0(\Lambda, \beta)$
  is the {\it critical constant} defined in \cite[Section
  1.2.2]{semisimple_char_classical}.
\end{enumerate}
A stratum $[\Lambda, n, r, \beta]$ is called a {\it semisimple
  stratum} if it is either a zero stratum or if it satisfies the
following conditions:
\begin{enumerate}
\item We have $n\geq r\geq 0$ and  $\nu_\Lambda(\beta)=-n$. 
\item There exists a decomposition $V=\oplus_{i=1}^kV_i$ for which
  $\Lambda(k)=\oplus_{i=1}^k(\Lambda(k)\cap V_i)$, for all $k\in
  \mathbb{Z}$.
\item Let ${\bf 1}_i$ be the projection of $V$ onto $V_i$ with kernel
  $\oplus_{j\neq i}V_j$. We have $\beta=\sum_{i=1}^k\beta_i$,
  where $\beta_i={\bf 1}_i\beta {\bf 1}_i$, for $1\leq i\leq k$.
\item The stratum $[\Lambda_i, q_i, r, \beta_i]$--with $q_i=r$ if
  $\beta_i=0$ and $q_i=-\nu_{\Lambda_i}(\beta_i)$ otherwise--is either
  a zero stratum or a simple stratum, and this data must satisfy the
  following crucial condition:
\item the stratum $[\Lambda_i+\Lambda_j, q, r, \beta_i+\beta_j]$, with
  $q=\max\{q_i, q_j\}$, is non-equivalent to a zero stratum or a
  simple stratum, for $1\leq i, j\leq k$ and $i\neq j$.
\end{enumerate}
The decomposition $V=\oplus_{i=1}^kV_i$ is uniquely determined by the
element $\beta$, called the {\it underlying splitting} of the
semisimple stratum $[\Lambda, n, r, \beta]$.  A semisimple stratum
$[\Lambda, n, r, \beta]$ is called a {\it skew semisimple stratum} if
the decomposition $V=\oplus_{i=1}^kV_i$ is an orthogonal decomposition
with respect to the form $h$ on $V$, and $\sigma_h(\beta_i)=-\beta_i$,
for $1\leq i\leq k$. Observe that the algebra $F[\beta]$ is isomorphic
to the algebra
$$F[\beta_1]\oplus F[\beta_2]\oplus \cdots\oplus F[\beta_k].$$
We use the notation $\mathfrak{x}$ for a general skew semisimple
stratum $[\Lambda, n, 0, \beta]$. 

Let $C_\beta(A)$ be the centraliser of $F[\beta]$ in $\End_F(V)$.  The
group $G\cap C_\beta(A)$ is denoted by $G_\beta$.  Let $n$ be any
integer and let $\tilde{\mathfrak{b}}_n(\Lambda)$ and
$\mathfrak{b}_n(\Lambda)$ be the groups
$\tilde{\mathfrak{a}}_n(\Lambda)\cap C_\beta(A)$ and
${\mathfrak{a}}_n(\Lambda)\cap C_\beta(A)$ respectively. For any
non-negative integer $n$, let $\tilde{P}_n(\Lambda_\beta)$ and
${P}_n(\Lambda_\beta)$ be the groups $\tilde{P}_n(\Lambda)\cap
C_\beta(A)^\times$ and ${P}_n(\Lambda)\cap G_\beta$ respectively.
\subsection{}\label{simple_characters}
We recall the construction of a cuspidal representation starting from
a skew semisimple stratum $\mathfrak{x}=[\Lambda, n, 0,
\beta]$. Stevens, generalising the Bushnell--Kutzko's construction,
associates some special compact subgroups: $J^0(\Lambda, \beta)$ and
$H^0(\Lambda, \beta)$ of $G$.  Then a particular class of
representations of $J^0(\Lambda, \beta)$ are compactly induced to the
group $G$ to obtain cuspidal representations. We will not go into the
details of the construction of these compact subgroups here.  It will
suffice to briefly recall the general scheme of this
construction. However, we describe these compact subgroups, in more
detail, as required in the later part of the article.

Let $J^i(\Lambda, \beta)$ be the compact open subgroup
$J^0(\Lambda, \beta)\cap P_i(\Lambda)$, for any non-negative integer
$i$.  For any skew-semisimple stratum
$\mathfrak{x}=[\Lambda, n, 0, \beta]$, the construction of cuspidal
representations of $G$ begins with a specific set of characters of the
group $H^1(\Lambda, \beta)$ called {\it skew semisimple characters},
denoted by $\mathcal{C}(\Lambda, 0, \beta)$, (see \cite[Section
3.6]{semisimple_char_classical} and the set
$\mathcal{C}(\Lambda, 0, \beta)$ is denoted by
$\mathcal{C}_{-}(\Lambda, 0, \beta)$ there). The group
$P_{(n/2)+}(\Lambda)$ is contained in $H^1(\Lambda, \beta)$, and we
have $\res_{P_{(n/2)+}(\Lambda)}\theta=\psi_\beta$, for any
$\theta\in \mathcal{C}(\Lambda, 0, \beta)$. For any character
$\theta\in \mathcal{C}(\Lambda, 0, \beta)$, the map sending
$g_1, g_2\in J^1(\Lambda, \beta)$ to $\theta([g_1, g_2])$ induces a
perfect alternating pairing:
$$\kappa_\theta:\dfrac{J^1(\Lambda, \beta)}{H^1(\Lambda, \beta)}\times
\dfrac{J^1(\Lambda, \beta)}{H^1(\Lambda, \beta)}\rightarrow
\mathbb{C}^\times.$$ Using the theory of Heisenberg lifting, for any
character $\theta\in \mathcal{C}(\Lambda, 0, \beta)$, there exists a
unique representation $\eta_\theta$ of $J^1(\Lambda, \beta)$ such that
$\res_{H^1(\Lambda, \beta)}\eta_\theta$ is equal to a power of
$\theta$. There are a particular set of extensions of the representation
$\eta_\theta$ to the group $J^0(\Lambda, \beta)$ called
$\beta$-extensions; these representations are denoted by $\kappa$ (see
\cite[Section 4]{cusps_classical}).

The group $P_0(\Lambda_\beta)$ is contained in
$J^0(\Lambda, \beta)$. The inclusion of $P_0(\Lambda_\beta)$ in
$J^0(\Lambda, \beta)$ induces an isomorphism
$$P_0(\Lambda_\beta)/P_1(\Lambda_\beta)
\simeq J^0(\Lambda, \beta)/J^1(\Lambda, \beta).$$ The group
$P_0(\Lambda_\beta)/P_1(\Lambda_\beta)$ is the $k_{F_0}$-rational
points of a (non-necessarily connected) reductive group over
$k_{F_0}$. Let $\tau$ be a cuspidal representation of
$P_0(\Lambda_\beta)/P_1(\Lambda_\beta)$. If $P^0(\Lambda_\beta)$ is a
maximal parahoric subgroup of $C_\beta(A)\cap G$, then the induced
representation
\begin{equation}\label{compact_ind_cusps}
\ind_{J^0(\Lambda, \beta)}^G(\kappa\otimes \tau)
\end{equation}
is irreducible, and this construction exhausts all cuspidal
representations of $G$. The pair $(J^0(\Lambda, \beta),
\kappa\otimes\tau)$ is a Bushnell--Kutzko type for the Bernstein
component containing the representation \eqref{compact_ind_cusps}. Let
$\Pi_\mathfrak{x}$ be the set of cuspidal representations of $G$
containing a Bushnell--Kutzko type of the form $(J^0(\Lambda, \beta),
\kappa\otimes\tau)$, for some $\kappa$ and $\tau$ as above.
\subsection{}
For the convenience of the reader, we recall some frequently used
results from \cite{sp_4_genericity}. Let us begin with the following
lemma, which is useful in calculating the group
$H^1(\Lambda, \beta)\cap U$.
\begin{lemma}[Blondel--Stevens]\label{bl_st_lem}
  Let $[\Lambda, n, 0, \beta]$ be a semisimple stratum in $\End_F(V)$
  such that $C_\beta(A)$ does not contain any nilpotent elements. Let
  $N$ be a maximal unipotent subgroup of $G$. For $k\geq m\geq 1$, we
  have
$$P_m(\Lambda_\beta)P_k(\Lambda)\cap N
=P_k(\Lambda)\cap N.$$
\end{lemma}
We refer to \cite[Section 6.3, Lemma 6.5]{sp_4_genericity} for a proof
of the above lemma. The following result is proved in greater
generality by Blondel and Stevens (see \cite[Section 4, Corollary 4.2,
Theorem 4.3]{sp_4_genericity}); however, in the present context we
will use the following simple version.
\begin{proposition}[Blondel--Stevens]\label{bl_st_prop}
  Let $[\Lambda, n, 0, \beta]$ be a skew semisimple stratum such that
  $\mathfrak{X}_\beta(F_0)$ is non-empty. Assume that
  $J^0(\Lambda, \beta)/J^1(\Lambda, \beta)$ is anisotropic. Let
  $\sch{B}\in \mathfrak{X}_\beta(F_0)$, and let $\sch{U}$ be the unipotent
  radical of $\sch{B}$. If
  $$\res_{H^1(\Lambda, \beta)\cap U}\theta=\res_{H^1(\Lambda,
    \beta)\cap U}\psi_\beta,$$ for all
  $\theta\in \mathcal{C}(\Lambda, 0, \beta)$, then every cuspidal
  representation in the set $\Pi_\mathfrak{x}$ is generic.
\end{proposition}
\begin{proof}
  Since we use this result in a crucial way, we briefly sketch the
  proof (see \cite[Section 4, Corollary 4.2, Theorem
  4.3]{sp_4_genericity}). Let $\pi$ be a cuspidal representation in
  the set $\Pi_\mathfrak{x}$. Let $(J^0(\Lambda, \beta), \kappa)$ be a
  Bushnell--Kutzko type contained in the representation $\pi$. We have
  $$\pi\simeq \ind_{J^0(\Lambda, \beta)}^G\kappa.$$
  Let $\theta$ be a skew semisimple character contained in
  $\res_{H^1(\Lambda, \beta)}\kappa$.

  Let $\tilde{H}^1$ be the group
  $(J^0(\Lambda, \beta)\cap U)H^1(\Lambda, \beta)$ and $\Theta_\beta$
  be the character of $\tilde{H}^1$ defined by:
  $$\Theta_\beta(jh)=\psi_\beta(j)\theta(h),\ \text{for all}\ j\in
  J^0(\Lambda, \beta)\cap U,\ h\in H^1(\Lambda, \beta).$$ The group
  $\tilde{H}^1\cap J^1(\Lambda, \beta)$ is equal to
  $(J^1(\Lambda, \beta)\cap U)H^1(\Lambda, \beta)$, and it is a
  totally isotropic subspace for the pairing $\kappa_\theta$ on
  $J^1(\Lambda, \beta)/H^1(\Lambda, \beta)$. The representation
  $\eta_\theta$ is the induced representation from an extension of the
  character $\Theta_\beta$ to the inverse image in
  $J^1(\Lambda, \beta)$ of a maximal isotropic subspace of
  $J^1(\Lambda, \beta)/H^1(\Lambda, \beta)$. Hence,
  $\res_{J^1\cap U}\eta_\theta$ contains the character
  $\res_{J^1(\Lambda, \beta)\cap U}\psi_\beta$. Since the group
  $J^0(\Lambda, \beta)\cap U$ is equal to $J^1(\Lambda, \beta)\cap U$,
  using Mackey-decomposition, we get that
  $\ho_{U}(\pi, \psi_\beta)\neq 0$.
\end{proof}
It is convenient to partition the set of skew-semisimple strata in
$\End_F(V)$ into four disjoint classes. A skew semisimple stratum
$\mathfrak{x}=[\Lambda,n, 0, \beta]$ is of type {\bf (A)} if
$F[\beta]$ is a field. A skew semisimple stratum $\mathfrak{x}$ is of
type {\bf (B)} if $F[\beta]$ is a direct sum of two fields with one of
them a quadratic extension of $F$. A skew semisimple stratum
$\mathfrak{x}$ is of type {\bf (C)} if the algebra $F[\beta]$ is a
direct sum of two copies of $F$, and finally a skew semisimple stratum
$\mathfrak{x}$ is of type {\bf (D)} if the algebra $F[\beta]$ is a
direct sum of three copies of $F$.
\section{The simple case.}\label{the-simple_case}
\subsection{}
A skew semisimple stratum $[\Lambda, n, 0, \beta]$ is of type {\bf
  (A)} if the algebra $F[\beta]$ is a field. When the characteristic
of $F_0$ is zero, it was shown by Blasco that the cardinality of the
$L$-packet, containing a cuspidal representation in the set
$\Pi_\mathfrak{x}$, is equal to one (see
\cite{blasco_singeton_pac}). When the characteristic of $F_0$ is zero,
a tempered $L$-packet is known to contain a generic member (see
\cite{gen_unit_3_var}), and hence every cuspidal representation in
the set $\Pi_\mathfrak{x}$ is generic. In this section, for any
non-Archimedean local field $F_0$ with odd residue characteristic, we
directly prove that every cuspidal representation in the set
$\Pi_\mathfrak{x}$ is generic.
\begin{lemma}\label{type_A_existence_flags}
  Let $[\Lambda, n, 0, \beta]$ be a skew simple stratum of type {\bf
    (A)}, then the set $\mathfrak{X}_\beta(F_0)$ is non-empty.
\end{lemma}
\begin{proof}
  The involution $\sigma_h$ stabilises the field $F[\beta]$, and let
  $D$ be the fixed field of $\sigma_h$ in $F[\beta]$. We have the
  following diagram of field extensions:
\begin{displaymath}
  \begin{tikzpicture}[node distance = 1cm, auto]
      \node (Q) {$F_0$};
      \node (E) [above of=Q, left of=Q] {$F$};
      \node (F) [above of=Q, right of=Q] {$D=F_0[\delta\beta]$};
      \node (K) [above of=E, right of=E] {$F[\beta]$};
      \draw[-] (Q) to node {} (E);
      \draw[-] (Q) to node {} (F);
      \draw[-] (E) to node {} (K);
      \draw[-] (F) to node {} (K);
      \end{tikzpicture}
    \end{displaymath}

    Let $\lambda$ be a $\sigma_h$-equivariant non-zero $F$-linear form
    on $F[\beta]$.  There exists a unique hermitian form
    $h_1:V\times V\rightarrow F[\beta]$ such that
  $$h_1(xv, yw)=x\sigma_h(y)\sigma_h(h_1(w, v)),\ \text{for all}\ x,
  y\in F[\beta]\ \text{and}\ v, w\in V,$$
$$h(v, w)=\lambda((h_1(v, w)), \ \text{for all}\ v, w\in V.$$
The set $\mathfrak{X}_\beta(F_0)$ is non-empty if and only there exists a
non-zero vector, $v\in V$, such that
$$h(v, v)=0\  \text{and}\  h(v, \beta
v)=0.$$ Observe that
$h(v, \beta v)=\lambda(h_1(v, \beta v))=\lambda(\beta h_1(v,
v))$. Since $\beta$ does not stabilise a proper non-zero subspace of
$F[\beta]$, the $F$-linear forms $\lambda$ and $\lambda\circ\beta$ are
linearly independent. Let $W$ be the $F$-vector space
$\ker(\lambda)\cap \ker(\lambda\circ\beta)$, we have
$\dim_FW=1$. Since $\sigma_h(\beta)=-\beta$ and $\lambda$ is
$\sigma_h$-equivariant, the $F$-vector space $W$ is stable under the
action of $\sigma_h$. Let $W_0$ be the $F_0$-vector space $W\cap D$,
and note that $W_0\otimes_{F_0} F=W$.

The form $h_1(x, y)$ is equal to $xa\sigma_h(y)$, for some $a\in
D$. Let $N_a$ be the set $\{xa\sigma_h(x): x\in F[\beta]^\times\}$.
Assume that the inclusion of $F_0^\times$ in $D^\times$ induces a
surjection of $F_0^\times$ onto the quotient
$D^\times/({\rm Nr}_{F[\beta]/D}(F[\beta]^\times)$. Let $w$ be a
non-zero vector in $W_0$. If $w\not\in N_a$, then there exists some
element $x\in F_0^\times$ such that $xw_1\in N_a$. Hence, the set
$N_a\cap W_0$ is non-empty.  Since $h(w, w)=h(w, \beta w)=0$, for all
$w\in N_a\cap W_0$, the set $\mathfrak{X}_\beta(F_0)$ is non-empty.

Now, we will prove that the inclusion of $F_0^\times$ in $D^\times$
induces a surjection of $F_0^\times$ onto the quotient
$$D^\times/({\rm Nr}_{F[\beta]/D}(F[\beta]^\times).$$  
    Recall the notations $\varpi$ and $\varpi_0$ for a choice
    uniformizers of $F$ and $F_0$, respectively, from Paragraph
    \ref{local_fields}. Fix a valuation
    $\nu:F[\beta]^\times \rightarrow 1/e[F[\beta]:F]\mathbb{Z}$ such
    that $\nu(\varpi)=1$.  Assume that
    $\varpi_0= {\rm Nr}_{F[\beta]/D}(x)$, for some $x\in
    F[\beta]$. Then, we have $\nu(\varpi)=2\nu(x)$.  If $F/F_0$ is
    unramified, then the equality $\nu(\varpi)=2\nu(x)$ is impossible
    and therefore, we get that
    $\varpi_0\not\in {\rm Nr}_{F[\beta]/D}(F[\beta]^\times)$. Consider
    the case where $F$ is a ramified extension of $F_0$. Let $x$ be an
    element of $\mathfrak{o}_{F_0}^\times$ such that
    $\bar{x}\in k_{F_0}$ is not a square in $k_{F_0}$. The element $x$
    does not belong to ${\rm Nr}_{F[\beta]/D}(F[\beta]^\times)$; since
    the involution $\sigma_h$ must act trivially on the residue field
    of $k_{F[\beta]}$ (note that $[k_{F[\beta]}:k_{F_0}]|3$).
\end{proof}
\begin{proposition}\label{type_A_end_prop}
  Let $\mathfrak{x}=[\Lambda, n, 0, \beta]$ be any skew simple stratum
  of type {\bf (A)}. Then every cuspidal representation in the set
  $\Pi_\mathfrak{x}$ is generic.
\end{proposition}
\begin{proof}
  Let $\pi$ be a cuspidal representation in the set
  $\Pi_\mathfrak{x}$.  Since $\dim_FV=3$, we may twist the
  representation $\pi$ by a character, if necessary, and assume that
  $\mathfrak{x}$ is minimal. Let $(J^0(\Lambda, \beta), \kappa)$ be a
  Bushnell--Kutzko type contained in $\pi$. Then the representation
  $\pi$ is isomorphic to $\ind_{J^0(\Lambda, \beta)}^G\kappa$. Here,
  the group $J^0(\Lambda, \beta)$ is equal to
  $\mathfrak{o}^\times_{F[\beta]}P_{(n/2)+}(\Lambda)$, and $\kappa$ is
  a $\beta$-extension of the Heisenberg lift $\eta_\theta$ of a skew
  semisimple character $\theta$ of $H^1(\Lambda, \beta)$. Now, Lemma
  \ref{type_A_existence_flags} implies that the set
  $\mathfrak{X}_\beta(F_0)$ is non-empty. Let $\sch{U}$ be the
  unipotent radical of a Borel subgroup in the set
  $\mathfrak{X}_\beta(F_0)$.  Using Lemma \ref{bl_st_lem}, we get that
  the group $J^0(\Lambda, \beta)\cap U$ is equal to
  $P_{(n/2)+}(\Lambda)\cap U$. Hence, we have
  $$\res_{H^1(\Lambda, \beta)\cap U}\theta
  =\res_{H^1(\Lambda, \beta)\cap U}\psi_\beta.$$ Now using Proposition
  \ref{bl_st_prop}, we get that the representation $\pi$ is generic.
\end{proof}
\section{The non simple type (B) strata}\label{type_B_sec}
A skew semisimple stratum $\mathfrak{x}=[\Lambda, n, 0, \beta]$ is of
type {\bf (B)} if the underlying splitting of $\mathfrak{x}$ is of the
form $V=V_1\perp V_2$ with $\dim_{F}(V_2)=2$ and the algebra
$F[\beta_2]$ is a quadratic extension of $F$. Recall that $\beta_i$ is
equal to ${\bf 1}_{V_i}\beta {\bf 1}_{V_i}$, for $i\in \{1,2\}$. From
the definition of a skew semisimple stratum, we have
$\beta=\beta_1+\beta_2$.  If $F/F_0$ is unramified, from the equality
$\sigma_h(\beta_2)=-\beta_2$, we get that the extension $F[\beta_2]/F$
is ramified. Similarly, if $F/F_0$ is ramified, then the extension
$F[\beta_2]/F$ is unramified. We recall the notation that
$q_i=-\nu_{\Lambda_i}(\beta_i)$, for $i\in\{1,2\}$. Genericity of a
cuspidal representation, in the set $\Pi_\mathfrak{x}$, depends only
on the isomorphism class of the hermitian space $(V_2, h)$ and the
integers $q_1$ and $q_2$. Since $\dim_F(V_2)=2$, after twisting by a
character, if necessary, we may assume that $\beta_2$ is
minimal.
\subsection{Lattice sequences}
Since $G_\beta$ is anisotropic, the lattice sequence $\Lambda$ is
uniquely determined by $\beta$. First, we need to fix a Witt-basis of
$V$ which provides a splitting for $\Lambda$.
\subsubsection{The Unramified case}\label{type_B_unram_lat_seq_info}
Let $F$ be an unramified quadratic extension of $F_0$, and consider
the case where $(V_2, h)$ is isotropic. Since $F/F_0$ is unramified,
the extension $F[\beta_2]/F$ is ramified. In this case, we have
$e(\Lambda)=4$. There exists a Witt-basis, $(e_1, e_{-1})$, for the
space $V_2$, and a unit vector $e_0\in V_1$ (i.e., $h(e_0, e_0)=1$)
such that the lattice sequence $\Lambda$ is given by:
\begin{align*}
  \Lambda(-1)=\mathfrak{o}_Fe_1\oplus \mathfrak{o}_Fe_0\oplus
  \mathfrak{o}_Fe_{-1}, &\ \Lambda(0)=\mathfrak{o}_Fe_1\oplus
                          \mathfrak{o}_Fe_0\oplus\mathfrak{p}_Fe_{-1},\\
  \Lambda(1)=\mathfrak{o}_Fe_1\oplus \mathfrak{p}_Fe_0\oplus
  \mathfrak{p}_Fe_{-1},\ &\ \Lambda(2)=\mathfrak{p}_Fe_1\oplus
                           \mathfrak{p}_Fe_0\oplus \mathfrak{p}_Fe_{-1}.
\end{align*}
The group $P(\Lambda)$ is an Iwahori subgroup of $G$. The filtration
$\{P_m(\Lambda): m>0\}$ is sometimes called a non-standard filtration
on the Iwahori subgroup $P(\Lambda)$ (see \cite[Page
612]{murnaghan_char_emp_u_3}).

Now, consider the case where $F/F_0$ is unramified and $(V_2, h)$ is
anisotropic. There exists an orthogonal basis $(v_2, v_3)$ of $V_2$
such that $h(v_2, v_2)=1$ and $h(v_3, v_3)=\varpi_0$. Let $v_1$ be any
vector in $V_1$ such that $h(v_1, v_1)=\varpi_0$. The period of the
lattice sequence, $e(\Lambda)$, is equal to $2$, and
\begin{equation}\label{type_B_unram_aniso_lat_seq}
  \Lambda(0)=\mathfrak{o}_Fv_1\oplus \mathfrak{o}_Fv_2\oplus
  \mathfrak{o}_Fv_3\ \text{and}\ \Lambda(1)=
  \mathfrak{o}_Fv_1\oplus \mathfrak{p}_Fv_2\oplus
  \mathfrak{o}_Fv_3.
\end{equation}
\begin{lemma}\label{type_B_seq_unram_basis}
  Let $(V_2, h)$ be anisotropic, and let $\Lambda$ be the lattice
  sequence in \eqref{type_B_unram_aniso_lat_seq}. There exists a
  Witt-basis $(e_1, e_{-1})$ for $(\langle v_1, v_3\rangle, h)$ such
  that
  $$\mathfrak{o}_Fv_1\oplus
  \mathfrak{o}_Fv_3=\mathfrak{p}_Fe_1\oplus \mathfrak{o}_Fe_{-1}.$$
\end{lemma}
\begin{proof}
  Let $\epsilon\in F$ be an element with
  $\epsilon\sigma(\epsilon)=-1$; such an element exists because
  $F/F_0$ is unramified. The vectors $e_{1}=\epsilon v_1/2+v_3/2$
  and $\varpi_0e_{-1}=-\epsilon v_1+v_3$ are isotropic vectors, and
  $h(e_1, e_{-1})=1$.  This implies that $(e_1, e_{-1})$ is a
  Witt-basis for $\langle v_1, v_3\rangle$, and since
  $\nu_F(\epsilon)=0$, the tuple $(e_1, \varpi_0e_{-1})$ is a
  $\mathfrak{o}_F$ basis for the lattice
  $\mathfrak{o}_Fv_1\oplus \mathfrak{o}_Fv_3$.
\end{proof}
\noindent
When $(V_2, h)$ is anisotropic as above, we set $e_0$ to be the vector
$v_2$.  The Witt-basis $(e_1, e_0, e_{-1})$, where $(e_1, e_{-1})$ is
a Witt-basis for $(\langle v_1, v_3\rangle, h)$ as in Lemma
\ref{type_B_seq_unram_basis} provides a splitting for $\Lambda$.  In
the basis $(e_1, e_0, e_{-1})$, the lattice sequence $\Lambda$, with
$e(\Lambda)=2$, is given by
$$
\Lambda(0)=
\mathfrak{o}_Fe_1\oplus \mathfrak{o}_Fe_0\oplus \mathfrak{p}_Fe_{-1}\
\text{and}\ \Lambda(1)= \mathfrak{o}_Fe_1\oplus
\mathfrak{p}_Fe_0\oplus \mathfrak{p}_Fe_{-1}.
$$
\subsubsection{Ramified case}\label{type_B_ram_lat_seq_info}
Let $F/F_0$ be a ramified quadratic extension. Then the extension
$F[\beta_2]/F$ is an unramified quadratic extension. First consider
the case where $(V_2, h)$ is isotropic. The lattice sequence $\Lambda$
has period $2$. There exists a Witt-basis, $(e_1, e_0, e_{-1})$, for
the hermitian space $(V, h)$ with $e_1, e_{-1}\in V_2$ such that
$$\Lambda(0)=\mathfrak{o}_Fe_1\oplus
\mathfrak{o}_Fe_0\oplus \mathfrak{p}_Fe_{-1}\ \text{and}\
\Lambda(1)=\mathfrak{o}_Fe_1\oplus \mathfrak{p}_Fe_0 \oplus
\mathfrak{p}_Fe_{-1}.$$ Now, consider the case where $(V_2, h)$ is
anisotropic. There exists an orthogonal basis $(v_2, v_3)$ of $V_2$
and a non-zero vector $v_1\in V_1$ such that:
$h(v_i, v_i)=\lambda_i\in \mathfrak{o}_F^\times$, for $1\leq i\leq 3$,
and the hermitian space $(\langle v_1, v_3\rangle, h)$ is isotropic.
The period of the lattice sequence $\Lambda$ is equal to $2$, and we
have:
$$\Lambda_2(-1)=\Lambda_2(0)=
\mathfrak{o}_Fv_1\oplus \mathfrak{o}_Fv_2\oplus \mathfrak{o}_Fv_3.$$
\begin{lemma}\label{type_B_ram_seq_split}
  Let $(V_2, h)$ be anisotropic. Then there exists a Witt-basis
  $(e_1, e_{-1})$ for the hermitian space
  $(\langle v_1, v_3\rangle, h)$ such that
$$\mathfrak{o}_Fv_1\oplus \mathfrak{o}_Fv_3=\mathfrak{o}_Fe_1\oplus 
\mathfrak{o}_Fe_{-1}.$$
\end{lemma}
\begin{proof}
  We fix an $\epsilon\in F$ such that
  $\epsilon\sigma(\epsilon)=-\lambda_3\lambda_1^{-1}$. The vectors
  $e_1=\epsilon v_1/2+v_3/2$ and $e_{-1}=(-\epsilon
  v_1+v_3)\lambda_3^{-1}$ are isotropic and $h(e_1,
  e_{-1})=1$. Moreover, the vectors $e_1$ and $e_{-1}$ are a basis for
  the $\mathfrak{o}_F$-lattice $\mathfrak{o}_Fv_1\oplus
  \mathfrak{o}_Fv_3$.
\end{proof}
\noindent
Let $(e_1, e_{-1})$ be a
Witt-basis for the hermitian space $(\langle v_1, v_3\rangle, h)$ as
in Lemma \ref{type_B_ram_seq_split}. Let $e_0$ be the vector $v_2$.
  In the basis $(e_1, e_0, e_{-1})$, the period $2$
lattice sequence $\Lambda$ is given by
$$\Lambda(-1)=\Lambda(0)=
\mathfrak{o}_Fe_1\oplus \mathfrak{o}_Fe_0\oplus
\mathfrak{o}_Fe_{-1}.$$
\subsection{}
In each of the above cases, we fixed a Witt-basis $(e_1, e_0, e_{-1})$
which gives a splitting for the lattice sequence $\Lambda$. Let
$\sch{B}$ be the Borel subgroup of $\sch{G}$ such that $B$ fixes the
space $\langle e_1 \rangle$. Let $\sch{T}$ be the maximal $F_0$-split
torus of $\sch{G}$ such that $T$ stabilises the decomposition
$V=\langle e_1 \rangle\oplus \langle e_0 \rangle\oplus \langle e_{-1}
\rangle$.
Let $\bar{\sch{B}}$ be the opposite Borel subgroup of $\sch{B}$ with
respect to $\sch{T}$.  Let $\sch{U}$ (resp. $\bar{\sch{U}}$) be the
unipotent radical of $\sch{B}$ (resp. $\bar{\sch{B}}$). We also recall
the notations $u(c, d)$ and $\bar{u}(c,d)$ for elements in $U$ and
$\bar{U}$ respectively. In the basis $(e_1, e_0, e_{-1})$, let
$\mathcal{I}$ be the Iwahori subgroup:
$$\mathcal{I}=\begin{pmatrix}
  \mathfrak{o}_F&\mathfrak{o}_F&\mathfrak{o}_F\\
  \mathfrak{p}_F&\mathfrak{o}_F&\mathfrak{o}_F\\
  \mathfrak{p}_F&\mathfrak{p}_F&\mathfrak{o}_F.
\end{pmatrix}\cap G.$$
From the Iwasawa decomposition, we get that
\begin{equation}
  G=\coprod_{w\in W_G}\mathcal{I}wB.
\end{equation}
\subsection{}\label{type_B_valuations}
The integers $q_1$ and $q_2$ have the following constraints: If
$F/F_0$ is unramified and $(V_2, h)$ is isotropic, then $q_2=4m_2+2$
and $q_1=4m_1$ for some $m_1, m_2\in \mathbb{Z}$. If $F/F_0$ is
unramified and $(V_2, h)$ is anisotropic, then $q_2=2m_2+1$ and
$q_1=2m_1$, for some $m_1, m_2\in \mathbb{Z}$. Finally consider the
case where $F/F_0$ is ramified. Since the image of the element
$$y_{\beta_2}=\varpi ^{q_2/g}\beta_2^{e(\Lambda_2)/g}=\varpi ^{q_2/2}\beta_2$$
(here, $g=(q_2,e(\Lambda_2))=2$) in $k_{F[\beta_2]}$ must generate the
degree $2$ field extension $k_{F[\beta_2]}$ over $k_F$, we get that
$q_2=4m_2$ and $q_1=4m_1+2$, for some $m_1, m_2\in \mathbb{Z}$. Hence,
in all the above cases $q_1\neq q_2$.
\subsection{}
We will need the structure of the compact subgroups $J^0(\Lambda, \beta)$
and $H^0(\Lambda, \beta)$ associated to $\mathfrak{x}$. First consider
the case $q_2<q_1$, and in this case the constant $k_0(\Lambda, \beta)$,
defined in \cite[equation 3.6]{semisimple_char_classical}, is equal to
$q_2$.  The stratum $[\Lambda, n, q_2, \beta_1]$ is a skew semisimple
stratum equivalent to the stratum $[\Lambda, n, q_2, \beta]$. We then
have:
\begin{align*}
  J^0(\Lambda, \beta)=&P_0(\Lambda_\beta)
                        P_{q_2/2}(\Lambda_{\beta_1})P_{n/2}(\Lambda).\\
  H^1(\Lambda, \beta)=&P_1(\Lambda_\beta)
                        P_{(q_2/2)+}(\Lambda_{\beta_1})P_{(n/2)+}(\Lambda).
\end{align*}
If $q_2>q_1$, then we have
\begin{align*}
  J^0(\Lambda,
  \beta)=&P_0(\Lambda_\beta)P_{n/2}(\Lambda).\\
  H^1(\Lambda, \beta)=&P_1(\Lambda_\beta)P_{(n/2)+}(\Lambda).
\end{align*}
\subsection{}\label{criterion_non-empty}
When $\mathfrak{x}$ is a skew semisimple stratum of type {\bf (B)}, we
prove a necessary and sufficient condition for the non-emptiness of
the set $\mathfrak{X}_\beta(F_0)$. The field $F[\beta_2]$ is stable
under the action of $\sigma_h$ and let $D$ be the fixed field of the
automorphism $\sigma_h$. Let $\lambda$ be the $F$-linear
$\sigma_h$-equivariant linear functional
$\lambda:F[\beta_2]\rightarrow F$ such that $\lambda(\beta_2)=\beta_1$
and $\lambda(1)=1$.  There exists a unique hermitian form
$h':V_2\times V_2\rightarrow F[\beta_2]$ such that
\begin{equation}\label{type_B_X_beta_eqs}
h'(xv, yw)=x\sigma_h(y)\sigma_h(h'(w,v)),
\ \text{for all}\ v, w\in V\ \text{and}\ x, y\in F[\beta_2],
\end{equation}
and
$$h(v, w)=\lambda(h'(v, w)),\ \text{for all}\ v, w\in V_2.$$
Since $F_0[\delta]$ is a quadratic extension of $F_0$, the kernel of
$\lambda$ is equal to
$$(\beta_2-\beta_1)F_0\oplus \delta(\beta_2-\beta_1)F_0.$$
The set $\mathfrak{X}_\beta(F_0)$ is non-empty if and only if there
exists $v_1\in V_1$, $v_2\in V_2$ with $v_1+v_2\neq 0$ such that
\begin{align}\label{type_B_crite_eqn_1}
h(v_1, v_1)+\lambda(h'(v_2, v_2))&=0\\\label{type_B_crite_eqn_2}
\beta_1h(v_1, v_1)+\lambda(\beta_2h'(v_2, v_2))&=0.
\end{align}
For any two vectors $v_1$ and $v_2$ as above we must have $v_1\neq 0$
and $v_2\neq 0$. Note that $(\beta_1-\beta_2)h'(v_2, v_2)$ is
contained in the kernel of $\lambda$, and this implies that
$h'(v_2, v_2)\in F_0^\times$. If $d_1$ and $d_2$ are the determinants
of $(V_2, h')$ and $(V_1, h)$, then we have
\begin{equation}\label{type_B_X_beta_criterion}
-d_1d_2^{-1}\in {\rm
  Nr}_{F[\beta_2]/D}(F[\beta_2]^\times).
\end{equation} 
Conversely, if the condition \eqref{type_B_X_beta_criterion} holds,
then we can always find a simultaneous solution to the equations
\eqref{type_B_crite_eqn_1} and \eqref{type_B_crite_eqn_2}, and hence,
the set $\mathfrak{X}_\beta(F_0)$ is non-empty.
\subsection
{The case where \texorpdfstring{$(V_2, h)$}{} is isotropic}
In this part, we treat the case where $(V_2, h)$ is isotropic, and let
us begin with the case where $q_2>q_1$.
\begin{lemma}\label{type_B_iso_q_1<q_2}
  Let $F/F_0$ be an unramified extension and let $\mathfrak{x}$ be a
  stratum of type {\bf (B)}. If $(V_2, h)$ is isotropic and $q_2>q_1$,
  then every cuspidal representation in the set $\Pi_\mathfrak{x}$ is
  non-generic. Moreover, we have $\mathfrak{X}_\beta(F_0)=\emptyset$.
\end{lemma}
\begin{proof}
  Let $\pi$ be a representation in the set $\Pi_\mathfrak{x}$. There
  exists a $\beta$-extension $\kappa$ of $J^0(\Lambda, \beta)$ such
  that $\pi=\ind_{J^0(\Lambda, \beta)}^G\kappa$. If $\pi$ is genetic,
  then there exists a non-trivial character $\Psi$ of $U$, $p\in
  \mathcal{I}$, and $w\in W_G$ such that
\begin{equation}\label{type_B_unram_iso_q_2>q_1_eq_1}
\ho_{J^0(\Lambda, \beta)^p\cap U^w}(\kappa^p, \Psi^w)\neq 0.
\end{equation}
Let $\overline{U^w}$ be the unipotent radical of the opposite Borel
subgroup, $\overline{B^w}$, of $B^w$ containing $T$.  Using the
Iwahori decomposition of $\mathcal{I}$, with respect to
$(\overline{B^w}, T)$, we have $p=p^+u^-$ with
$p^+\in B^w\cap \mathcal{I}$ and
$u^-\in \overline{U^w}\cap \mathcal{I}$. Now the equation
\eqref{type_B_unram_iso_q_2>q_1_eq_1} implies that
$$\ho_{J^0(\Lambda, \beta)^{u^-}\cap U^w}(\kappa^{u^-}, \Psi')\neq 0,
$$
for some character $\Psi'$ of $U^w$. The group $P_{(q_2/2)+}(\Lambda)$
is normalised by the element $u^-=\dot{u}(x,y)$, and hence,
$P_{(q_2/2)+}(\Lambda)^{u^-}\cap U^w$ is equal to
$P_{(q_2/2)+}(\Lambda)\cap U^w$.

We set $q_2=8k+2r$, for some integer $k$ and $r\in\{1,3\}$. For the
following calculations it is convenient to refer to the appendix
\ref{unram_fil_3} for an explicit description of the filtration
$\{a_m(\Lambda): m\in \mathbb{Z}\}$.  We have $(q_2/2)+=4k+r+1$ and
the intersection $P_{(q_2/2)+}(\Lambda)\cap U^w$ is given by:
\begin{eqnarray}\label{type_B_unram_lrt_wht}
P_{(q_2/2)+}(\Lambda)\cap U^w=
\begin{cases}
  U(k+\lfloor r/2\rfloor),
  &\ \text{if}\  w=\id,\\
  U^w(k+1+\lfloor r/2\rfloor), &\ \text{if}\ w\neq\id.
\end{cases}
\end{eqnarray}
We define $e_w$ and $e_{-w}$ by setting $e_w=we_1$ and
$e_{-w}=we_{-1}$, $w\in W_G$. Then, we have
\begin{equation}\label{type_B_q_2>q_1_iso_epan_1}
  h(u^-e_{w}, \beta u^-e_{w})=
  \beta_1x\bar{x}+h(e_{w}, \beta_2 e_{w})+(y+\bar{y}) h(e_w,
  \beta_2e_{-w})+y\bar{y}h(e_{-w}, \beta_2 e_{-w}).\end{equation}
Now, the valuation of the term $h(e_{w}, \beta_2
e_{w})$ is strictly less than the valuation of all other terms in the
right-hand side of 
\eqref{type_B_q_2>q_1_iso_epan_1}.
Hence, using Iwasawa decomposition, we get that
$h(ge_1, \beta ge_1))\neq 0$, for all $g\in G$.  This shows that the
set $\mathfrak{X}_\beta(F_0)$ is empty. We also get the valuation of
$\delta h(u^-e_w, \beta u^-e_w)$ is equal to
\begin{eqnarray}\label{type_B_unram_iso_height}
  \nu_F(\delta h(u^-e_w, \beta u^-e_w))=\nu_F(h(e_{w}, \beta_2
  e_{w}))=
  \begin{cases}
    -(2k+\lfloor r/2\rfloor)\ &\text{if}\ w=\id,\\
    -(2k+\lfloor r/2\rfloor+1)\ &\text{if}\ w\neq \id.
    \end{cases}
\end{eqnarray}
From the equations \eqref{type_B_unram_lrt_wht}
and \eqref{type_B_unram_iso_height}, we get that 
$$\nu(h(u^-e_w, u^-e_w))\leq -s,$$
where $s$ is given by the equality
$$U^w_{\rm der}(s)=P_{(q_2/2)+}(\Lambda)\cap U^w.$$
Now, Lemma \ref{basic_inequality} implies that the character
$\psi_\beta^{u^-}$ is non-trivial on the group
$P_{(q_2/2)+}(\Lambda)\cap U^w$ and we get a contradiction to
\eqref{type_B_unram_iso_q_2>q_1_eq_1}. 
\end{proof}
\begin{lemma}\label{type_B_ram_iso_q_1>q_2_non_gen}
  Let $F/F_0$ be a ramified extension and let $\mathfrak{x}$ be a
  stratum of type {\bf (B)} such that $(V_2, h)$ is isotropic. If
  $q_2>q_1$, then the set $\mathfrak{X}_\beta(F_0)$ is empty, and
  every cuspidal representation in the set $\Pi_\mathfrak{x}$ is
  non-generic.
\end{lemma}
\begin{proof}
  First, we note that the group $P^0(\Lambda)$ is a special parahoric
  subgroup of $G$.  Let $\pi=\ind_{J^0(\Lambda, \beta)}^G\kappa$ be a
  generic representation in the set $\Pi_\mathfrak{x}$.  Using the
  Iwasawa decomposition $G=P(\Lambda)TU$, and the Mackey decomposition
  for the representation $\res_U\pi$, we get that
\begin{equation}\label{type_B_ram_iso_q_2>q_1_eq_1}
\ho_{J^0(\Lambda, \beta)^g\cap U}(\kappa^g, \Psi)\neq 0,
\end{equation}
for some $g\in P(\Lambda)$, and a non-trivial character
$\Psi$ of $U$.  We refer to the appendix \ref{unram_fil_2} for an
explicit description of the filtration
$\{a_m(\Lambda): m\in \mathbb{Z}\}$.

We set $q_2=4k_2$ and $q_1=4k_1+2$, for some integers $k_1$ and
$k_2$. The group $P(\Lambda)$ normalises $P_{(q_2/2)+}(\Lambda)$ and
hence, $P_{(q_2/2)+}(\Lambda)^g\cap U_{\rm der}$ is equal to
$P_{(q_2/2)+}(\Lambda)\cap U_{\rm der}$. We have
\begin{equation}\label{type_B_ram_lrt_wht}
  P_{(q_2/2)+}(\Lambda)\cap U_{\rm der}=
  U_{\rm der}(\lceil(k_2-1)/2\rceil).
\end{equation}
Since, $\nu_\Lambda(e_1)=1$, we have $ge_1=ae_1+\varpi be_0\oplus
\varpi ce_{-1}$, for some $a, b, c\in \mathfrak{o}_F$. We now try to
estimate the valuation of $h(ge_1, \beta ge_1)$. Observe that $h(ge_1,
\beta ge_1)$ is equal to
\begin{align*}
-\varpi^2\beta_1b\sigma(b)+ a\sigma(a)h(e_1, \beta_2e_1)
  -\varpi
  a\sigma(c)h(e_1, \beta_2e_{-1})+\\
  \varpi c\sigma(a)h(e_{-1}, \beta_2e_{1})-\varpi^2c\sigma(c)h(e_{-1},
  \beta_2e_{-1}).
\end{align*}
Recall that $F[\beta_2]$ is the unramified quadratic extension of $F$
and the element $\tilde{\beta}_2=\varpi^{q_2/2}\beta_2$ belongs to
$a_0(\Lambda_2)$. Since $\nu_{\Lambda}(e_1)=1$, the constants $a$ and
$c$ both cannot belong to $\mathfrak{p}_F$. From the observation that
$\nu_F({\rm Nr}_{F[\beta_2]/F}(\beta_2))=0$, we have
\begin{align*}
  a\sigma(a)h(e_1, 1/\varpi\tilde{\beta}_2e_1) - a\sigma(c)h(e_1,
  \tilde{\beta}_2e_{-1})+
  c\sigma(a)h(e_{-1}, \tilde{\beta}_2e_{1})\\
  +c\sigma(c)h(e_{-1}, \varpi \tilde{\beta_2}e_{-1})\not\equiv 0\
  (\text{mod}\ \mathfrak{p}_F).\end{align*}
We note that
$$\nu_{F/F_0}(\varpi^2\beta_1b\sigma(b))\geq -k_1+1/2.$$
Since $q_2>q_1$, we have $-k_2<-k_1-1/2$, we deduce that
$h(ge_1, \beta ge_1)\neq 0$ and
$\nu_{F/F_0}(\delta h(ge_1, \beta ge_1))$ is equal to $-k_2+1$.
From the equation \eqref{type_B_ram_lrt_wht}, we get that
$$\nu_{F/F_0}(\delta h(ge_1, \beta ge_1))\leq -s,$$
where $P_{(q_2/2)+}(\Lambda)\cap U_{\rm der}= U_{\rm der}(s)$. Now,
using Lemma \ref{basic_inequality}, we arrive at a contradiction to
\eqref{type_B_ram_iso_q_2>q_1_eq_1}. Hence, any representation $\pi$
in the set $\Pi_\mathfrak{x}$ is non-generic. Finally, using the
Iwasawa decomposition $G=P(\Lambda)B$, we get that
$h(ge_1, \beta ge_1)\neq 0$, for $g\in G$. Hence, the set
$\mathfrak{X}_\beta(F_0)$ is empty.
\end{proof}
\begin{lemma}\label{type_B_iso_q_1>q_2_whit_data}
  Let $F/F_0$ be an unramified extension and let $\mathfrak{x}$ be a
  stratum of type {\bf (B)} such that $(V_2, h)$ is isotropic. If
  $q_1>q_2$, then the set $\mathfrak{X}_\beta(F_0)$ is non-empty.
\end{lemma}
\begin{proof}
  We continue using the notations introduced in \ref{criterion_non-empty}.
  Since $(V_2, h)$ is isotropic, we get that $\ker(\lambda)$ has
  non-trivial intersection with the set $\{h'(v, v): v\in V_2, v\neq
  0\}$. Since $F[\beta_2]$ is a ramified extension of $F$ and $F/F_0$
  is an unramified extension, we get that $F_0^\times$ is contained in
  ${\rm Nr}_{F[\beta_2]/D}(F[\beta_2]^\times)$.  This implies that the
  determinant of the form $(V_2, h')$ is the class in $D^\times/{\rm
    Nr}_{F[\beta_2]/D}(F[\beta_2]^\times)$ containing
  $\delta(\beta_2-\beta_1)$. Since, $q_1>q_2$ we have
$$\nu_{F[\beta_2]}(\delta(\beta_2-\beta_1))=\nu_{F[\beta_2]}(\beta_1)
\in 2\mathbb{Z}.$$ From the observation that $d_1$, the determinant of
$(V_1, h)$, belongs to $F_0$, we see that $$-d_1d_2^{-1}\in {\rm
  Nr}_{F[\beta_2]/D}(F[\beta_2]^\times).$$ Now, using the criterion
\eqref{type_B_X_beta_criterion}, we get that the set
$\mathfrak{X}_\beta(F_0)$ is non-empty.
\end{proof}
\begin{lemma}\label{type_B_ram_iso_q_1>q_2_whit_mod}
  Let $F/F_0$ be a ramified extension, and let $\mathfrak{x}$ be a
  stratum of type {\bf (B)} such that $(V_2, h)$ is isotropic. If
  $q_1\geq q_2$, then the set $\mathfrak{X}_\beta(F_0)$ is non-empty.
\end{lemma}
\begin{proof}
  We continue using notations introduced in \ref{criterion_non-empty}.
  Since the space $(V_2, h)$ is isotropic, the set
  $\{h'(v, v): v\in V_2, v\neq 0\}$ has a nontrivial intersection with
  $\ker(\lambda)$. Hence, the determinant class of $(V, h')$ is equal
  to the coset in
  $D^\times/{\rm Nr}_{F[\beta_2]/D}(F[\beta_2]^\times)$ containing the
  element $\delta\beta_1(1-\beta_2\beta_1^{-1})$. We note that
  $(1-\beta_2\beta_1^{-1})$ belongs to $1+\mathfrak{p}_{F[\beta_2]}$
  and $\beta_1\in \delta F_0$. Since $k_D$ is a quadratic extension of
  $k_{F_0}$, we get that $F_0^\times$ is contained in
  ${\rm Nr}_{F[\beta_2]/D}(F[\beta_2]^\times)$.  Hence, we have
  $\delta\beta_1\in F_0$ and we get that $\delta(\beta_1-\beta_2)$
  belongs to ${\rm Nr}_{F[\beta_2]/D}(F[\beta_2]^\times)$. Now, using
  the criterion \eqref{type_B_X_beta_criterion} we get that the set
  $\mathfrak{X}_\beta(F_0)$ is non-empty.
\end{proof}
\begin{lemma}\label{type_B_iso_q_1>q_2_generic}
  Let $F/F_0$ be a quadratic extension and let $\mathfrak{x}$ be a
  stratum of the type {\bf (B)} such that $(V_2, h)$ is
  isotropic. If $q_1>q_2$, then any representation in the set
  $\Pi_\mathfrak{x}$ is generic.
\end{lemma}
\begin{proof}
  Let $\sch{U}$ be the unipotent radical of a Borel subgroup of in the
  set $\mathfrak{X}_\beta(F_0)$. Let $\theta$ be any skew semisimple
  character in the set $\mathcal{C}(\Lambda, 0, \beta)$. We will first
  check that
\begin{equation}\label{type_B_iso_q_1>q_2_generic_eq_1}
  \res_{H^1(\Lambda, \beta)\cap U}\theta=\psi_\beta.
\end{equation}
The group $H^1(\Lambda, \beta)$ is equal to
$$P_1(\Lambda_\beta)
P_{(q_2/2)+}(\Lambda_{\beta_1}) P_{(q_1/2)+}(\Lambda).$$ We define
$H'$ to be the subgroup
$P_1(\Lambda_{\beta})P_{(q_2/2)+}(\Lambda)$. Now, using Lemma
\ref{bl_st_lem}, we get that $H'\cap U$ is equal to
$P_{(q_2/2)+}(\Lambda)\cap U$. Since $H^1(\Lambda, \beta)\cap U$ is
equal to $H^1(\Lambda, \beta)\cap (H'\cap U)$, we get that
$H^1(\Lambda,\beta)\cap U$ is equal to
$P_{(q_2/2)+}(\Lambda_{\beta_1})P_{(q_1/2)+}(\Lambda)\cap U$. Let
$g_1g_2\in U$ for some $g_1\in P_{(q_2/2)+}(\Lambda_{F[\beta_1]})$ and
$g_2\in P_{(q_1/2)+}(\Lambda)$. We recall the notations defined in
\ref{criterion_non-empty}. Let $v=v_1+v_2$ be a non-trivial vector
fixed by $U$.  If $v_1=0$, then we have $\lambda(h'(v_2, v_2))=0$ and
$\lambda(\beta_2h'(v_2, v_2))=0$. This implies that $\beta_2$
stabilises the kernel of $\lambda$ and this absurd. Hence, we get that
$v_1\neq 0$. We have
$$g_2(v_1+v_2)=g_1(v_1+v_2)=xv_1+g_1v_2.$$
Now, comparing both sides we get that
$x\in F^\times\cap P_{(q_1/2)+}(\Lambda)$. This implies that
${\bf 1}_{V_1}g_1{\bf 1}_{V_1}\in F^\times\cap P_{(q_1/2)+}(\Lambda)$.
Since, the determinant of $g_2g_1$ is equal to $1$, we get that
determinant of ${\bf 1}_{V_2}g_1{\bf 1}_{V_2}$ belongs to
$F^\times\cap P_{(q_1/2)+}(\Lambda)$. From the definition of simple
character (see \cite[Definition 3.2.3(a)]{Orrangebook}), we get
\eqref{type_B_iso_q_1>q_2_generic_eq_1}. The Lemma is now a
consequence of Proposition \ref{bl_st_prop}.
\end{proof}
\subsection{The case where \texorpdfstring{$(V_2, h)$}{} is
  anisotropic}
As in the case where $(V_2, h)$ is isotropic, the genericity of a
cuspidal representation in the set $\Pi_\mathfrak{x}$ depends only on
the integers $q_1$ and $q_2$. However, the condition for genericity
becomes the opposite to the case where $(V_2, h)$ is isotropic, i.e.,
the inequality $q_2>q_1$ is necessary and sufficient condition for
genericity of a cuspidal representation in $\Pi_\mathfrak{x}$.
\begin{lemma}\label{type_B_aniso_q_1>q_2}
  Let $F/F_0$ be an unramified extension, and let $\mathfrak{x}$ be a
  stratum of type ${\bf (B)}$ such that $(V_2, h)$ is anisotropic.
  If $q_1>q_2$, then the set $\mathfrak{X}_\beta(F_0)$ is empty, and
  every cuspidal representation contained in the set
  $\Pi_\mathfrak{x}$ is non-generic.
\end{lemma}
\begin{proof}
  Let $\pi=\ind_{J^0(\Lambda, \beta)}^G\kappa$ be a generic
  representation in the set $\Pi_\mathfrak{x}$. The group $P(\Lambda)$
  is a special parahoric subgroup of $G$. Using the Iwasawa
  decomposition $G=P(\Lambda)B$, we have
\begin{equation}\label{type_B_aniso_q_1>q_2_eq_1}
\ho_{J^0(\Lambda, \beta)^g\cap U}(\kappa^g, \Psi)\neq 0,
\end{equation}
for some $g\in P(\Lambda)$ and a character $\Psi$ of $U$.

Let $q_1=4m_1+2r$ and $q_2=2m_2+1$, for some integers $m_1$, $m_2$ and
$r\in \{0,1\}$.  Note that $P_{(q_1/2)+}(\Lambda)\cap U_{\rm der}$ is
equal to $U_{\rm der}(m_1)$ (see \ref{type_B_unram_lat_seq_info} and
\eqref{unram_fil_2} for an explicit description of the lattice
sequence $\Lambda$ and the induced filtrations on $\End_F(V)$). We
have $\nu_{\Lambda}(e_1)=1$, and we get that $ge_1=av_1+\varpi_0
bv_2+cv_3$, for some $a, b, c\in \mathfrak{o}_F$. Since $e_{1}$ is
isotropic, we get that
$$a\sigma(a)+\varpi_0b\sigma(b)+ c\sigma(c)=0.$$
Since $\nu_\Lambda(e_1)=1$, the above equality implies that
$a, c\in \mathfrak{o}_F^\times$. Note that
$\nu_{F[\beta_2]}(\beta_2\beta_1^{-1})>0$, and in the basis
$(v_2, v_3)$ for $V_2$, the element $\beta_2\beta_1^{-1}\in \End_F(V)$
belongs to the following lattice of $\End_F(V_2)$:
$$\begin{pmatrix}\mathfrak{o}_F&\mathfrak{p}_F\\
  \mathfrak{o}_F&\mathfrak{o}_F\end{pmatrix}.$$
Hence, we simultaneously get that $h(ge_1, \beta ge_1)\neq 0$ and
$$\nu_F(h(ge_1, \beta ge_1))=
\nu_F(\beta_1 a\sigma(a)+h(\varpi_0 bv_2+cv_3, \beta_2(\varpi_0
bv_2+cv_3)))=\nu_F(\beta_1)=-2m_1-r\leq -m_1$$ Now, using Lemma
\ref{basic_inequality}, we get that the character $\psi_\beta^{g}$ is
non-trivial on the group $P_{(q_1/2)+}(\Lambda)\cap U_{\rm
  der}$. Hence, we get a contradiction for the assumption
\eqref{type_B_aniso_q_1>q_2_eq_1}. This shows that the representation
$\pi$ in the set $\Pi_\mathfrak{x}$ is non-generic. Using the Iwasawa
decomposition $G=P(\Lambda)B$, we get that
$h(ge_1, \beta ge_1)\neq 0$, for all $g\in G$. Hence, the set
$\mathfrak{X}_\beta(F_0)$ is empty.
\end{proof}
\begin{lemma}\label{type_B_ram_aniso_q_1>q_2}
  Let $F/F_0$ be a ramified extension, and let $\mathfrak{x}$ be a
  stratum of type {\bf (B)} such that $(V_2, h)$ is
  anisotropic. If $q_1>q_2$, then the set $\mathfrak{X}_\beta(F_0)$ is
  empty, and every representation in the set $\Pi_\mathfrak{x}$ is
  non-generic.
\end{lemma}
\begin{proof}
  Let $\pi=\ind_{J^0(\Lambda, \beta)}^G\kappa$ be a generic
  representation in the set $\Pi_\mathfrak{x}$. 
  The group $P^0(\Lambda)$ is a special parahoric subgroup of $G$, and
  using the Iwasawa decomposition $G=P(\Lambda)B$, we get that 
\begin{equation}\label{type_B_ram_aniso_q_1>q_2_eq_1}
\ho_{J^0(\Lambda, \beta)^g\cap U}(\kappa^g, \Psi)\neq 0,
\end{equation}
for some $g\in P(\Lambda)$, and a non-trivial character $\Psi$ of $U$.

Let $q_1=4m_1+2$ and $q_2=4m_2$, for some integers $m_1$ and $m_2$.
Observe that $P_{(q_1/2)+}(\Lambda)^g\cap U_{\rm der}$ is equal to
$U_{\rm der}(\lceil m_1/2\rceil)$. Since $\nu_{\Lambda}(e_{1})=0$, we
get that $ge_{1}=av_1+bv_2+cv_3$ for some $a, b,c\in
\mathfrak{o}_F$. As the vector $e_1$ is isotropic we get that
$$\lambda_1a\sigma(a)+\lambda_2b\sigma(b)+\lambda_3c\sigma(c)=0.$$
Now, the space $(V_2, h)$ is anisotropic and this implies that
$a\in\mathfrak{o}_F^\times.$ Now, we have
\begin{align*}
  h(ge_1, \beta ge_1)=&\beta_1a\sigma(a)+
                        h(bv_2+cv_2, \beta_2(bv_2+cv_3))\\
  =&\beta_1(a\sigma(a)+\beta_1^{-1}h(bv_2+cv_2,
     \beta_2(bv_2+cv_3))).
\end{align*}
Since, $\nu_{F[\beta_2]}(\beta_2\beta_1^{-1})>0$, we simultaneously
get that $h(ge_1, \beta ge_1)\neq 0$, and
$\nu_F(\delta h(ge_1, \beta ge_1))$ is equal to
$\nu_F(\delta\beta_1)$.  From the inequality
$\nu_F(\delta\beta_1)=-m_1\leq -\lceil m_1/2\rceil$, we get that
$\psi_\beta^g$ is a non-trivial character on
$P_{(q_1/2)+}(\Lambda)^g\cap U_{\rm der}$. This is a contradiction to
the assumption \eqref{type_B_ram_aniso_q_1>q_2_eq_1}. Using the
Iwasawa decomposition $G=P(\Lambda)B$, we get that
$h(ge_1, \beta ge_1)\neq 0$, for all $g\in G$. This shows that the set
$\mathfrak{X}_\beta(F_0)$ is empty.
\end{proof}
In the case where $(V_2, h)$ is anisotropic, we could not use the
criterion in \ref{criterion_non-empty}. However, the following
observation motivates the fact that $\mathfrak{X}_\beta(F_0)$ is
non-empty in the case where $q_2>q_1$. We suppose that $F/F_0$ is
unramified. In the basis $(v_1, v_2, v_3)$ consider a skew element
$\beta$ of the form
\begin{equation}\label{type_B_q_1>q_2_exp}
  \begin{pmatrix}0&0&0\\0&0&\beta_2\\0&\beta_3&0\end{pmatrix}.
      \end{equation}
      Let $e_1$ be an isotropic vector in $\langle v_1, v_3\rangle$.
      Let $\sch{B}$ be a Borel subgroup of $\sch{G}$ such that $B$
      fixes the space $\langle e_1\rangle$. The Borel subgroup
      $\sch{B}$ belongs to the set $\mathfrak{X}_\beta(F_0)$. Now, the
      class of $\beta$ in the quotient
      $a_{-n}(\Lambda)/a_{1-n}(\Lambda)$, is represented by a matrix
      as in \eqref{type_B_q_1>q_2_exp}.  This suggests that we may
      lift a point from the special fibre of an integral model of
      $\mathfrak{X}_\beta$, and we do this in the following lemma.
      \begin{lemma}\label{type_B_aniso_q_2>q_1_exis_whit_mod}
        Let $F/F_0$ be a quadratic extension and let $\mathfrak{x}$ be
        a stratum of type {\bf (B)} such that $(V_2, h)$ is
        anisotropic. If $q_2>q_1$, then the set
        $\mathfrak{X}_\beta(F_0)$ is non-empty.
\end{lemma}
\begin{proof}
  We will lift a point from the special fibre of a smooth model for
  $\mathfrak{X}_\beta$.  Let $(v_1, v_2, v_3)$ be a basis of $V$ as
  defined in \ref{type_B_unram_lat_seq_info}, when $F/F_0$ is
  unramified, and let $(v_1, v_2, v_3)$ be a basis of $V$ as defined
  in \ref{type_B_ram_lat_seq_info}, when $F/F_0$ is ramified. Let
  $\beta=(\beta_{ij})$ be the matrix representation of $\beta$ in the
  basis $(v_1, v_2, v_3)$. We have $\beta_{11}=\beta_1$.

  First consider the case where $F/F_0$ is unramified. A Borel
  subgroup $\sch{B}$, corresponding to the line
  $\langle xv_1+yv_2+zv_3\rangle$, belongs to $\mathfrak{X}_\beta$ if
  and only if $(x,y, z)$ satisfy the following equations:
$$\varpi x\sigma(x)+y\sigma(y)+\varpi z\sigma(z)=0$$
and 
$$\varpi \beta_1x\sigma(x)+\beta_{22}y\sigma(y)+
\beta_{33}\varpi z\sigma(z)+ \beta_{23}y\sigma(z)+\varpi  \beta_{32}
z\sigma(y)=0.$$ Changing $y$ to $\varpi y'$ and rescaling the second
equation by $\varpi^{\alpha-2}$ with $\alpha=-\nu_F(\beta_{32})$, we
get the following set of equations:
\begin{equation}\label{type_B_aniso_q_2>q_1_sm_mod_1}
  x\sigma(x)+\varpi y'\sigma(y')+z\sigma(z)=0
\end{equation}
and 
\begin{equation}\label{type_B_aniso_q_2>q_1_sm_mod_2}
  \varpi^{\alpha-1}\beta_1x\sigma(x)+\varpi^{\alpha}\beta_{22}y'\sigma(y')+
  \varpi^{\alpha-1}\beta_{33}z\sigma(z)+
  \varpi^{\alpha-1}\beta_{23}y'\sigma(z)+\varpi^{\alpha}\beta_{32} z\sigma(y')=0.
\end{equation}
Note that the coefficients of \eqref{type_B_aniso_q_2>q_1_sm_mod_2}
are integral and the two equations
\eqref{type_B_aniso_q_2>q_1_sm_mod_1} and
\eqref{type_B_aniso_q_2>q_1_sm_mod_2} define a flat closed subscheme
$\mathcal{X}_\beta$ of $\mathbb{P}^5_{\mathfrak{o}_{F_0}}$ such that
the generic fibre is $\mathfrak{X}_\beta$. The special fibre is given
by the set of equations
$$x\sigma(x)+z\sigma(z)=0$$
and $$C_1(y'\sigma(z)-z\sigma(y'))+C_2x\sigma(x)+ C_3z\sigma(z)=0,$$
where $C_1=\overline{\varpi^{\alpha-1}\beta_{23}}$,
$C_2=\overline{ \varpi^{\alpha-1}\beta_1}$, and
$C_3=\overline{\varpi^{\alpha-1}\beta_{33}}$.  Note that $C_1\neq 0$,
and therefore, the special fibre $\overline{\mathcal{X}}_\beta$ is
smooth. Since $\mathcal{X}_\beta$ is flat we get that
$\mathcal{X}_\beta$ is smooth over $\mathfrak{o}_F$. The special fibre
has a rational point and hence by Hensel's lemma the set
$\mathfrak{X}_\beta(F_0)$ is non-empty.

Now, consider the case where $F/F_0$ is a ramified extension. A Borel
subgroup $\sch{B}$ of $\sch{G}$ fixing the isotropic subspace
$\langle av_1+bv_2+cv_3\rangle$, belongs to $\mathfrak{X}_\beta$ if
and only if:
$$
\lambda_1a\sigma(a)+ \lambda_2b\sigma(b)+\lambda_3c\sigma(c)=0,
  $$
  and
 $$ \lambda_1\beta_1a\sigma(a)+\lambda_2\beta_{22}b\sigma(b)
    +\lambda_3\beta_{33}c\sigma(c)+
    b\sigma(c)\beta_{23}\lambda_2+\lambda_3\beta_{32}c\sigma(b)=0.
    $$
    In the present case $F[\beta_2]$ is an unramified extension of
    $F$. After rescaling by a power of $\varpi$, if necessary, we
    may assume that 
$$\beta_2=\begin{pmatrix}\beta_{22}&\beta_{23}\\
  \beta_{32}&\beta_{33}\end{pmatrix}\in
\begin{pmatrix}\mathfrak{o}_F&\mathfrak{o}_F\\
  \mathfrak{o}_F&\mathfrak{o}_F\end{pmatrix},$$ and
$\beta_1\in \mathfrak{o}_F$. Because $\beta_1$ is skew, we get that
$\beta_1\in \mathfrak{p}_F$. By a change of variable $b$ to
$\varpi b'$, the above set of equations become:
$$\lambda_1a\sigma(a)+\lambda_2\varpi_0b'\sigma(b')
+\lambda_3c\sigma(c)=0
$$
and
$$\varpi^{-1}\lambda_1\beta_1a\sigma(a)-\varpi\lambda_2\beta_{22}b\sigma(b)
+\varpi^{-1}\lambda_3\beta_{33}c\sigma(c)+ \lambda_2\beta_{23}
b\sigma(c)-\lambda_3\beta_{32}c\sigma(b)=0.$$ Since $\beta$ is skew,
we get that $\sigma(\beta_{22})+\beta_{22}=0$,
$\sigma(\beta_{33})+\beta_{33}=0$, and
$\lambda_1\beta_{23}=-\lambda_3\sigma(\beta_{32})$. Hence, the above
two equations have integral coefficients.  The above two equations
define a flat projective sub-variety $\mathcal{X}_\beta$ in
$\mathbb{P}^5_{\mathfrak{o}_{F_0}}$ with generic fibre
$\mathfrak{X}_\beta$. The special fibre is given by
$$\overline{\lambda_1}a^2+\overline{\lambda_3}c^2=0,$$
and
$$C_1a^2+C_2c^2+C_3bc=0,$$
where $C_1=\overline{\varpi^{-1}\lambda_1\beta_1}$,
$C_2=\overline{\varpi^{-1}\lambda_3\beta_{33}}$, and
$C_3=\overline{\lambda_2\beta_{23}}$. Clearly $C_{3}\neq 0$ as the
element $\beta_2$ is minimal.  The special fibre is smooth and hence,
$\mathcal{X}_\beta$ is a smooth model for $\mathfrak{X}_\beta$. Note
that the special fibre has a rational point. Using Hensel's lemma, we
get that the set $\mathfrak{X}_\beta(F_0)$ is non-empty.
\end{proof}
\begin{lemma}\label{type_B_aniso_q_2>q_1}
  Let $F/F_0$ be a quadratic extension and let $\mathfrak{x}$ be a
  stratum of type {\bf (B)} such that $(V_2, h)$ is anisotropic. If
  $q_2>q_1$, then every representation in the set $\Pi_\mathfrak{x}$
  is generic.
\end{lemma}
\begin{proof}
  Using Lemma \ref{type_B_aniso_q_2>q_1_exis_whit_mod}, we get that
  the set $\mathfrak{X}_\beta(F_0)$ is non-empty. Let $\sch{U}$ be the
  unipotent radical of a Borel subgroup in the set
  $\mathfrak{X}_\beta(F_0)$.  Note that the group
  $H^1(\Lambda, \beta)$ is equal to
  $P_1(\Lambda_\beta)P_{(q_2/2)+}(\Lambda)$ and it follows from Lemma
  \ref{bl_st_lem} that $H^1(\Lambda, \beta)\cap U$ is equal to
  $P_{(q_2/2)+}(\Lambda)\cap U$. Hence, we get that
$$\res_{H^1(\Lambda, \beta)\cap U}\theta=\psi_\beta,$$
where $\theta$ is any skew semisimple character of
$H^1(\Lambda, \beta)$. Now, genericity is a consequence of Proposition
\ref{bl_st_prop}. 
\end{proof}
\section{Non simple type (C) strata}\label{type_C_sec}
We say that a skew semisimple stratum $[\Lambda, n, 0, \beta]$,
denoted by $\mathfrak{r}$, is of type {\bf (C)} if the underlying
splitting of $\mathfrak{x}$ is given by $V=V_1\perp V_2$ with
$\dim_FV_i=i$, and $[F[\beta_i]: F]=1$. Here, $\beta_i$ is equal to
${\bf 1}_{V_i}\beta {\bf 1}_{V_i}$, for $i\in\{1,2\}$. From the
definition of skew semisimple stratum, we have
$\sigma(\beta_i)=-\beta_i$, for $i\in\{1,2\}$, and
$\beta=\beta_1+\beta_2$. We will show that every representation
contained in the set $\Pi_\mathfrak{x}$ is non-generic.
\subsection{Lattice sequences}\label{type_C_lat_seq_disc}
We will describe the lattice sequences up to $G_\beta$-conjugacy. Note
that $\Lambda$ is a lattice sequence on $V$ such that
$P^0(\Lambda_\beta)$ is a maximal parahoric subgroup of $G_\beta$. We
will  fix a Witt-basis of $(V, h)$ which gives a splitting for
these lattice sequences.
\subsubsection{The unramified case}\label{type_C_unram_seq}
Consider the case where $F/F_0$ is unramified and $(V_2,h)$ is
isotropic. We fix a Witt-basis $(e_1, e_0, e_{-1})$ for $(V, h)$ such
that $e_1, e_{-1}\in V_2$. The lattice sequence $\Lambda$, upto
$G_\beta$-conjugation, is given by one of the following two lattice
sequences: 
$$e(\Lambda)=2,\  \Lambda(-1)=\Lambda(0)=\mathfrak{o}_Fe_1\oplus
\mathfrak{o}_Fe_0\oplus \mathfrak{o}_Fe_{-1}$$ or
$$e(\Lambda)=2,\ \Lambda(0)=\mathfrak{o}_Fe_1\oplus\mathfrak{o}_Fe_0\oplus 
\mathfrak{p}_Fe_{-1}\ \text{and}\ \Lambda(1)=
\mathfrak{o}_Fe_1\oplus\mathfrak{p}_Fe_0\oplus \mathfrak{p}_Fe_{-1}.$$
If $F/F_0$ is an unramified extension and $(V_2, h)$ is anisotropic,
we fix vectors $v_1\in V_1$, $v_2, v_3\in V_2$ such that $(v_1, v_2,
v_3)$ is an orthogonal basis for $V$ and
$$h(v_1, v_1)=h(v_3, v_3)=\varpi\ \text{and}\ h(v_2, v_2)=1.$$
The lattice sequence $\Lambda$, upto $G_\beta$-conjugation, is
given by the following lattice sequence with
$$e(\Lambda)=2,\ \Lambda(0)=\mathfrak{o}_Fv_1\oplus
\mathfrak{o}_Fv_2\oplus\mathfrak{o}_Fv_3 \ \text{and}\
\Lambda(1)=\mathfrak{o}_Fv_1\oplus
\mathfrak{p}_Fv_2\oplus\mathfrak{o}_Fv_3.$$ Using Lemma
\ref{type_B_seq_unram_basis}, there exists a Witt-basis
$(e_1, e_0, e_{-1})$ of $V$ with
$e_1, e_{-1}\in \langle v_1, v_3\rangle$ such that
$(e_1, e_0, e_{-1})$ provides a splitting for the lattice sequence
$\Lambda$, and we have
$$\Lambda(0)= \mathfrak{o}_Fe_1\oplus
\mathfrak{o}_Fe_0 \oplus \mathfrak{p}_Fe_{-1}, \ \text{and}\
\Lambda(1)=
\mathfrak{o}_Fe_1\oplus \mathfrak{p}_Fe_0 \oplus
\mathfrak{p}_Fe_{-1}.
$$
\subsubsection{The ramified case}\label{type_C_ram_seq}
In this part, we assume that the extension $F/F_0$ is ramified. If
$(V_2, h)$ is isotropic, we fix a Witt-basis $(e_1, e_{-1})$ for
$(V_2, h)$ and an unit vector $e_0\in V_1$ such that
$(e_1, e_0, e_{-1})$ is a Witt-basis for $(V, h)$. The lattice
sequence $\Lambda$, upto $G_\beta$-conjugation, is given by the
following lattice sequence:
$$e(\Lambda)=2, \Lambda(0)=\mathfrak{o}_Fe_1\oplus 
\mathfrak{o}_Fe_0\oplus \mathfrak{p}_Fe_{-1}\ \text{and}\
\Lambda(1)=\mathfrak{o}_Fe_1\oplus \mathfrak{p}_Fe_0\oplus
\mathfrak{p}_Fe_{-1}.$$

Now assume that $(V_2, h)$ is anisotropic, and fix an orthogonal basis
$(v_1, v_2, v_3)$ of $(V, h)$ such that $v_1\in V_1$ and
$v_2, v_3\in V_2$, $h(v_i, v_i)\in \mathfrak{o}_F^\times$, and the
hermitian space $(\langle v_1, v_3\rangle, h)$ is isotropic. We denote
by $\lambda_i$ the constant $h(v_i, v_i)$, for $1\leq i\leq 3$. Up to
$G_\beta$-conjugacy, the lattice sequence $\Lambda$ is given by the
following lattice sequence:
$$e(\Lambda)=2,\ \Lambda(-1)=\Lambda(0)=\mathfrak{o}_Fv_1\oplus
\mathfrak{o}_Fv_2\oplus\mathfrak{o}_Fv_3.$$ There exists a Witt-basis
$(e_1, e_0, e_{-1})$ for the space $(V, h)$ with
$e_1, e_{-1}\in \langle v_1, v_3\rangle$ such that
$$\Lambda(-1)=\Lambda(0)=\mathfrak{o}_Fe_1\oplus
\mathfrak{o}_Fe_0\oplus\mathfrak{o}_Fe_{-1}$$ The groups $J^0(\Lambda,
\beta)$ and $H^1(\Lambda, \beta)$ are given by
$$P_0(\Lambda_{\beta})P_{(n/2)}(\Lambda)$$
and
$$P_1(\Lambda_{\beta})P_{(n/2)+}(\Lambda)$$
respectively.
\subsection{The case where
  \texorpdfstring{$(V_2, h)$}{} is anisotropic}
In this subsection, we assume that the hermitian space $(V_2, h)$ is
anisotropic. To show that any representation $\pi\in \Pi_\mathfrak{x}$
is non-generic, we will prove that the character $\psi_{\beta}^g$ is
non-trivial on $U_{\rm der}\cap P_{(n/2)+}(\Lambda)$, for all
$g\in P(\Lambda)$.
\begin{lemma}\label{type_C_aniso}
  Let $F/F_0$ be any quadratic extension and let $\mathfrak{x}$ be a
  strata of type {\bf (C)} such that $(V_2, h)$ is anisotropic. Every
  cuspidal representation in the set $\Pi_\mathfrak{x}$ is
  non-generic.
\end{lemma}
\begin{proof}
  We recall, from Subsection \ref{type_C_lat_seq_disc} the two
  $F$-basis $(v_1, v_2, v_3)$ and $(e_1, e_0, e_{-1})$ for the vector
  space $V$, when $(V_2, h)$ is anisotropic. Note that $\Lambda$ is
  uniquely determined by $\beta$, and recall the description from
  Subsection \ref{type_C_lat_seq_disc}. Let $\sch{U}$ be the unipotent
  radical of a Borel subgroup $\sch{B}$ of $\sch{G}$ such that $U$
  fixes the vector $e_1$. Since $(e_1, e_0, e_{-1})$ provides a
  splitting for the lattice $\Lambda$ and $P^0(\Lambda)$ is a special
  parahoric subgroup, we get that $G=P(\Lambda)B$. Let $\pi$ be a
  cuspidal representation in the set $\Pi_\mathfrak{x}$. Then we have
  $\pi\simeq \ind_{J^0(\Lambda, \beta)}^G\kappa$, where
  $(J^0(\Lambda, \beta), \kappa)$ is a Bushnell--Kutzko type contained
  in $\pi$. Now, assume that $\pi$ is generic. Then there exists a
  $g\in P(\Lambda)$ and a non-trivial character $\Psi$ of $U$ such
  that
\begin{equation}\label{type_C_aniso_eq_1}
\ho_{J^0(\Lambda, \beta)^g\cap U}(\kappa^g, \Psi)\neq 0.
\end{equation}

First consider the case where $F/F_0$ is a ramified extension. Let
$ge_1=av_1+bv_2+cv_3$, for some $a, b, c\in \mathfrak{o}_F$. Since
$e_1$ is isotropic, we get that
$$\lambda_1a\sigma(a)+\lambda_2b\sigma(b)+\lambda_3c\sigma(c)=0.$$
Since
$-\lambda_2\lambda_3^{-1}\not\in {\rm Nr}_{F/F_0}(F^\times)$, the
above equality implies
that $a\in \mathfrak{o}_F^\times$. Let
$\nu_{\Lambda}(\beta)=-n=-(4m+2r)$, for some integer $m$ and
$r\in \{0,1\}$.  Since $\sigma(\beta)=-\beta$, we get that $r=1$. Now,
$(n/2)+=2m+2$ and we have $P_{(n/2)+}(\Lambda)\cap U_{\rm der}$ is
equal to $U_{\rm der}([(m+1)/2])$. We also have
$$\nu_{F/F_0}(\delta h(ge_1, \beta ge_1))
=\nu_{F/F_0}(\delta(\beta_1-\beta_2)\lambda_1a\sigma(a))=-m.$$
As $\nu_{F/F_0}(h(ge_1, ge_1)\leq -[(m+1)/2]$, for any $m\geq 1$, we
get that $\psi_\beta^g$ is a non-trivial character on the group
$P_{(n/2)+}(\Lambda)\cap U_{\rm der}$. This is a contradiction to the
equation \eqref{type_C_aniso_eq_1}. Hence, any representation in the
set $\Pi_\mathfrak{x}$ is non-generic.

Consider the case where $F/F_0$ is unramified.  Since the isotropic
vector $e_1$ belongs to the lattice $\Lambda(1)$, we get that
$ge_1=av_1+bv_2+cv_3$, for some $ a, c\in \mathfrak{p}_F$ and
$b\in \mathfrak{p}_F$, with
$$\varpi a\sigma(a)+ b\sigma(b)+\varpi c\sigma(c)=0.$$
By a change of variable: $b'=\varpi b$, we have
$$a\sigma(a)+\varpi b'\sigma(b')+c\sigma(c)=0.$$
Since $\nu_\Lambda(e_1)=1$, the above equality implies that
$a, c\in \mathfrak{o}_F^\times$.  Now, we have
\begin{align*}
  \nu_F(h(ge_1, \beta ge_1))=&\nu_F(\beta_1\varpi a\sigma(a)+
  \beta_2( b\sigma(b)+ \varpi c\sigma(c)))\\
  &=\nu_F((\beta_1-\beta_2))+1
\end{align*}
We note that $2\nu_F(\beta_1)=\nu_{\Lambda_1}(\beta_1)$ and
$\nu_F(\beta_2)=\nu_{\Lambda_2}(\beta_2)$. We also have
$$-n=\nu_{\Lambda}(\beta)=\min\{\nu_{\Lambda_1}(\beta_1),
\nu_{\Lambda_2}(\beta_2)\}.$$ Assume that
$-n=\nu_{\Lambda_1}(\beta_1)\leq \nu_{\Lambda_2}(\beta_2)$.  In this
case, $n=4m+2r$, where $m$ is an integer and $r\in \{0,1\}$.  Now, we
have $(n/2)+=2m+r+1$ and the group $P_{(n/2)+}(\Lambda)\cap U_{\rm
  der}$ is equal to $U_{\rm der}(m)$. We may have two possibilities
either $\nu_{\Lambda_2}(\beta_2)\leq \nu_F(\beta_1)=-(2m+r)$ or
$\nu_{\Lambda_2}(\beta_2)\geq \nu_F(\beta_1)=2m+r$. In the first case,
we have
$$\nu_F(\beta_1-\beta_2)+1=\nu_F(\beta_2)+1\leq -(2m+r)+1\leq -m.$$
In the second case, we have
$$\nu_F(\beta_1-\beta_2)+1=\nu(\beta_1)+1=-(2m+r)+1\leq -m.$$
Hence, the character $\psi_{\beta^g}$ on $P_{(n/2)+}(\Lambda)$ is
non-trivial on $P_{(n/2)+}(\Lambda)\cap U_{\rm der}$ and we obtain a
contradiction to the equation \eqref{type_C_aniso_eq_1}.

Assume that $-n=\nu_{\Lambda_2}(\beta_2)\leq \nu_{\Lambda_1}(\beta_1)$
and set $n=4m+r$, for some integer $m$ and $0\leq r\leq 3$. In this
case, the group $P_{(n/2)+}(\Lambda)\cap U_{\rm der}$ is equal to
$U_{\rm der}(m)$. Since $\nu_{\Lambda_1}(\beta)\leq \nu_F(\beta_1)$, we
get that
$$\nu_F(\beta_1-\beta_2)+1=\nu_{\Lambda_2}(\beta_2)=-(4m+r)+1\leq -m.$$
The above inequality implies that the character $\psi_{\beta}^g$ is
non-trivial on the group $P_{(n/2)+}(\Lambda)\cap U_{\rm der}$ and
hence we get a contradiction to the equation
\eqref{type_C_aniso_eq_1}. In every case, we get that the
representation $\pi$ is non-generic.
\end{proof}
\subsection{\texorpdfstring{$(V_2, h)$}{} is isotropic}
In this part, we assume that $(V_2, h)$ is isotropic. Note that the
set $\mathfrak{X}_\beta(F_0)$ is non-empty. However, it turns out that
every cuspidal representation in the set $\Pi_\mathfrak{x}$ is
non-generic; the essential reason for this is that any cuspidal
representation of $P_0(\Lambda_\beta)/P_1(\Lambda_\beta)$ is
generic. The group $P_0(\Lambda_\beta)/P_1(\Lambda_\beta)$ is equal to
$U(1,1)(k_F/k_{F_0})\times U(1)(k_F/k_{F_0})$, when $F/F_0$ is
unramified and is equal to ${\rm SL}_2(k_F)\times\{\pm 1\}$, when
$F/F_0$ is ramified.

Let $\sch{B}$ be the Borel subgroup of $\sch{G}$ such that $B$ fixes the
subspace $\langle e_1\rangle$. Let $\sch{U}$ be the unipotent radical
of $\sch{B}$.  Although $P^0(\Lambda)$ is a special parahoric subgroup
of $G$, it is convenient to use the decomposition
$$G=\mathcal{I}W_GB$$
for Mackey-decompositions. Here, $\mathcal{I}$ is the Iwahori subgroup
contained in the subgroups $P(\Lambda)$, where $\Lambda$ varies over
the two (representatives for $G_\beta$ conjugacy classes) lattice
sequences in \ref{type_C_unram_seq} and \ref{type_C_ram_seq} when
$F/F_0$ is unramified and ramified respectively.
\subsubsection{Shallow elements}
To understand Mackey decompositions, we need to control the
conjugation by elements in $\mathcal{I}$. For the present purposes, it
is easy to understand the conjugation action of an element which is
contained in a sufficiently small compact subgroup of $\mathcal{I}$.
So, we introduce a measure of shallowness, relative to the
group~$P_{(n/2)+}(\Lambda)$, on the elements $u(x,y)$ and
$\bar{u}(x, y)$ in $\mathcal{I}\cap U$ and $\mathcal{I}\cap U^w$
respectively (here, $w$ is the non-trivial element in $W_G$).

We begin with defining an integer $d(\mathfrak{x}, w, x)$. The main
purpose of the definition of $d(\mathfrak{x}, w, x)$ becomes apparent
in Lemma \ref{type_C_iso_weight}.  Let $n=4m+2r$, for some positive
integer $m$ and $r\in\{0,1\}$. Consider the case where $F/F_0$ is
unramified. For any $x, y\in F$ such that $x\sigma(x)+y+\sigma(y)=0$,
$w\in W_G$ and $\Lambda$ a lattice sequence defined in
\ref{type_C_unram_seq} or in \ref{type_C_ram_seq}, we set
$d(\mathfrak{x}, w, x)$ to be
\begin{eqnarray}
  d(\mathfrak{x}, w, x)=
  \begin{cases}
    \max\{1, m+1-\nu_F(x)\},&\ \text{if}\ \Lambda(0)\cap
    V_2=\mathfrak{o}_Fe_1\oplus \mathfrak{o}_Fe_{-1},\\
    \max\{0, m+r-\nu_F(x)\}&\ \text{if}\ \Lambda(0)\cap
    V_2=\mathfrak{o}_Fe_1\oplus \mathfrak{p}_Fe_{-1}, w=\id,\\
    \max\{2, m+r+1-\nu_F(x)\}&\ \text{if}\ \Lambda(0)\cap
    V_2=\mathfrak{o}_Fe_1\oplus \mathfrak{p}_Fe_{-1}, w\neq \id.
  \end{cases}
\end{eqnarray}
If $F/F_0$ is ramified, then the lattice sequences $\Lambda$, defined
in \ref{type_C_ram_seq}, up to $G_\beta$ conjugacy, is the unique
lattice sequence such that $P^0(\Lambda_{F[\beta]})$ is a maximal
parahoric subgroup in $G_\beta$.  Note that
$\Lambda(0)=\mathfrak{o}_Fe_1\oplus \mathfrak{p}_Fe_{-1}$. We define
the integer $d(\mathfrak{x}, w, x)$ as follows:
\begin{eqnarray}
  d(\mathfrak{x}, w, x)=
  \begin{cases}
    \max\{0, \lceil
    (m+r-1)/2-\nu_F(x)\rceil\}&\ \text{if}\ w=\id,\\
    \max\{1, \lceil (m+r)/2-\nu_F(x)\rceil\}&\ \text{if}\ w\neq \id.
  \end{cases}
\end{eqnarray}
Note that $d(\mathfrak{x}, w, x)$ is a constant for $\nu_F(x)>>0$, and
we denote this constant by $d(\mathfrak{x}, w)$. For example, when
$F/F_0$ is unramified,
$\Lambda(0)=\mathfrak{o}_Fe_1\oplus \mathfrak{p}_Ee_{-1}$, and
$w\neq \id$ we have $d(\mathfrak{x}, w)=2$. If $F/F_0$ is ramified and
$w\neq \id$, we have $d(\mathfrak{x}, w)=1$.
\subsubsection{}
With these preliminaries we are ready to prove that any representation
in the set $\Pi_\mathfrak{x}$ is non-generic. Let
$\overline{\sch{U}^w}$ be the unipotent radical of the opposite Borel
subgroup of $\sch{B}^w$ with respect to the torus $\sch{T}$, for all
$w\in W_G$. Recall that $\sch{T}$ is the maximal $F_0$-split torus of
$\sch{G}$ such that $T$ stabilises the decomposition
$$V=\langle e_1\rangle\oplus \langle e_0\rangle\oplus \langle e_{-1}\rangle. $$
\begin{lemma}\label{type_C_iso_weight}
  Let $F/F_0$ be any quadratic extension and let $\mathfrak{x}$ be a
  skew semisimple strata of type {\bf (C)} such that $(V_2, h)$ is
  isotropic. Let $u^-=\bar{u}(x,y)$ be an element of
  $\mathcal{I}\cap \overline{U^w}$, then we have
  $$U_{{\rm der}}^w(d(\mathfrak{x}, w, x))\subseteq 
  H^1(\Lambda, \beta)^{u^-}\cap U_{{\rm der}}^w.$$
\end{lemma}
\begin{proof}We prove the lemma in the case where $w=\id$; the case
  where $w\neq \id$ is entirely similar.  We have to show that
  $\{U_{{\rm der}}(d(\mathfrak{x}, \id, x))\}^{u^-}$ is contained in
  the group $H^1(\Lambda, \beta)$. First we note the following matrix
  identity
  \begin{equation}\label{type_C_iso_weight_eq_1}
\bar{u}(x, y)u(0, a)\bar{u}(-x, -y-x\sigma(x))=
\begin{pmatrix}1-a(x\sigma(x)+y)&a\sigma(x)&a\\
  ax(-y-x\sigma(x))&1+ax\sigma(x)&ax\\
  -ay(y+x\sigma(x))&a\sigma(x)y&ay+1\end{pmatrix}.
\end{equation}
Using the equality \eqref{type_C_iso_weight_eq_1}, for any element
$u(0, a)\in U_{{\rm der}}(d(\mathfrak{x}, \id, x))$, we get that 
 $$\bar{u}(x, y)u(0, a)\bar{u}(-x, -y-x\sigma(x))\in 
 P_1(\Lambda_{F[\beta]}) P_{(n/2)+}(\Lambda).$$
 The lemma now follows because we have
$$P_1(\Lambda_{F[\beta]})
P_{(n/2)+}(\Lambda)\subseteq H^1(\Lambda, \beta).$$
\end{proof}
\begin{lemma}\label{type_C_iso_sim_char_res}
  With the same assumptions and notations as in Lemma
  \ref{type_C_iso_weight}, for any skew semisimple character $\theta\in
  \mathcal{C}(\Lambda, 0, \beta)$ of the group
  $H^1(\Lambda, \beta)$, we have
$$\res_{U^w_{{\rm der}}(d(\mathfrak{x}, w, x))}\theta^{u^-}=
\psi_{\beta^{u^-}}.$$
\end{lemma}
\begin{proof}
  Let $u^+=u(0, a)$ be an element in
  $U_{{\rm der}}(d(\mathfrak{x}, w, x))$, and $u^{-}=\bar{u}(x,y)$ be
  any element as in Lemma \ref{type_C_iso_weight}. Assume that
  $u^-u^+(u^{-})^{-1}=g_1g_2$, where $g_1\in P_1(\Lambda_{F[\beta]})$
  and $g_2\in P_{(n/2)+}(\Lambda)$. From the matrix identity
  \ref{type_C_iso_weight_eq_1} and from the definition of
  $d(\mathfrak{x}, \id, x)$, the constant $1+ax\sigma(x)$ belongs to
  $F^\times\cap P_{(n/2)+}(\Lambda)$. Hence, we get that
  ${\bf 1}_{V_1}g{\bf 1}_{V_1}\in F^\times \cap P_{(n/2)+}(\Lambda)$.
  This implies that the determinant of the element
  ${\bf 1}_{V_2}g_2{\bf 1}_{V_2}$ is contained in
  $F^\times\cap P_{(n/2)+}(\Lambda)$. Now, from the definition of a
  simple character, we get that $\theta(u^-u^+(u^-)^{-1})$ is equal to
  $\psi_\beta(u^-u^+(u^-)^{-1})$ and we get the lemma.
\end{proof}
\begin{lemma}\label{type_C_iso_non_gen_at_id}
  With the same assumptions in the lemma \ref{type_C_iso_weight}, the
  restriction
  $$\res_{J^0(\Lambda, \beta)\cap U^w_{\rm der}}(\kappa\otimes\tau)$$ is
  equivalent to a direct sum of non-trivial characters.
\end{lemma}
\begin{proof}
  We essentially follow ideas from \cite[Theorem
  2.6]{reali_ep_fac_pas_st} and \cite[Theorem 4.3]{sp_4_genericity}.
  We prove this lemma in the case where $w=\id$, and the other case is
  similar. If $F/F_0$ is unramified, then $P_0(\Lambda)/P_1(\Lambda)$
  is isomorphic to
$$ U(1,
1)(k_F/k_{F_0})\times U(1)(k_F/k_{F_0}).$$
If $F/F_0$ is ramified, then $P_0(\Lambda)/P_1(\Lambda)$ is isomorphic
to
$${\rm SL}_2(k_F)\times \{\pm 1\}.$$
Let $\tilde{J}_1$ be the group
$(J^0(\Lambda, \beta)\cap U)J^1(\Lambda, \beta)$ and observe that we
have $\tilde{J}_1$ is equal to
$(J^0(\Lambda, \beta)\cap U_{\text{der}})J^1(\Lambda, \beta)$.  Note
that the image of $J^0(\Lambda, \beta)\cap U$ in the quotient
$P_0(\Lambda)/P_1(\Lambda)$ is its $p$-Sylow subgroup. Let
$\tilde{H}_1$ be the group
$(J^0(\Lambda, \beta)\cap U)H^1(\Lambda, \beta)$.

We observe that $\psi_\beta(u)=1$, for all $u\in U$. Let $u$ be an
element of $H^1(\Lambda, \beta)\cap U$. Write $u=g_2g_1g_1'$ where
$g_2\in P_{(n/2)+}(\Lambda)$,
$g_2\in P_1(\Lambda_{\beta_1})\cap U(V_2, h)$, and
$g_1'\in P_1(\Lambda_{\beta_1})\cap U(V_1, h)$.  Since $g_1'e_1=e_1$,
we get that $g_2g_1e_1=e_1$. Hence, we get that $g_2g_1\in U$ and
$g_1'=\id$. Now, the determinant of $g_2^{-1}$ and $g_1$ are the
same. With this observation, we get that
\begin{equation}\label{type_C_iso_non_gen_at_id_eq_new_1}
\res_{H^1(\Lambda, \beta)\cap U}\theta=\res_{H^1(\Lambda, \beta)\cap U}\psi_\beta=\id.
\end{equation}
Let $\theta$ be a skew semisimple character of $H^1(\Lambda, \beta)$
and $\eta$ be a Heisenberg lift of $\theta$ to the group
$J^1(\Lambda, \beta)$. Using the equality
\eqref{type_C_iso_non_gen_at_id_eq_new_1}, we define a character
$\tilde{\theta}$ of the group $\tilde{H}_1$ by setting:
$$\tilde{\theta}(jh)=\theta(h),\ \text{for all}\ 
j\in J^0(\Lambda, \beta)\cap U,\ h\in H^1(\Lambda, \beta).$$
Using \cite[Lemma 2.5]{reali_ep_fac_pas_st}, we get that the
representation
$\ind_{\tilde{H}_1\cap J^1(\Lambda, \beta)}^{J^1(\Lambda,
  \beta)}\tilde{\theta}$ is isomorphic to $\eta$.

Note that the group $J^1(\Lambda, \beta)\cap U_{\text{der}}$
is equal to $H^1(\Lambda, \beta)\cap U_{\text{der}}$ and we get that
the representation 
$$\res_{J^1(\Lambda, \beta)\cap U_{\text{der}}}\eta\simeq 
\res_{H^1(\Lambda, \beta)\cap U_{\text{der}}}\eta\simeq
(\id)^{\dim_\mathbb{C}\eta}.$$
This implies that $\eta$ extends as a representation of $\tilde{J}_1$
such that $J^1(\Lambda, \beta)\cap U_{\text{der}}$ acts trivially on
this extension; let us denote this extension by $\tilde{\eta}$. By
Frobenius reciprocity we get a map
\begin{equation}\label{type_C_iso_non_gen_at_id_eq_1}
\ind_{\tilde{H}_1}^{\tilde{J}_1}\tilde{\theta}\rightarrow \tilde{\eta}.
\end{equation}
The representation $\eta$ is irreducible. The dimension of the
representation $\ind_{\tilde{H}_1}^{\tilde{J}_1}\tilde{\theta}$ is
equal to $[J^1(\Lambda, \beta):H^1(\Lambda, \beta)]^{1/2}$. Hence, the
map \eqref{type_C_iso_non_gen_at_id_eq_1} is an isomorphism.

Let $[\lambda^m, n_1, 0, \beta]$ be a skew semisimple stratum such
that $\tilde{J}_1$ is equal to
$(P_1(\Lambda^m_\beta)\cap U_{\rm der})J^1(\Lambda, \beta)$. Let
$\theta_m$ be the skew semisimple character of $H^1(\Lambda^m, \beta)$
obtained as a transfer from the skew semisimple character $\theta$ of
$H^1(\Lambda^m, \beta)$. We note that the groups $\tilde{H}_1$ and
$H^1(\Lambda^m, \beta)$ have the Iwahori decomposition with respect to
the pair $(B, T)$. We then get that
$$\ho_{\tilde{H}_1\cap H^1(\Lambda^m, \beta)}(\tilde{\theta},
\theta_m)\neq 0.$$ This implies that
$$\ind_{\tilde{J}^1}^{P_1(\Lambda_m)}\tilde{\eta}\simeq
\ind_{J^1(\Lambda^m, \beta)}^{P_1(\Lambda_m)}\eta_m.$$ From the
uniqueness properties of $\beta$-extensions, we get that
$\res_{\tilde{J}_1}\kappa\simeq \tilde{\eta}$. This shows that the
representation
$$\res_{J^0(\Lambda, \beta)\cap
  U_{\text{der}}}(\kappa\otimes\tau)$$
is a direct sum of non-trivial characters.
\end{proof}
The following lemma is a numerical verification to be used in the
subsequent Lemma \ref{type_C_iso_non-gen}.
\begin{lemma}\label{type_C_iso_numerical}
  Let $F/F_0$ be any quadratic extension and let
  $x, y\in \mathfrak{o}_F$ such that $x\sigma(x)+y+\sigma(y)=0$. If
  $d(\mathfrak{x}, w, x)>d(\mathfrak{x}, w)$, then we have
$$\nu_{F_0}(\delta(\beta_1-\beta_2)x\sigma(x))\leq -d(\mathfrak{x}, w, x).$$
\end{lemma}
\begin{proof}
  Assume that $F/F_0$ is unramified and
  $\Lambda(0)\cap V_2=\mathfrak{o}_Fe_1\oplus
  \mathfrak{o}_Fe_{-1}$. If
  $d(\mathfrak{x}, w, x)>d(\mathfrak{x}, w)$, then we get that
  $\nu_F(x)<m$, and this implies that
$$2\nu_F(x)-2m-r\leq -(m+1)+\nu_F(x).$$ Now, consider the case
where
$\Lambda(0)\cap V_2=\mathfrak{o}_Fe_1\oplus \mathfrak{p}_Fe_{-1}$ and
assume that $d(\mathfrak{x}, w, x)>d(\mathfrak{x}, w)$. If $r=0$ and
$w=\id$, then we get that $\nu_F(x)<m$ which implies
that $$2\nu_F(x)-2m\leq -m+\nu_F(x).$$ If $r=0$ and $w\neq \id$, then
we get that $\nu_F(x)<m-1$ and hence
$$2\nu_F(x)-2m\leq -m-1+\nu_F(x).$$ If $r=1$ and $w=\id$, then we have
$\nu_F(x)<m+1$ and hence we get that
$$2\nu_F(x)-(2m+1)\leq -(m+1)+\nu_F(x).$$ Finally, we consider the
case where $r=1$ and $w\neq \id$; we then have $\nu_F(x)<m$. Hence, we
get that
  $$2\nu_F(x)-(2m+1)\leq -m-1+\nu_F(x).$$
  Now, assume that $F/F_0$ is a ramified extension and
  $d(\mathfrak{x}, w, x)>d(\mathfrak{x}, w)$. Note that
  $\nu_{F/F_0}(\beta_1-\beta_2)=-(2m+1)$, for some $m\in \mathbb{Z}$. If
  $w=\id$, then we have $\nu_{F/F_0}(x)<m/2$ and hence
  $$2\nu_{F/F_0}(x)-(2m+1)/2+1/2<-m/2+\nu_{F/F_0}(x)\leq
  -d(\mathfrak{x}, w, x).$$
If $w\neq \id$, then we have $\nu_{F/F_0}(x)<(m-1)/2$ and we have
$$2\nu_{F/F_0}(x)-(2m+1)/2+1/2<-(m+1)/2+\nu_{F/F_0}(x)\leq
-d(\mathfrak{x}, w, x).$$ From the above inequalities, in all
exhaustive cases, gives the required inequality:
$$\nu_{F_0}(\delta(\beta_2-\beta_1)x\sigma(x))\leq -d(\mathfrak{x}, w, x)$$
\end{proof}
\begin{lemma}\label{type_C_iso_non-gen}
  Let $F/F_0$ be an unramified extension and let $\mathfrak{x}$ be a
  strata of type {\bf (C)} such that $(V_2, h)$ is isotropic. Every
  cuspidal representation in the set $\Pi_\mathfrak{x}$ is
  non-generic.
\end{lemma}
\begin{proof}
  Let $\pi$ be a cuspidal representation in the set
  $\Pi_\mathfrak{x}$.  Then we have
  $\pi\simeq \ind_{J^0(\Lambda, \beta)}^G(\kappa\otimes\tau)$, where
  $(J^0(\Lambda, \beta), \kappa\otimes\tau)$ is a Bushnell--Kutzko
  type contained in $\pi$.  Assume that the representation $\pi$ is
  generic. Then there exists a $g\in \mathcal{I}$, an element
  $w\in W_G$, and a character $\Psi$ of $U^w$ such that
\begin{equation}\label{type_C_iso_non-gen_eq_1}
\ho_{J^0(\Lambda, \beta)^g\cap U^w}((\kappa\otimes\tau)^g, \Psi)\neq 0.
\end{equation}
We write $g=p^+u^-$ such that $p^+\in B^w\cap \mathcal{I}$ and
$u^-\in \overline{U^w}\cap \mathcal{I}$, where $\overline{U^w}$ is the
unipotent radical of the opposite Borel subgroup of $B^w$ with respect
to the torus $T$. From the expression \eqref{type_C_iso_non-gen_eq_1},
we get that
\begin{equation}\label{type_C_iso_non-gen_eq_2}
\ho_{J^0(\Lambda, \beta)^{u^-}\cap U^w}((\kappa\otimes\tau)^{u^-}, \Psi')\neq 0,
\end{equation}
for some character $\Psi'$ of $U^w$.  We set $e_w=we_1$ and
$e_{-w}=we_{-1}$, we then have
$$h(u^-e_w, \beta u^-e_w)=
\beta_1x\sigma(x)+\beta_2h(e_w+ye_{-w}, e_w+ye_{-w})
=(\beta_1-\beta_2)x\sigma(x).$$ Hence, $\nu_F(\delta h(u^-e_w, \beta
u^-e_w))$ is equal to $\nu_F(\delta(\beta_1-\beta_2)x\sigma(x))$.  If
$d(\mathfrak{x}, w, x)>d(\mathfrak{x}, w)$, then Lemma
\ref{type_C_iso_numerical} implies that
$$\nu_F(\delta(\beta_1-\beta_2)x\sigma(x))\leq -d(\mathfrak{x}, w, x).$$
Now, using Lemmas \ref{type_C_iso_sim_char_res} and
\ref{basic_inequality}, we get that
$$\res_{U_{\rm der}(d(\mathfrak{x}, w, x))}\theta=\psi_{\beta}^{u^-}\neq \id.$$ 
But, this is a contradiction to the assumption in
\eqref{type_C_iso_non-gen_eq_1}. Hence, we obtain
$d(\mathfrak{x}, w, x)=d(\mathfrak{x}, w)$ and this implies that
$u^-\in H^1(\Lambda, \beta)$. We may as well assume that
$u^-=\id$. The lemma now follows from Lemma
\ref{type_C_iso_non_gen_at_id}.
\end{proof}
\section{Non simple type (D) strata}\label{type_D_sec}
\subsection{Inducing data}
A skew semisimple stratum $[\Lambda, n, 0, \beta]$, denoted by
$\mathfrak{x}$, is of type {\bf (D)} if the underlying splitting is of
the form $V=V_1\perp V_2\perp V_3$ with $\dim_FV_i=1$, for
$1\leq i\leq 3$. We use the notation $W_i$ for the space
$\oplus_{j\neq i}V_j$, for $1\leq i\leq 3$.

When $F/F_0$ is unramified, we fix a vector $v_i\in V_i$ such that
$\nu_{F_0}(h(v_i, v_i))\in\{0,1\}$, for $1\leq i\leq 3$.  When $F/F_0$
is ramified, we fix a vector $v_i\in V_i$ such that
$\nu_{F_0}(h(v_i, v_i))=0$, for $1\leq i\leq 3$. We denote by
$\lambda_i$ the inner product $h(v_i, v_i)$, for $1\leq i\leq 3$.  We
have $\beta=\beta_1+ \beta_2+ \beta_3$, where
$\beta_i={\bf 1}_{V_i}\beta{\bf 1}_{V_i}$, for $1\leq i\leq 3$. The
lattice sequence $\Lambda$ is uniquely determined by the element
$\beta$ and we have $e(\Lambda)=2$. Let $\Lambda_i$ be the
$\mathfrak{o}_F$-lattice sequence $\Lambda\cap V_i$, for
$1\leq i\leq 3$.

Since, $\mathfrak{x}$ is a skew semisimple stratum of
type {\bf (D)}, we have
$$|\{\beta_i:\beta_i\neq 0, 1\leq i\leq 3\}|\leq 1.$$
Without loss of generality, we assume that 
\begin{equation}\label{type_D_inducing_data}
\beta_1\neq 0,
\beta_2\neq 0.
\end{equation}
 Moreover, we assume that 
\begin{equation}\label{type_D_ordering}
\begin{array}{cc}
  -n=\nu_{\Lambda_1}(\beta_1)\leq
  \nu_{\Lambda_2}(\beta_2)\leq 0&\  \text{if}\ \beta_3=0, \\
  -n=\nu_{\Lambda_1}(\beta_1)\leq
  \nu_{\Lambda_2}(\beta_2)\leq \nu_{\Lambda_3}(\beta_3)\leq 0&\
  \text{if}\
  \beta_3\neq 0.
\end{array}\end{equation}
We have $e(\Lambda_i)=2$, for $1\leq i\leq 3$, and 
\begin{align*}
  \Lambda_i(-1)=\Lambda_i(0)=\mathfrak{o}_Fv_i\ &
  \text{if}\ \nu_{F_0}(\lambda_i)=0,\\
  \Lambda_i(0)=\Lambda_i(1)=\mathfrak{o}_{F}v_i\ &\text{if}\
  \nu_{F_0}(\lambda_i)=1.
\end{align*}
As the lattice sequence $\Lambda$ depends on various possibilities on
$V_i$, we will describe these lattice sequences and a Witt-basis,
which gives a splitting for $\Lambda$, as required in each individual
case. However, the group $J^0(\Lambda, \beta)$ and
$H^1(\Lambda, \beta)$ are given by
$$P_0(\Lambda_\beta)P_{(q_2/2)}(\Lambda_{\beta_1})P_{(n/2)}(\Lambda)$$
and 
$$P_1(\Lambda_\beta)P_{(q_2/2)+}(\Lambda_{\beta_1})P_{(n/2)+}(\Lambda)$$
respectively.
\subsection{Criterion for non-emptiness of the set
  \texorpdfstring{$\mathfrak{X}_\beta(F_0)$}{}}
\label{type_D_unram_whit_data}
When $F/F_0$ is unramified, the non-emptiness of the set
$\mathfrak{X}_\beta(F_0)$ depends only on the integers
$\{\nu_F(\beta_i), \nu_{F_0}(\lambda_i): 1\leq i\leq 3\}$. This will
be made precise in the following lemmas. However, when $F/F_0$ is
ramified, one requires more information on
$\{\beta_1, \beta_2, \beta_3\}$ to determine whether the set
$\mathfrak{X}_\beta(F_0)$ empty or not. In the case where $F/F_0$ is
ramified we will not make these conditions explicit, but we will show
that a cuspidal representation in $\Pi_\mathfrak{x}$ is generic if and
only if $\mathfrak{X}_\beta(F_0)$ is non-empty.
\begin{lemma}\label{type_D_unram_uniform_x_beta}
  Let $F/F_0$ be an unramified extension and let $\mathfrak{x}$ be a
  stratum of type {\bf (D)} such that $(W_i, h)$ is isotropic, for
  $1\leq i\leq 3$. The set $\mathfrak{X}_\beta(F_0)$ is non-empty if and
  only if $\nu_F(\beta_1)-\nu_F(\beta_2)$ is an even integer.
\end{lemma}
\begin{proof}
  Since the extension $F/F_0$ is unramified, we may assume that
  $\lambda_i=1$, for $1\leq i\leq 3$. The set
  $\mathfrak{X}_\beta(F_0)$ is non-empty if and only if there exists a
  non-zero vector $v\in V$ such that
\begin{equation}\label{type_D_unram_uniform_x_beta_eq_2}
  h(v, v)=h(v, \beta v)=0.
  \end{equation}
  Let $v=av_1+bv_2+cv_3$, for some $a,b,c\in F$. From the equation
  \eqref{type_D_unram_uniform_x_beta_eq_2}, we get that
  \begin{equation}\label{type_D_unram_uniform_x_beta_eq_1}
    a\sigma(a)+b\sigma(b)+c\sigma(c)=0\ \text{and}\ 
    \beta_1a\sigma(a)+\beta_2b\sigma(b)+\beta_3c\sigma(c)=0.
  \end{equation}
  Using the assumption \eqref{type_D_ordering} on the skew semisimple
  stratum $\mathfrak{x}$ we get that
  $(1-\beta_2\beta_1^{-1})\in \mathfrak{o}_F^\times$ and
  $(1-\beta_3\beta_2^{-1})\in \mathfrak{o}_F^\times$.  The set of
  equations \eqref{type_D_unram_uniform_x_beta_eq_1} imply that
$$a\sigma(a)=-c\sigma(c)\beta_2\beta_1^{-1}
(1-\beta_3\beta_2^{-1})(1-\beta_2\beta_1^{-1})^{-1}.$$ Hence, we get
that $\nu_F(\beta_2)-\nu_F(\beta_1)$ is an even integer. Conversely, if
$\nu_F(\beta_2)-\nu_F(\beta_1)$ is even, we find can find a non-zero
tuple $(a, b, c)\in F^3$ satisfying the equalities in equation
\eqref{type_D_unram_uniform_x_beta_eq_1}; therefore, the set
$\mathfrak{X}_\beta(F_0)$ is non-empty.
\end{proof}
\begin{lemma}\label{type_D_unram_non_uni_x_beta}
  Let $F/F_0$ be an unramified extension and let $\mathfrak{x}$ be a
  stratum of type {\bf (D)} such that $(W_i, h)$ is anisotropic,
  for some $i$, $1\leq i\leq 3$. The set $\mathfrak{X}_\beta(F_0)$ is
  non-empty if and only if $\nu_{F}(\lambda_2)=\nu_{F}(\lambda_3)=1$
  and $\nu_{F}(\beta_1)-\nu_{F}(\beta_2)$ is an odd integer.
\end{lemma}
\begin{proof}
  Recall that the determinant of $(V, h)$ is the trivial class in
  $F_0^\times/{\rm Nr}_{F/F_0}(F^\times)$.  With the hypothesis on the
  spaces $W_i$, there exists an unique $i\in\{1,2,3\}$ such that
  $\nu_{F}(\lambda_i)=0$. The set $\mathfrak{X}_\beta(F_0)$ is
  non-empty if and only if the following equations
  \begin{equation}\label{type_D_unram_non_uni_x_beta_eq_1}
    \lambda_1a\sigma(a)+\lambda_2b\sigma(b)+
    \lambda_3c\sigma(c)=0
    \ \text{and}\ \beta_1\lambda_1a\sigma(a)+\beta_2
\lambda_2 b\sigma(b)+ \beta_3 \lambda_3 c\sigma(c)=0
\end{equation}have a non-trivial
simultaneous solution. Recall that $(1-\beta_2\beta_1^{-1})\in
\mathfrak{o}_F^\times$ and $(1-\beta_3\beta_2^{-1})\in
\mathfrak{o}_F^\times$.
Assume that there exists 
a non-trivial simultaneous solution to the equations in
\eqref{type_D_unram_non_uni_x_beta_eq_1}. Then, we get that
$$(1-\beta_2\beta_1^{-1})\lambda_2b\sigma(b)+
(1-\beta_3\beta_1^{-1})\lambda_3c\sigma(c)=0.$$ This implies that
$\nu_F(\lambda_2)=\nu_F(\lambda_3)$. From the assumption on the spaces
$W_i$, for $1\leq i\leq 3$, we get that
$\nu_F(\lambda_2)=\nu_F(\lambda_3)=1$ and $\nu_F(\lambda_1)=0$. From
the equation \eqref{type_D_unram_non_uni_x_beta_eq_1} we have
$$\lambda_1a\sigma(a)=
\lambda_3c\sigma(c)\beta_2\beta_1^{-1}
(1-\beta_3\beta_2^{-1})/(1-\beta_2\beta_1^{-1}).$$ Hence,
$\mathfrak{X}_\beta(F_0)$ is non-empty if and only if $(W_1, h)$ is
isotropic and $\nu_F(\beta_2)-\nu_F(\beta_1)$ is an odd integer.
\end{proof}
\subsection{Estimating the valuation of
  \texorpdfstring{$h(gv, \beta gv)$}{}}
As observed in the previous sections, our approach to show
non-genericity is by showing an appropriate inequality on the function
sending $g\in P(\Lambda)$ to $\nu_F(h(gv, gv))$, where $v$ is a well
chosen isotropic vector with respect to $P(\Lambda)$. Hence, we need
some technical lemmas to understand the growth of this function.
\begin{lemma}\label{type_D_unram_non_gen_val_sep}
  Let $F/F_0$ be an unramified extension and let $\mathfrak{x}$ be a
  skew semisimple stratum of type {\bf (D)} such that
  $\mathfrak{X}_\beta(F_0)$ is the empty set. Let $v$ be an isotropic
  vector in $V$, and let $g\in G$. Assume that
  $gv=av_1+bv_2+cv_3$ for some $a, b, c\in F$.  If $a\neq 0$ and
  $\beta_2\lambda_2b\sigma(b)+ \beta_3\lambda_3c\sigma(c)\neq 0$, then
  we have
$$\nu_F(h(gv, \beta gv))=\min\{\nu_F(\beta_1\lambda_1a\sigma(a)), 
\nu_F(\beta_2\lambda_2b\sigma(b)+ \beta_3\lambda_3c\sigma(c))\}.$$
\end{lemma}
\begin{remark}
  Since the set $\mathfrak{X}_\beta(F_0)$ is the empty-set, we get
  that $h(v, \beta v)\neq 0$, for all isotropic vectors $v\in V$.
  \end{remark}
  \begin{proof}[Proof of Lemma \ref{type_D_unram_non_gen_val_sep}]
  Before we begin the proof, it is useful to recall that the
  determinant of $(V, h)$ is the trivial class in
  $F_0^\times/{\rm Nr}_{F/F_0}(F^\times)$.  Since $v$ is an isotropic
  vector, we get that
  \begin{equation}\label{type_D_unram_non_gen_val_sep_eqn_1}
    \lambda_1a\sigma(a)+\lambda_2b\sigma(b)+
    \lambda_3c\sigma(c)=0.
  \end{equation}
  Observe  that rescaling the constants
$a, b, c$ does not effect the lemma. Hence, rescaling $a$, $b$ and
$c$, if necessary, we assume that $a, b, c\in \mathfrak{o}_F$ and the
$\mathfrak{o}_F$-ideal $(a, b, c)$ is equal to $\mathfrak{o}_F$.

First consider the case where $(W_i, h)$ is isotropic, for
$1\leq i\leq 3$.  In this case we have $\nu_{F_0}(\lambda_i)=0$, for
$1\leq i\leq 3$ and $\nu_{F}(\beta_1)-\nu_{F}(\beta_2)$ is an odd
integer. 
If $\nu_F(a)=0$, we get that
\begin{align*}&\nu_F(\beta_1\lambda_1a\sigma(a)+
  \beta_2\lambda_2b\sigma(b)+\beta_3\lambda_3c\sigma(c))\\
  =&\nu_F(\beta_1)+\nu_F(\lambda_1a\sigma(a)+
  \beta_2\beta_1^{-1}\lambda_2b\sigma(b)+\beta_3\beta_1^{-1}
  \lambda_3c\sigma(c))&\\
  =&\nu_F(\beta_1)\leq
  \nu_F(\beta_2\lambda_2b\sigma(b)+\beta_3\lambda_3c\sigma(c)).
\end{align*} 
Thus we prove the lemma in the case where $\nu_F(a)=0$. Consider the
case where $a\in \mathfrak{p}_F$; we necessarily have
$b, c\in \mathfrak{o}_F^\times$.  We have
$$\nu_F(\beta_2 \beta_1^{-1}\lambda_2b\sigma(b)+
\beta_3\beta_1^{-1}\lambda_3c\sigma(c)) =\nu_F(\beta_2\beta_1^{-1})+
\nu_F(\lambda_2b\sigma(b)+\lambda_3\beta_3\beta_2^{-1}c\sigma(c)).$$
Since $\mathfrak{x}$ is a skew semisimple stratum, we get that
$1-\beta_3\beta_2^{-1}\in \mathfrak{o}_F^\times$. This implies that
$\nu_F(\lambda_2b\sigma(b)+\lambda_3\beta_3\beta_2^{-1}c\sigma(c))=0$.
Observe that $\nu_F(a\sigma(a))$ is an even integer.  Therefore, we
conclude that
$$\nu_F(\beta_1\lambda_1a\sigma(a))\neq 
\nu_F(\beta_2\lambda_2b\sigma(b)+ \beta_3\lambda_3c\sigma(c)),$$ and
this proves the lemma in the case where $(W_i, h)$ is isotropic, for
$1\leq i\leq 3$.

Assume that $(W_i, h)$ is anisotropic for some $1\leq i\leq 3$. Using
Lemma \ref{type_D_unram_non_uni_x_beta}, the set
$\mathfrak{X}_\beta(F_0)$ is empty in either of the following cases:
case {\bf (I)} where $\nu_F(\lambda_2)\neq \nu_F(\lambda_3)$, case
{\bf (II)} where $\nu_F(\lambda_2)=\nu_F(\lambda_3)$ and
$\nu_F(\beta_1)-\nu_F(\beta_2)$ is an even integer. We first assume
that $\nu_{F_0}(\lambda_2)=0$, and we get that
$\nu_{F_0}(\lambda_1)=\nu_{F_0}(\lambda_3)=1$. We may also assume that
$\lambda_2=1$ and $\lambda_1=\lambda_3=\varpi$. Using equation
\eqref{type_D_unram_non_gen_val_sep_eqn_1}, we get that
$b\in \mathfrak{p}_F$, and set $b=\varpi b'$, for some
$b'\in \mathfrak{o}_F$.  Since we have
$$\varpi a\sigma(a)+\varpi^2b'\sigma(b')+\varpi c\sigma(c)=0,$$
we get that $a, c\in \mathfrak{o}_F^\times$.  We now have
$$\nu_F(h(ge_1, ge_1))=\nu_F(\beta_1)+
\nu_F((1-\beta_2\beta_1^{-1})\varpi^2 b'\sigma(b')+
(1-\beta_3\beta_1^{-1})\varpi c\sigma(c))$$
and note that 
$$\nu_F((1-\beta_2\beta_1^{-1})\varpi^2 b'\sigma(b')+
(1-\beta_3\beta_1^{-1})\varpi c\sigma(c))=1.$$ Hence, we get that
$\nu_F(h(ge_1, ge_1))=\nu_F(\beta_1)+1$.  Note that
$\nu_F(\beta_2\varpi^2 b'\sigma(b')+ \beta_3\varpi c\sigma(c))$ is
equal to $\nu_F(\beta_2)+1$. Hence, the lemma follows in case {\bf
  (I)}, from the observation that
$$\nu_F(\beta_1\lambda_1a\sigma(a))= 
\nu_F(\beta_1\varpi a\sigma(a))=
\nu_F(\beta_1)+1\leq \nu_F(\beta_2)+1.$$
The case where $\nu_F(\lambda_2)=1$--in which case $\nu_F(\lambda_1)=
\nu_F(\lambda_2)=1$ and $\nu_F(\lambda_3)=0$--is entirely similar. 

Assume that we are in case {\bf (II)}. In this case we may assume that
$\lambda_2=\lambda_3=\varpi$ and $\lambda_1=1$. Then the equation
\eqref{type_D_unram_non_gen_val_sep_eqn_1} implies that
$b, c\in \mathfrak{o}_F^\times$ and $a\in \mathfrak{p}_F$. We now have
$$h(gv, \beta gv)=
\beta_1\{a\sigma(a)+\beta_2\beta_1^{-1} \varpi
(b\sigma(b)+\beta_3\beta_2^{-1}c\sigma(c))\}.$$ Since $\mathfrak{x}$
is a skew semisimple strata and $b, c\in \mathfrak{o}_F^\times$, we
get that $\nu_F((b\sigma(b)+\beta_3\beta_2^{-1}c\sigma(c)))=0$. The
lemma now follows from the observation that the integer
$\nu_{F}(\lambda_1a\bar{a})$ is even and the integer
$\nu_F(\beta_2\beta_1^{-1}
\varpi_{F_0}(b\sigma(b)+\beta_3\beta_2^{-1}c\sigma(c)))$ is always odd.
\end{proof}
\begin{lemma}\label{type_D_ram_non_gen_val_sep}
  Let $F/F_0$ be a ramified extension and let $\mathfrak{x}$ be a
  stratum of type {\bf (D)} such that $\mathfrak{X}_\beta(F_0)$ is the
  empty set. Let $v$ be any isotropic vector in $V$ and let $g\in G$.
  Assume that $gv=av_1+bv_2+cv_3$ for some $a, b, c\in F$. If
  $a\neq 0$ and
  $\beta_2\lambda_2b\sigma(b)+ \beta_3\lambda_3c\sigma(c)\neq 0$, then
  we have
$$\nu_F(h(gv, \beta gv))=\min\{\nu_F(\beta_1\lambda_1a\sigma(a)), 
\nu_F(\beta_2\lambda_2b\sigma(b)+ \beta_3\lambda_3c\sigma(c))\}.$$
\end{lemma}
\begin{proof}
Since $v$ is an isotropic vector, we get that 
\begin{equation}\label{type_D_ram_non_gen_val_sep_eq_2}
  \lambda_1a\sigma(a)+\lambda_2b\sigma(b)+\lambda_3c\sigma(c)=0.
\end{equation}
Rescaling the constants $a$, $b$, and $c$, if necessary, we assume
that $a, b, c\in \mathfrak{o}_F$ and the $\mathfrak{o}_F$-ideal $(a,
b, c)$ is equal to $\mathfrak{o}_F$.

Since $\mathfrak{x}$ is a skew semisimple stratum, using the
assumptions in \eqref{type_D_ordering}, we get that
$(1-\beta_i\beta_j^{-1})\in \mathfrak{o}_F^\times$, for
$1\leq j<i\leq 3$.  The set $\mathfrak{X}_\beta(F_0)$ is empty in one
of the two cases
\begin{enumerate}
  \item The case where
$$-(1-\beta_2\beta_1^{-1})(1-\beta_3\beta_1^{-1})^{-1}
\lambda_2\lambda_3^{-1}\not\in {\rm Nr}_{F/F_0}(F^\times)$$
\item The case where
$$-(1-\beta_2\beta_1^{-1})(1-\beta_3\beta_1^{-1})^{-1}
\lambda_2\lambda_3^{-1}\in {\rm Nr}_{F/F_0}(F^\times)$$
and
\begin{equation}\label{type_D_ram_non_gen_val_sep_eq_3}
\lambda_3/\lambda_1\beta_2\beta_1^{-1}(1-\beta_3\beta_2^{-1})
(1-\beta_2\beta_1^{-1})^{-1}\not\in {\rm Nr}_{F/F_0}(F^\times).
\end{equation}
\end{enumerate}

In Case $(1)$, unless $b, c\in \mathfrak{p}_F$, we have
\begin{equation}\label{type_D_ram_non_gen_val_sep_eq_1}
  (1-\beta_2\beta_1^{-1})\lambda_2b\sigma(b)
  +(1-\beta_3\beta_1^{-1})\lambda_3c\sigma(c)\not\in \mathfrak{p}_F.
\end{equation}
Hence, we get that
$$\nu_F(h(gv, \beta gv))=\min\{\nu_F(\beta_1\lambda_1a\sigma(a)),
\nu_F(\beta_2\lambda_2b\sigma(b) +\beta_3\lambda_3c\sigma(c))\}.$$ 

Now consider Case $(2)$ and assume that
\begin{equation}\label{type_D_ram_non_gen_val_sep_eq_4}
  (1-\beta_2\beta_1^{-1})\lambda_2b\sigma(b)
  +(1-\beta_3\beta_1^{-1})\lambda_3c\sigma(c)\in \mathfrak{p}_F;
\end{equation}
if the condition \eqref{type_D_ram_non_gen_val_sep_eq_4} is false we
get the lemma immediately.  If $\nu_F(\beta_1)=\nu_F(\beta_2)$, then
we have
$$-(1-\beta_2\beta_1^{-1})\beta_2^{-1}\beta_1\lambda_1a\sigma(a)+
(1-\beta_3\beta_2^{-1})\lambda_3c\sigma(c)\in \mathfrak{p}_F$$ But,
the above containment is a contradiction to the second condition in
case $(2)$, the equation
\eqref{type_D_ram_non_gen_val_sep_eq_3}. Hence, we get that
$\nu_F(\beta_1)\neq \nu_F(\beta_2)$.  If $a\in \mathfrak{o}_F^\times$,
the valuation of $h(gv, \beta gv)$ is equal to
$$\nu_F(\beta_1)+\nu_F(\lambda_1a\sigma(a)+
\beta_2\beta_1^{-1}(\lambda_2b\sigma(b)+
\beta_3\beta_2^{-1}\lambda_3c\sigma(c)))=
\nu_F(\beta_1).$$ Note that
the lemma follows from the observation that
$$\nu_F(\beta_1a\sigma(a))\leq
\nu_F(\beta_2\lambda_2b\sigma(b)+\beta_3\lambda_3c\sigma(c)).$$ We
consider the case where $a\in \mathfrak{p}_F$ and this implies that
$b, c\in \mathfrak{o}_F^\times$.  Thus we have
$\lambda_2b\sigma(b)+ \beta_3\beta_2^{-1}\lambda_3c\sigma(c)\in
\mathfrak{o}_F^\times$. Suppose
$$\lambda_1a\sigma(a)+\beta_2\beta_1^{-1}(\lambda_2b\sigma(b)+
\beta_3\beta_2^{-1}\lambda_3c\sigma(c))\in \mathfrak{p}_F^{2\nu_F(a)+1},$$ 
then we have 
$$\beta_1\beta_2^{-1}\lambda_1a\sigma(a)
+\lambda_2b\sigma(b)+\beta_3\beta_2^{-1}
\lambda_3c\sigma(c)\in \mathfrak{p}_F.$$
Using the equality \eqref{type_D_ram_non_gen_val_sep_eq_2}, we get that
$$(1-\beta_1\beta_2^{-1})\lambda_1a\sigma(a)+
(1-\beta_3\beta_2^{-1})\lambda_3c\sigma(c)\in \mathfrak{p}_F.$$ This is
a contradiction to the condition
\eqref{type_D_ram_non_gen_val_sep_eq_3}. Hence, we obtain
$$\nu(h(gv, \beta gv))=\min\{\nu_F(\beta_1\lambda_1a\sigma(a)), 
\nu_F(\beta_2\lambda_2b\sigma(b)+ \beta_3\lambda_3c\sigma(c))\}.$$
 \end{proof}
 \subsection{Generic cuspidal representations of type
   \texorpdfstring{{\bf (D)}}{}}
\begin{lemma}\label{type_D_gen_main}
  Let $F/F_0$ be any quadratic extension and let $\mathfrak{x}$ be a
  stratum of type {\bf (D)} such that $\mathfrak{X}_\beta(F_0)$ is
  non-empty. Every cuspidal representation contained in the set
  $\Pi_\mathfrak{x}$ is generic.
\end{lemma}
\begin{proof}
  Let $\theta$ be any semisimple character in
  $\mathcal{C}(\Lambda, 0, \beta)$. Let $\sch{U}$ be the unipotent
  radical of a Borel subgroup in the set $\mathfrak{X}_\beta(F_0)$.
  We will first show that
$$\res_{H^1(\Lambda, \beta)\cap U}\theta=\psi_\beta.$$
If $q_1=q_2$, then the group $H^1(\Lambda, \beta)$ is equal to
$P_1(\Lambda_\beta)P_{q_1/2}(\Lambda)$ and from Lemma \ref{bl_st_lem},
we get that $H^1(\Lambda, \beta)\cap U$ is equal to
$P_{q_1/2+}(\Lambda)\cap U$. In the case where $q_1>q_2$, the group
$H^1(\Lambda, \beta)$ is equal to
$$P_1(\Lambda_\beta)P_{(q_2/2)+}(\Lambda_{\beta_1})P_{(q_1/2)+}(\Lambda).$$
Let $H'$ be the group $P_1(\Lambda_\beta)P_{(q_2/2)+}(\Lambda)$.  From
Lemma \ref{bl_st_lem} we get that the group $H'\cap U$ is equal to
$P_{(q_2/2)+}(\Lambda)\cap U$ and hence, $H^1(\Lambda, \beta)\cap U$ is
equal to $P_{(q_2/2)+}(\Lambda_{\beta_1})P_{(q_1/2)+}(\Lambda)\cap U$.

Let $v=av_1+bv_2+cv_3$ be an isotropic vector fixed by $U$.  Since
$\mathfrak{x}$ is a skew semisimple stratum of type {\bf (D)}, we get
that $a\neq 0$.  Assume that $g_1g_2\in U$, for some
$g_1\in P_{(q_2/2)+}(\Lambda_{\beta_1})$ and
$g_2\in P_{(q_1/2)+}(\Lambda)$. From the equality
$$g_1(av_1+bv_2+cv_3)=a'av_1+g_1(bv_2+cv_3)=g_2(av_1+bv_2+cv_3)$$
we get that 
$${\bf 1}_{\langle v_1\rangle }g_1{\bf 1}_{\langle v_1\rangle}=a'\in F^\times \cap
P_{(q_1/2)+}(\Lambda).$$
Which implies that the determinant of
${\bf 1}_{W_1}g_1{\bf 1}_{W_1}$ belongs to
$F^\times \cap P_{(q_1/2)+}(\Lambda)$. Hence, the simple character
$\theta(g_1g_2)$ is equal to $\psi_\beta(g_1g_2)$. Now, the lemma
follows from Proposition \ref{bl_st_prop}.
\end{proof}
 \subsection{Non-generic cuspidal 
representations of type
\texorpdfstring{{\bf (D)}}{}}
We will show that $\pi\in\Pi_\mathfrak{x}$ is non-generic if and only
if the set $\mathfrak{X}_\beta(F_0)$ is empty. We will devide the
proof into several cases beginning with the easier case where
$(W_1, h)$ is anisotropic; in which case we will show that
$\psi_\beta^g$ is non-trivial on
$P_{(n/2)+}(\Lambda)\cap U_{\rm der}$, for all $g\in P(\Lambda)$. In
the case where $(W_1, h)$ is isotropic, the method of proof is more
involved and we had to deal with conjugation of some shallow elements
in the group $P(\Lambda)$.
\subsubsection{The case where \texorpdfstring{$(W_1, h)$}{} is
  anisotropic}
\begin{lemma}\label{type_D_main_W_1_aniso}
  Let $F/F_0$ be a quadratic extension and let $\mathfrak{x}$ be a
  skew semisimple stratum of type {\bf (D)} such that $(W_1, h)$ is
  anisotropic. If the set $\mathfrak{X}_\beta(F_0)$ is empty, then
  every representation in the set $\Pi_\mathfrak{x}$ is non-generic.
\end{lemma}
\begin{proof}
Let us begin with the case where $F/F_0$ is ramified. In this case,
$e(\Lambda)=2$, and we have
$$\Lambda(-1)=\Lambda(0)=\mathfrak{o}_Fv_1\oplus
\mathfrak{o}_Fv_2\oplus \mathfrak{o}_Fv_3.$$
It is possible that neither of the spaces $(W_2, h)$ and $(W_3, h)$
are isotropic. However, there exists an orthogonal
$\mathfrak{o}_F$-basis, $(\tilde{v}_2, \tilde{v}_3)$, for the lattice
$\mathfrak{o}_Fv_2\oplus \mathfrak{o}_Fv_3$ such that
$\langle v_1, \tilde{v}_2\rangle$ is isotropic. Using arguments
similar to Lemma \ref{type_B_ram_seq_split}, we can we choose a
Witt-basis $(e_1, e_{-1})$ of $\langle v_1, \tilde{v}_2\rangle$ and an
unit vector $e_0\in \langle v_1, \tilde{v}_2\rangle^\perp$ such that
the Witt-basis $(e_1, e_0, e_{-1})$ provides a splitting for the
lattice sequence $\Lambda$. In the basis $(e_1, e_0, e_{-1})$, we have
$$\Lambda(-1)=\Lambda(0)=\mathfrak{o}_Fe_1\oplus
\mathfrak{o}_Fe_0\oplus \mathfrak{o}_Fe_{-1}.$$

Let $\sch{U}$ be the unipotent radical of the Borel subgroup,
$\sch{B}$, such that $\langle e_1 \rangle$ is fixed by $B$. We set
$$-\nu_{\Lambda}(\beta)=4m+2,$$ for some integer $m$. Assume that
$\pi\in \Pi_\mathfrak{x}$ is a generic representation. Let
$(J^0(\Lambda, \beta), \kappa)$ be a Bushnell--Kutzko type contained
in $\pi$. There exists a $g\in P(\Lambda)$, and a non-trivial
character $\Psi$ of $U$ such that
$$\ho_{J^0(\Lambda, \beta)^g\cap U}(\kappa^g, \Psi)\neq 0.$$
Since $g$ normalises $P_r(\Lambda)$, for $r>0$, we get that the group
$P_{(n/2)+}(\Lambda)\cap U_{\rm der}$ is equal to
$U_{\rm der}(\lceil m/2\rceil)$.  Let $ge_1=av_1+bv_2+cv_3$ for some
$a, b, c\in \mathfrak{o}_F$.  Then we have
$$\lambda_1a\sigma(a)+\lambda_2b\sigma(b)+\lambda_3c\sigma(c)=0.$$ If
$a\in \mathfrak{p}_F$, then $b, c\in \mathfrak{p}_F$ as $(W_1, h)$ is
anisotropic. Since $\nu_\Lambda(e_1)=0$, we get that
$a\in \mathfrak{o}_F^\times$. Now, we have
\begin{align*}
h(ge_1, \beta ge_1)&=\beta_1(\lambda_1a\sigma(a)+
\lambda_2\beta_2\beta_1^{-1}b\sigma(b)+
\lambda_3\beta_3\beta_1^{-1}c\sigma(c))\\
&=\beta_1(\lambda_1(1-\beta_3\beta_1^{-1})a\sigma(a)+
\lambda_2\beta_2\beta_1^{-1}(1-\beta_3\beta_2^{-1})b\sigma(b)).
\end{align*}
If $\nu_F(\beta_2)>\nu_F(\beta_1)$, then we get that
$\nu_F(h(ge_1, \beta ge_1))$ is equal to $\nu_F(\beta_1)$. Assume that
$\nu_F(\beta_1=\nu_F(\beta_2)$. Since, both the constants $b, c$
cannot be in $\mathfrak{p}_F$, without loss of generality, we assume
that $b\in \mathfrak{o}_F^\times$. If
$$\lambda_1(1-\beta_3\beta_1^{-1})a\sigma(a)+
\lambda_2\beta_2\beta_1^{-1}(1-\beta_3\beta_2^{-1})b\sigma(b)\in
\mathfrak{p}_F,$$
then we get that 
$$\overline{-\lambda_2\lambda_3^{-1}b\sigma(b)(1-\beta_2\beta_1^{-1})
  (1-\beta_3\beta_1^{-1})^{-1}}\in (k_F^\times)^2;$$
therefore, we get a contradiction to the assumption that the set
$\mathfrak{X}_\beta(F_0)$ is empty.  Hence, we have
$$\nu_{F/F_0}(\delta h(ge_1, \beta ge_1))=\nu_{F/F_0}(\beta_1)+1/2=m+1.$$
Thus we get that
$$\nu_{F/F_0}(h(ge_1, \beta ge_1))\leq -\lceil m/2\rceil.$$
Now, Lemma \ref{basic_inequality} implies that the character
$\psi_\beta^g$ is non-trivial on the group
$P_{(n/2)+}(\Lambda)\cap U_{\rm der}$, and we get a contradiction to
the assumption that $\pi$ is generic.

Consider the case where $F/F_0$ is unramified. In this case, we may
assume that $\lambda_1=\varpi$ and
$(\lambda_2, \lambda_3)\in \{(\varpi, 1), (1, \varpi)\}$. So we define
$\tilde{v}_3$ to be the vector in the set $\{v_2, v_3\}$ with
$h(\tilde{v}_3, \tilde{v}_3)=\varpi$, and the remaining vector in the
set $\{v_2, v_3\}$ is denoted by $\tilde{v}_2$. The notation
$\tilde{\beta}_i$ will be used for
${\bf 1}_{\langle\tilde{v}_i\rangle}\beta {\bf
  1}_{\langle\tilde{v}_i\rangle}$,
for $i\in \{2,3\}$. The period $2$ lattice sequence $\Lambda$ is given
by
$$\Lambda(0)=\mathfrak{o}_Fv_1\oplus\mathfrak{o}_F\tilde{v}_2
\oplus \mathfrak{o}_F\tilde{v}_3\ \text{and}\
\Lambda(1)=\mathfrak{o}_Fv_1\oplus\mathfrak{p}_F\tilde{v}_2 \oplus
\mathfrak{o}_F\tilde{v}_3.$$
Let $e_0$ be the vector $\tilde{v}_2$. Since the space
$\langle v_1, \tilde{v}_3 \rangle$ is isotropic, and there exists a
Witt-basis $(e_1, e_{-1})$ for the space
$\langle v_1, \tilde{v}_3 \rangle$ such that
$$\Lambda(0)=\mathfrak{o}_Fe_1\oplus\mathfrak{o}_Fe_0
\oplus \mathfrak{p}_Fe_{-1}\ \text{and}\
\Lambda(1)=\mathfrak{o}_Fe_1\oplus\mathfrak{p}_Fe_0 \oplus
\mathfrak{o}_Fe_{-1}.$$

Let $\sch{U}$ be the unipotent radical of the Borel subgroup,
$\sch{B}$, such that $\langle e_1 \rangle$ is fixed by $B$. Assume
that $\pi\in \Pi_\mathfrak{x}$, containing a Bushnell--Kutzko type
$(J^0(\Lambda, \beta), \kappa)$, is a generic representation. There
exists a $g\in P(\Lambda)$ and a non-trivial character $\Psi$ of $U$
such that
$$\ho_{J^0(\Lambda, \beta)^g\cap U}(\kappa^g, \Psi)\neq 0.$$
Note that $\nu_\Lambda(e_1)=1$ and hence we get that $ge_1=av_1+\varpi
b\tilde{v}_2+c\tilde{v}_3$, for some $a, b, c\in \mathfrak{o}_F$.
Since $e_1$ is an isotropic vector, we get that 
$$a\sigma(a)+\varpi b\sigma(b)+c\sigma(c)=0.$$
From the above equality and the fact that $\nu_\Lambda(e_1)=1$, we get
that $a, c\in \mathfrak{o}_F^\times$.  We set $n=4m+2r$, for some
integer $m$ and $r\in \{0,1\}$.  The valuation of
$h(ge_1, \beta ge_1)$ is equal to
$$\nu_F(\beta_1)+\nu_F(a\sigma(a)+\tilde{\beta}_2\beta_1^{-1}b\sigma(b)+
\tilde{\beta}_3\beta_1^{-1}c\sigma(c))=\nu_F(\beta_1)=-(2m+r).$$
We observe that $g$ normalises the group $P_{(n/2)+}(\Lambda)$;
therefore, we get that $P_{(n/2)+}(\Lambda)^g\cap U_{\rm der}$ is
equal to $U_{\rm der}(m)$. Since $\nu_F(h(ge_1, \beta ge_1))\leq -m$,
using Lemma \ref{basic_inequality}, we get that the character
$\psi_\beta^g$ is non-trivial on the group
$P_{(n/2)+}(\Lambda)^g\cap U_{\rm der}$. Thus we obtain a
contradiction to the assumption on the genericity of the
representation $\pi\in \Pi_\mathfrak{x}$.
\end{proof}
\subsubsection{The case where \texorpdfstring{$(W_1, h)$}{} is
  isotropic}\label{type_D_W_1_iso_prelim}
\paragraph{\it Lattice sequences:}
We begin with the description of the lattice sequence $\Lambda$.  Let
$(e_1, e_{-1})$ be a Witt-basis for the space $(W_1, h)$, and $e_0$ be
an unit vector in $V_1$.  Let $\sch{U}$ be the unipotent radical of
the Borel subgroup, $\sch{B}$, such that $\langle e_1 \rangle$ is
fixed by $B$. Let $\sch{T}$ be the maximal torus of $\sch{G}$,
contained in $\sch{B}$ such that $T$ fixes the decomposition
$$\langle e_1\rangle\oplus \langle e_0\rangle \oplus \langle e_{-1}\rangle.$$

If $F/F_0$ is unramified and $h(v_2, v_2)=h(v_3, v_3)=\varpi$, then
$e(\Lambda)=2$ and
$$\Lambda(0)=\mathfrak{o}_Fe_1\oplus
\mathfrak{o}_Fe_0\oplus \mathfrak{p}_Fe_{-1}\ \text{and}\
\Lambda(1)=\mathfrak{o}_Fe_1\oplus \mathfrak{p}_Fe_0\oplus
\mathfrak{p}_Fe_{-1}.$$
If $F/F_0$ is unramified and $h(v_2, v_2)=h(v_3, v_3)=1$, then
$e(\Lambda)=2$ and
$$
  \Lambda(-1)=\Lambda(0)=\mathfrak{o}_Fe_1\oplus
  \mathfrak{o}_Fe_0\oplus \mathfrak{o}_Fe_{-1}.
$$
If $F/F_0$ is ramified, then $e(\Lambda)=2$ and
$$
  \Lambda(-1)=\Lambda(0)=\mathfrak{o}_Fe_1\oplus
  \mathfrak{o}_Fe_0\oplus \mathfrak{o}_Fe_{-1}.
$$
Let $\mathcal{I}$ be the Iwahori subgroup defined, in its matrix form
with respect to the basis $(e_1, e_0, e_{-1})$, by
$$\begin{pmatrix}\mathfrak{o}_F&\mathfrak{o}_F&\mathfrak{o}_F\\
  \mathfrak{p}_F&\mathfrak{o}_F&\mathfrak{o}_F\\
  \mathfrak{p}_F&\mathfrak{p}_F&\mathfrak{o}_F
\end{pmatrix}\cap G.$$

Recall that we use the notation $q_i$ for the integer
$-\nu_{\Lambda_i}(\beta_i)$, for $1\leq i\leq 3$.  We set
$q_1=4m_1+2r_1$ and $q_2=4m_2+2r_2$, for some integers $m_1, m_2$ and
$r_1, r_2\in \{0,1\}$. Also recall our convention that $q_1\geq q_2$.

\paragraph{\it Shallow elements:}
For the purpose of understanding linear functionals supported on
$J^0(\Lambda, \beta)u U$, for $u\in \mathcal{I}$, we need a certain
measure of depth of an element $u$ with respect to the group
$P_{(n/2)+}(\Lambda)$.  Let $x, y\in \mathfrak{o}_F$ be two elements
such that $x\sigma(x)+y+\sigma(y)=0$, and let $w$ be an element in
$W_G$. The function $d(\mathfrak{x}, w, x)$, defined below, measures
the depth of an element $u(x, y)^w$ with respect to
$P_{(n/2)+}(\Lambda)$. When $F/F_0$ is an unramified extension, we set
\begin{eqnarray}
  d(\mathfrak{x}, w, x)=
  \begin{cases}
    \max\{m_2+1, m_1+1-\nu_F(x)\},&\ \text{if}\ \Lambda(0)\cap
    V_2=\mathfrak{o}_Fe_1\oplus \mathfrak{o}_Fe_{-1},\\
    \max\{m_2, m_1+r_1-\nu_F(x)\}&\ \text{if}\ \Lambda(0)\cap
    V_2=\mathfrak{o}_Fe_1\oplus \mathfrak{p}_Fe_{-1}, w=\id,\\
    \max\{m_2+2, m_1+r_1+1-\nu_F(x)\}&\ \text{if}\ \Lambda(0)\cap
    V_2=\mathfrak{o}_Fe_1\oplus \mathfrak{p}_Fe_{-1}, w\neq \id.
  \end{cases}
\end{eqnarray}
When $F/F_0$ is a ramified extension, the lattice sequences $\Lambda$
is uniquely determined (see subsection \ref{type_D_W_1_iso_prelim}).
We define the integer $d(\mathfrak{x}, w, x)$ as follows:
\begin{equation}
  d(\mathfrak{x}, w, x)=
  \max\{\lceil m_2/2\rceil, \lceil
  m_1/2-\nu_F(x)\rceil\}
\end{equation}
\begin{lemma}\label{type_D_unram_weight_der}
  Let $\mathfrak{x}$ be a skew semisimple stratum of type {\bf (D)}.
  Let $w$ be an element of $W_G$.  Let $u=\dot{u}(x,y)$ be an element
  of $\mathcal{I}\cap \overline{U^w}$, where $\overline{\sch{U}^w}$ is
  the unipotent radical of the opposite Borel subgroup of $\sch{B}^w$
  with respect to $\sch{T}$. We have
  $$U^w_{{\rm der}}(d(\mathfrak{x}, w, x))\subseteq
  H^1(\Lambda, \beta)^{u}\cap U^w_{{\rm der}}.$$
\end{lemma}
\begin{proof}
  We will prove the above lemma when $w=\id$, and the proof is
  entirely similar in the case where $w\neq \id$.  The group
  $H^1(\Lambda, \beta)$ is equal to
  $$P_1(\Lambda_{F[\beta]})P_{(q_2/2)+}(\Lambda_{F[\beta_1]})
  P_{(n/2)+}(\Lambda).$$ Let $u=\bar{u}(x, y)$ be an element of
  $\mathcal{I}\cap \bar{B}$. Using the the matrix identity
\begin{equation}\label{type_D_unram_weight_der_eq_1}
\bar{u}(x, y)u(0, a)\bar{u}(-x, -y-x\sigma(x))=
\begin{pmatrix}1-a(x\sigma(x)+y)&a\sigma(x)&a\\
  ax(-y-x\sigma(x))&1+ax\sigma(x)&ax\\
  -ay(y+x\sigma(x))&a\sigma(x)y&ay+1\end{pmatrix}
\end{equation}
and the definition of $d(\mathfrak{x}, \id, x)$, we see that the group
$\{U_{{\rm der}}(d(\mathfrak{x}, \id, x))\}^{u}$ is contained in the
group $H^1(\Lambda, \beta)$.
\end{proof}
\begin{lemma}
  With the same notations as in Lemma \ref{type_D_unram_weight_der},
  and for any skew semisimple character $\theta$ in
  $\mathcal{C}(\Lambda, 0, \beta)$ we have
  $$\res_{U^w_{\rm der}(d(\mathfrak{x}, w, x))}\theta^{u}=\psi_\beta^{u}.$$
\end{lemma}
\begin{proof}
  We prove this in the case where $\Lambda\cap W_1$ equal to
  $\mathfrak{o}_Fe_1\oplus \mathfrak{p}_Fe_{-1}$ and $w\neq \id$. The
  rest of the cases are similar and are simpler. Let $u^+$ be the
  element $u(x, y)\in \mathcal{I}\cap U$, and let $u=\bar{u}(0, a)$ be
  an element in the group $U^w_{\rm der}(d(\mathfrak{x}, w,
  x))$. Using the identity \eqref{type_D_unram_weight_der_eq_1}, we
  get that the element $u^+u(u^+)^{-1}$ is of the form $g_1g_2$, where
  $g_1\in P_{(q_2/2)+}(\Lambda_{F[\beta_1]})$ and
  $g_2\in P_{(n/2)+}(\Lambda)$. From the definition of the integer
  $d(\mathfrak{x}, w, x)$, we get that
  $1+ax\sigma(x)\in F^\times \cap P_{(n/2)+}(\Lambda)$; therefore, we
  get that ${\bf 1}_{V_1}g_1{\bf 1}_{V_1}$ belongs to
  $F^\times \cap P_{(n/2)+}(\Lambda)$. Hence, the determinant of
  ${\bf 1}_{W_1}g_1{\bf 1}_{W_1}$ belongs to
  $F^\times \cap P_{(n/2)+}(\Lambda)$. The lemma follows from the
  defining property, \cite[Definition 3.2.3(a)]{Orrangebook}, of the
  character $\theta$.
\end{proof}
\begin{lemma}\label{type_D_unram_non_gen}
  Let $F/F_0$ be a any quadratic extension and
  $\mathfrak{x}=[\Lambda, n, 0, \beta]$ be a stratum of type {\bf (D)}
  such that $\mathfrak{X}_\beta(F_0)$ is the empty set. Every cuspidal
  representation contained in the set $\Pi_\mathfrak{x}$ is
  non-generic.
\end{lemma}
\begin{proof}
  Let $\pi$ be a cuspidal representation in the set
  $\Pi_\mathfrak{x}$.  Let $(J^0(\Lambda, \beta), \kappa)$ be a
  Bushnell--Kutzko's type contained in the representation
  $\pi$. Assume that the representation $\pi$ is generic. Then there
  exists a $w\in W_G$, an element
  $u=\dot u(x, y)\in \mathcal{I}\cap \overline{U^w}$, and a
  non-trivial character $\Psi$ of $U$ such that
$$\ho_{J^0(\Lambda, \beta)^{u}\cap U^w}(\kappa^{u}, \Psi^w)\neq 0.$$
In particular, the above identity implies that 
\begin{equation}\label{type_D_W_1_isotropic_non-gen_gen_cond}
  \ho_{H^1(\Lambda, \beta)^{u}\cap U^w_{\rm der}}(\theta^{u},
  \id)\neq 0,
  \end{equation}
  where $\theta$ is the skew semisimple character contained in
  $\res_{H^1(\Lambda, \beta)}\kappa$.

  Consider the case where $F/F_0$ is any quadratic extension and
  $\Lambda\cap W_1$ is equal to $\mathfrak{o}_Fe_1\oplus
  \mathfrak{o}_Fe_{-1}$. In this case, we have $\nu_\Lambda(e_w)=0$;
  therefore, we get that $ue_w=bv_2+xv_1+cv_3$, for some $ b, c\in
  \mathfrak{o}_F$. If $\nu_{F/F_0}(x)=0$, then we have $$\nu_{F/F_0}(h(ue_w, \beta
  ue_w))=\nu_{F/F_0}(\beta_1).$$ If $\nu_{F/F_0}(x)>0$, then we have $b, c\in
  \mathfrak{o}_F^\times$. Since $\mathfrak{x}$ is skew semisimple
  stratum, we have $(1-\beta_2\beta_3^{-1})\in \mathfrak{o}_F^\times$.
  From the assumption that $\nu_{F/F_0}(x)>0$, we get that
  $\lambda_2b\sigma(b)+\lambda_3c\sigma(c)\in \mathfrak{p}_F$. Hence,
  we have
  $$\nu_{F/F_0}(\lambda_2b\sigma(b)+\beta_3\beta_2^{-1}\lambda_3c\sigma(c))=0.$$  
  Using Lemmas \ref{type_D_unram_non_gen_val_sep} and
  \ref{type_D_ram_non_gen_val_sep}, we get that
\begin{align*}
  \nu_{F/F_0}(h(ue_w, \beta ue_w))&=\min\{\nu_{F/F_0}(\beta_1)+2\nu_{F/F_0}(x),
  \nu_{F/F_0}(\beta_2\lambda_2b\sigma(b)+\beta_3\lambda_3c\sigma(c))\}\\
  &=\min\{\nu_{F/F_0}(\beta_1)+2\nu_{F/F_0}(x),
  \nu_{F/F_0}(\beta_2)+\nu_{F/F_0}(\lambda_2b\sigma(b)+
  \beta_3\beta_2^{-1}\lambda_3c\sigma(c))\}\\
  &=\min\{\nu_{F/F_0}(\beta_1)+2\nu_{F/F_0}(x), \nu_{F/F_0}(\beta_2)\}.
\end{align*}

Consider the case where $F/F_0$ is unramified and $\Lambda\cap W_1$ is
equal to $\mathfrak{o}_Fe_1\oplus \mathfrak{p}_Fe_{-1}$. In this case,
we have $\nu_\Lambda(e_{\pm 1})=\pm 1$; therefore, we get that
\begin{eqnarray}
ue_w=
\begin{cases}
xv_1+bv_2+cv_3,&\ b, c\in \mathfrak{o}_F,\  \text{if}\ w=\id\\
xv_1+bv_2+cv_3, &\  b, c\in \mathfrak{p}^{-1}_F,\  \text{if}\ w\neq\id.
\end{cases}
\end{eqnarray}
If $w=\id$, we observe that $b\sigma(b)+c\sigma(c)\in \mathfrak{p}_F$;
which together with $\nu_F(e_1)=1$ implies that $b, c\in
\mathfrak{o}_F^\times$. Since, $\mathfrak{x}$ is skew semisimple we
have get that $b\sigma(b)+\beta_3\beta_2^{-1}c\sigma(c)\in
\mathfrak{o}_F^\times$. From Lemma \ref{type_D_unram_non_gen_val_sep},
we get that
$$\nu_F(h(ue_1, ue_1))=\min\{\nu_F(\beta_1)+2\nu_F(x),
\nu_F(\beta_2)+1\}.$$
If $w\neq\id$, then similar arguments as above imply that
$\nu_F(b)=\nu_F(c)=-1$, and
$\nu_F(b\sigma(b)+\beta_3\beta_2^{-1}c\sigma(c))=-2$. Hence, using
lemma we get that
$$\nu_F(h(ue_1, ue_1))=\min\{\nu_F(\beta_1)+2\nu_F(x),
\nu_F(\beta_2)-1\}.$$
\begin{claim}\label{type_D_claim}
  We claim that
  $$\nu_{F/F_0}(\delta h(ue_w, ue_w))\leq -d(\mathfrak{x}, w, x).$$
\end{claim}
Assuming Claim \ref{type_D_claim} we complete the proof of the lemma.
Using Lemma \ref{basic_inequality}, we get that the character
$\psi_\beta^{u}$ is non-trivial on the group $U_{\rm
  der}(d(\mathfrak{x}, w, x))$. Thus, we get a contradiction to the
equation \eqref{type_D_W_1_isotropic_non-gen_gen_cond}, and hence, the
cuspidal representation $\pi$ is non-generic.

\begin{proof}[Proof of Claim \ref{type_D_claim}]
  {\bf Case 1:} First consider the case where $F/F_0$ is unramified
  and
  $$\Lambda(0)\cap W_1=\mathfrak{o}_Fe_1\oplus
  \mathfrak{o}_Fe_{-1}.$$ The integer $\nu_F(h(ue_1, ue_1))$ is equal
  to $\min\{-(2m_2+r_2), -(2m_1+r_1)+2\nu_F(x)\}$.  Recall that
  $-d(\mathfrak{x}, w, x)$ is equal to $\min\{-m_2-1,
  -m_1-1+\nu_F(x)\}$.  Observe that $-(2m_2+r_2)\leq -m_2-1$, unless
  $m_2+r_2=0$; since the stratum $\mathfrak{x}$ is a skew semisimple
  stratum of type {\bf (D)}, we have $m_2+r_2>0$ (see the assumption
  in the equation \eqref{type_D_inducing_data}). Assume that
  $-(2m_2+r_2)$ is equal to $\min\{-(2m_2+r_2),
  -(2m_1+r_1)+2\nu_F(x)\}$. Then we have
$$\nu_F(x)\geq m_1-m_2+(r_1-r_2)/2=
(m_1+1)-(2m_2+r_2)+(m_2-1+(r_1+r_2)/2)).$$ Using Lemma
\ref{type_D_unram_uniform_x_beta}, we get that $r_1-r_2$ is an odd
integer. Hence, we get that $m_2-1+(r_1+r_2)/2)\geq -1/2$.  Therefore,
we have
$$\nu_F(x)-m_1-1\geq -(2m_2+r_2).$$
Assume that $-(2m_1+r_1)+2\nu_F(x)$ is equal to $\min\{-(2m_2+r_2),
-(2m_1+r_1)+2\nu_F(x)\}$.
We have
$$\nu_F(x)\leq m_1-m_2+(r_1-r_2)/2\leq
m_1+r_1-1-(m_2+(r_1+r_2)/2-1).$$ Since, $r_1+r_2$ is odd, we get that
$m_2+(r_1+r_2)/2-1\geq -1/2$. Hence, we the inequality $\nu_F(x)\leq
m_1+r_1-1$ and we deduce that $-(2m_1+r_1)+2\nu_F(x)\leq
-(m_1+1)+\nu_F(x)$.  Finally, using the inequality
$$-(2m_1+r_1)+2\nu_F(x)\leq -(2m_2+r_2)\leq -(m_2+1),$$
we complete the verification of our claim in the present case.

{\bf Case 2:} Let us consider the case where $F/F_0$ is unramified and
 $$\Lambda(0)\cap W_1=\mathfrak{o}_Fe_1\oplus
 \mathfrak{p}_Fe_{-1}.$$ In this case, we have $r_1=r_2$. First,
 assume that $w=\id$. Then we have
 \begin{align*}
-d(\mathfrak{x}, \id, x)&=\min\{-m_2, -(m_1+r_1)+\nu_F(x)\},\\
 \nu_F(h(ue_1, ue_1))&=\min\{-(2m_2+r_2)+1,
 -(2m_1+r_1)+2\nu_F(x)\}.
\end{align*}
If $-(2m_2+r_2)+1$ is equal to $\nu_F(h(ue_1, ue_1))$, then we get
that
$$\nu_F(x)-m_1-r_1\geq -m_2-(r_1+r_2)/2+1/2\geq 
-(2m_2+r_2)+1+(m_2+(r_2-r_1)/2-1/2).$$ 
Since
$m_2+(r_2-r_1)/2-1/2\geq -1/2$, we get that 
$$\nu_F(x)-m_1-r_1\geq -(2m_2+r_2)+1.$$ 
If $-(2m_1+r_1)+2\nu_F(x)$ is equal to $\nu_F(h(ue_1, ue_1))$, then we
that
$$\nu_F(x)\leq m_1-m_2+1/2\leq m_1+1/2.$$
Since $\nu_F(x)$ is an integer, we get that $\nu_F(x)\leq
m_1$. Therefore, we get that
$$-(2m_1+r_1)+2\nu_F(x)\leq -(m_1+r_1)+\nu_F(x).$$
Hence, in the case where $w=\id$ we get that
$$\min\{-(2m_2+r_2)+1, -(2m_1+r_1)+2\nu_F(x)\}\leq -d(\mathfrak{x}, w, x).$$

Let us continue with the case considered in the previous paragraph but
with $w\neq \id$. We have 
\begin{align*}
-d(\mathfrak{x}, \id, x)=&
\min\{-(m_2+2), -(m_1+r_1+1)+\nu_F(x)\},\\
\nu_F(h(ue_{-1}, ue_{-1}))=&\min\{-(2m_2+r_2)-1,
-(2m_1+r_1)+2\nu_F(x))\}.
\end{align*}
If $-(2m_2+r_2)-1$ is equal to
$\nu_F(h(ue_{-1}, ue_{-1}))$, then we get that
$$\nu_F(x)\geq m_1-m_2-1/2.$$
From this we get that $\nu_F(x)\geq m_1$; therefore, we have
$$-(m_1+r_1+1)+\nu_F(x)\geq -(2m_2+r_2)-1.$$
If $-(2m_1+r_1)+2\nu_F(x)$ is equal to $\nu_F(h(ue_{-1}, ue_{-1}))$,
then we get that $\nu_F(x)\leq m_1-m_2-1/2$. Since $\nu_F(x)$ is an
integer, we get that $\nu_F(x)\leq m_1-1$. Hence, we complete the
verification of the claim in the present case.

{\bf Case 3:} Let us consider the case where $F/F_0$ is a ramified
extension.  In this case, we have
\begin{align*}
  -d(\mathfrak{x}, w, x)&=\min\{-\lceil m_2/2\rceil, -\lceil
  m_1/2-\nu_{F/F_0}(x)\rceil\},\\
  \nu_{F/F_0}(\delta h(ue_w, ue_w))&=\min\{-m_2, -m_1+2\nu_{F/F_0}(x)\}.
\end{align*}
We clearly have $-m_2\leq -\lceil m_2/2\rceil$, for $m_2\geq 0$. Now,
assume that $-m_2$ is equal to $ \nu_{F/F_0}(\delta h(ue_w, ue_w))$.  Then we
have
$$m_2\geq m_2/2\geq m_1/2-\nu_{F/F_0}(x).$$
From the above inequality we get that $m_2\geq \lceil
m_1/2-\nu_{F/F_0}(x)\rceil $. We assume that $-m_1+2\nu_{F/F_0}(x)$ is equal to $
\nu_{F/F_0}(\delta h(ue_w, ue_w))$. Then we get that $\nu_{F/F_0}(x)\leq
m_1/2-m_2/2$. Hence, we have $\nu_{F/F_0}(x)\leq m_1/2$ and this is
equivalent to the inequality
$$-m_1+2\nu_{F/F_0}(x)\leq -\lceil m_1/2-\nu_{F/F_0}(x) \rceil.$$
With this we complete the verification of the claim in all cases.
\end{proof}
\end{proof}
\section{The depth-zero case.}
When $F/F_0$ is unramified, the classification of generic depth-zero
cuspidal representations of $G$ can be deduced from the general work
of DeBacker--Reeder in the article \cite[Section
6]{depth-zero_debacker_reeder}. From their results, generic depth-zero
cuspidal representations are precisely the representations of the form
$$\ind_{P_0(\Lambda)}^G\sigma,$$
where $P_0(\Lambda)$ is a parahoric subgroup such that
$P_0(\Lambda)/P_1(\Lambda)$ is isomorphic to $U(2,1)(k_F/k_{F_0})$, and
$\sigma$ is the inflation of a cuspidal generic representation of
$P_0(\Lambda)/P_1(\Lambda)$. Now, we assume that $F/F_0$ is a ramified
extension and consider a cuspidal representation of $G$, isomorphic
to 
\begin{equation}\label{depth-zero_eq_1}
\ind_{P_0(\Lambda)}^G\sigma,
\end{equation} where $P^0(\Lambda)$ is a maximal
parahoric subgroup of $G$, and $\sigma$ is the inflation of a cuspidal
representation of $P_0(\Lambda)/P_1(\Lambda)$.

If $F/F_0$ is ramified, the groups $P_0(\Lambda)/P_1(\Lambda)$ is the
$k_F$-rational points of a disconnected reductive group over $k_F$. An
irreducible representation $\sigma$ of $P_0(\Lambda)/P_1(\Lambda)$ is
called a cuspidal representation if
$\res_{P^0(\Lambda)/P_1(\Lambda)}\sigma$ is a direct sum of cuspidal
representations.  An irreducible representation $\sigma$ of
$P_0(\Lambda)/P_1(\Lambda)$ is called generic if and only if its
restriction to a $p$-Sylow subgroup, say $H$, contains a non-trivial
character of $H$.

Let $(e_1, e_0, e_{-1})$ be any Witt-basis for $(V,h)$ then upto $G$
conjugation there are two lattice sequences $\Lambda_1$ and
$\Lambda_2$ such that $P^0(\Lambda_i)$ is a maximal parahoric
subgroup, for $i\in\{1,2\}$.  We have $e(\Lambda_i)=2$, for
$i\in\{1,2\}$, and
$$\Lambda_1(-1)=\Lambda_1(0)=
\mathfrak{o}_Fe_1\oplus \mathfrak{o}_Fe_0\oplus
\mathfrak{o}_Fe_{-1},$$
$$\Lambda_2(0)=\mathfrak{o}_Fe_1\oplus \mathfrak{o}_Fe_0\oplus 
\mathfrak{p}_Fe_{-1}\ \text{and}\ \Lambda_2(1)=\mathfrak{o}_Fe_1\oplus
\mathfrak{p}_Fe_0\oplus \mathfrak{p}_Fe_{-1}.$$ 

Let $\sch{B}$ be the Borel subgroup of $\sch{G}$ such that $\langle
e_1\rangle$ is fixed by $\sch{B}$. Let $\sch{U}$ be the unipotent
radical of $\sch{B}$.  The groups $P^0(\Lambda_1)$ and
$P^0(\Lambda_2)$ are special maximal compact subgroups of $G$, and we
have the Iwasawa decomposition $$G=P_0(\Lambda_i)B,$$ for
$i\in\{1,2\}$.  The representation of the form \eqref{depth-zero_eq_1}
is generic if and only if
$$\ho_{P_0(\Lambda)\cap U}(\sigma, \Psi)\neq 0,$$
for some character $\Psi$ of $U$.
\begin{lemma}\label{depth_zero_non_gen}
  Let $F/F_0$ be a ramified quadratic extension. A depth-zero cuspidal
  representation $\pi$ of $G$ is generic if and only if
  $$\pi\simeq\ind_{P_0(\Lambda_1)}^G\sigma,$$
  where $\sigma$ is a generic cuspidal representation of
  $P_0(\Lambda_1)/P_1(\Lambda_1)$.
\end{lemma}
\begin{proof}
  Let $\pi$ be a depth zero cuspidal representation isomorphic to
  $\ind_{P_0(\Lambda_1)}^G\sigma$. The image of $U\cap P_0(\Lambda_1)$
  in the quotient $P_0(\Lambda_1)/P_1(\Lambda_1)$ is the pro-$p$ Sylow
  subgroup of $P_0(\Lambda_1)/P_1(\Lambda_1)$.  Note that
  $P_0(\Lambda_1)\cap U$ is equal to
  $$\{u(x, y): x, y\in F, y+\sigma(y)+x\sigma(x)=0, \nu_{F/F_0}(y)\geq 0\},$$
  and the group $P_1(\Lambda_1)\cap U$ is equal to
  $$\{u(x, y): x, y\in F, y+\sigma(y)+x\sigma(x)=0, \nu_{F/F_0}(y)\geq 1/2\}.$$ 
  The quotient $(P_0(\Lambda_1)\cap U)/(P_1(\Lambda_1)\cap U)$ is
  isomorphic to $\{u(x, -x^2/2): x\in k_F\}$.  Let $\Psi$ be any
  non-trivial character of $U$ such that
  $\res_{P_0(\Lambda)\cap U}\Psi\neq \id$ and
  $\res_{P_1(\Lambda)\cap U}\Psi= \id$. For such a character $\Psi$,
 the space
  $$\ho_{P_0(\Lambda)\cap U}(\sigma, \Psi)\neq 0$$
  if and only if $\sigma$ is the inflation of a generic cuspidal
  representation of $P_0(\Lambda_1)/P_1(\Lambda_1)$. Hence, $\pi$ is
  generic if and only if $\sigma$ is generic.

  Let $\pi$ be a depth-zero cuspidal representation of the form
  $\pi\simeq \ind_{P_0(\Lambda_2)}^G\sigma$.  Assume that $\pi$ is
  generic, then we get that
\begin{equation}\label{depth-zero_eq_2}
\ho_{P_0(\Lambda_2)\cap U_{\rm der}}(\sigma, \id)\neq 0.
\end{equation}
The image of the group $P_0(\Lambda_2)\cap U_{\rm der}$ in the
quotient $P_0(\Lambda_2)/P_1(\Lambda_2)\simeq {\rm
  SL}_2(k_F)\times\{\pm 1\}$, is a $p$-Sylow subgroup, say $H$. Note
that $\res_H\sigma$ is a direct sum of non-trivial characters of
$H$. Thus we get a contradiction to the condition to the equation
\eqref{depth-zero_eq_2}. Hence, the representation $\pi$ is
non-generic.
\end{proof}
\section{Main theorem}
In this section, we combine the results obtained so far in the
following theorem. Recall the following notation: if $W$ is a
non-degenerate subspace of $V$, then ${\bf 1}_{W}$ is the projection
onto $W$ with kernel $W^\perp$.
\begin{theorem}\label{summary_main_theorem}
  Let $F$ be a non-Archimedean local
  field with odd residue characteristic.  Let
  $\mathfrak{x}=[\Lambda, n, 0, \beta]$ be a skew semisimple stratum
  and let $\Pi_\mathfrak{x}$ be the set of cuspidal representations
  containing a Bushnell--Kutzko type of the form
  $(J^0(\Lambda, \beta), \lambda)$. The cuspidal representations
  in the set $\Pi_\mathfrak{x}$ are either all generic or all
  non-generic. Furthermore, the following holds. 
\begin{enumerate}[label=(\Alph*)]
\item Let $\mathfrak{x}$ is a skew simple stratum, i.e., the case
  where $F[\beta]$ is a degree $3$ field extension of $F$.  Then the
  set $\mathfrak{X}_\beta(F_0)$ is non-empty, and every representation
  contained in the set $\Pi_\mathfrak{x}$ is generic.
\item Let $\mathfrak{x}$ be a skew semisimple stratum with the
  underlying splitting $V=V_1\perp V_2$ such that $\dim_FV_i=i$, for
  $i\in\{1,2\}$. Assume that $\beta=\beta_1+\beta_2$, where $\beta_i$
  is equal to ${\bf 1}_{V_i}\beta{\bf 1}_{V_i}$,
  $\sigma_h(\beta_i)=-\beta_i$, $F[\beta_2]$ is a degree $2$ field
  extension of $F$. Let $q_i$ be the integer
  $\nu_{\Lambda_i}(\beta_i)$, for $i\in\{1,2\}$.  If $q_1>q_2$, then a
  cuspidal representation in the set $\Pi_\mathfrak{x}$ is generic if
  and only if $(V_2, h)$ is isotropic. If $q_2>q_1$, then a cuspidal
  representation in the set $\Pi_\mathfrak{x}$ is generic if and only
  if $(V_2, h)$ is anisotropic. In this case, a cuspidal
  representation in the set $\Pi_\mathfrak{x}$ is generic if and only
  if the set $\mathfrak{X}_\beta(F_0)$ is non-empty.
\item Let $\mathfrak{x}$ be a skew semisimple stratum with the
  underlying splitting $V=V_1\perp V_2$ such that $\dim_FV_i=i$, for
  $i\in\{1,2\}$. We assume that $\beta=\beta_1+\beta_2$, where
  $\beta_i$ is equal to ${\bf 1}_{V_i}\beta{\bf 1}_{V_i}$,
  $\beta_i\in F$, and $\sigma(\beta_i)=-\beta_i$, for
  $i\in\{1,2\}$. Every representation in the set $\Pi_\mathfrak{x}$ is
  non-generic. The set $\mathfrak{X}_\beta(F_0)$ is non-empty if and only
  if $(V_2, h)$ is isotropic.
\item Let $\mathfrak{x}$ be a skew semisimple stratum with the underlying 
splitting $V=V_1\perp V_2\perp V_3$. Then a representation in the set 
$\Pi_\mathfrak{x}$ is generic if and only if $\mathfrak{X}_\beta(F_0)$ is 
non-empty. 
\end{enumerate}
\end{theorem}
In part $(2)$ of the above theorem, we have $q_1\neq q_2$ (see
paragraph \ref{type_B_valuations}).  In the case where the underlying
splitting of a skew semisimple strata $\mathfrak{x}$ is equal to
$V=V_1\perp V_2\perp V_3$, it is fairly easy to determine the
necessary and sufficient conditions on $\beta$ for the non-emptiness
of the set $\mathfrak{X}_\beta(F_0)$. Assume that
$\beta_i={\bf 1}_{V_i}\beta {\bf 1}_{V_i}$, for $1\leq i\leq 3$. When
$F/F_0$ is unramified, the non-emptiness of $\mathfrak{X}_\beta(F_0)$
depends only on the set of integers
$\{\nu_F(\beta_1), \nu_F(\beta_2), \nu_F(\beta_3)\}$ and the
isomorphism classes of $(V_i, h)$, for $1\leq i \leq 3$. We refer to
Lemmas \ref{type_D_unram_uniform_x_beta} and
\ref{type_D_unram_non_uni_x_beta} for these results.  However, in the
case where $F/F_0$ is ramified, one requires more invariants on
$\beta$ to determine whether $\mathfrak{X}_\beta(F_0)$ is empty or
not. Since these invariants are not the natural invariants attached to
a stratum, we did not make them explicit.  For details, we refer to
the proof of Lemma \ref{type_D_ram_non_gen_val_sep}.
\begin{proof}[Proof of Theorem \ref{summary_main_theorem}]
  We indicate the precise references to proofs of various parts
  enumerated in the theorem. The first part, case $(A)$, follows from
  Proposition \ref{type_A_end_prop}. In case $(B)$, when $q_1>q_2$,
  the corresponding statements are proved in Lemmas
  \ref{type_B_iso_q_1>q_2_generic}, \ref{type_B_aniso_q_1>q_2}, and
  \ref{type_B_ram_aniso_q_1>q_2}. In case $(B)$, when $q_2>q_1$, the
  corresponding statements are proved in Lemmas
  \ref{type_B_iso_q_1<q_2}, \ref{type_B_ram_iso_q_1>q_2_non_gen}, and
  \ref{type_B_aniso_q_2>q_1}. The statements in case $(C)$ are proved
  in Lemmas \ref{type_C_aniso} and \ref{type_C_iso_non-gen}.  The
  statements in case $(D)$ are proved in \ref{type_D_gen_main},
  \ref{type_D_main_W_1_aniso} and \ref{type_D_unram_non_gen}. Now, the
  main statement of the theorem on genericity or non-genericity of all
  representations in $\Pi_\mathfrak{x}$ follows from statements in
  cases $(A), (B), (C)$, and $(D)$.
  \end{proof}
\appendix
\section{Appendix: Filtration of
  \texorpdfstring{$U_{\rm der}$}{} induced by lattice sequences}
In this section we fix some representatives for $G$-conjugacy classes
of self-dual lattice sequences on $V$ and describe $a_{n}(\Lambda)$,
for $n\in \mathbb{Z}$. Then we use them to determine
$U_{\rm der}\cap a_n(\Lambda)$, for $n\in \mathbb{Z}$. These
calculations are used in showing certain representations are
non-generic. 
\subsection{The unramified case:}
We begin with
the case where $F/F_0$ is unramified. Let $\Lambda_1$ be the lattice
sequence of periodicity $2$ and
$$\Lambda_1(-1)=\Lambda_1(0)=\mathfrak{o}_Fe_1\oplus
\mathfrak{o}_Fe_0\oplus \mathfrak{o}_Fe_{-1}.$$
The filtration $\{a_n(\Lambda)\ |\ n\in \mathbb{Z}\}$ of $\End_F(V)$
is given by
\begin{equation}\label{unram_fil_1}
a_{2m-1}(\Lambda_1)=\varpi^{m}\begin{pmatrix}
  \mathfrak{o}_F&\mathfrak{o}_F&\mathfrak{o}_F\\
  \mathfrak{o}_F&\mathfrak{o}_F&\mathfrak{o}_F\\
  \mathfrak{o}_F&\mathfrak{o}_F&\mathfrak{o}_F
\end{pmatrix}\cap \mathfrak{g}\ \text{and}\ 
a_{2m}(\Lambda_1)=\varpi^{m}\begin{pmatrix}
  \mathfrak{o}_F&\mathfrak{o}_F&\mathfrak{o}_F\\
  \mathfrak{o}_F&\mathfrak{o}_F&\mathfrak{o}_F\\
  \mathfrak{o}_F&\mathfrak{o}_F&\mathfrak{o}_F
\end{pmatrix}\cap \mathfrak{g},
\end{equation}
for all $m\in \mathbb{Z}$. Let $\Lambda_2$ be a
period $2$ lattice sequence given by
$$\Lambda_2(0)=\mathfrak{o}_Fe_1\oplus
\mathfrak{o}_Fe_0\oplus \mathfrak{p}_Fe_{-1}\ \text{and}\
\Lambda_2(1)=\mathfrak{o}_Fe_1\oplus \mathfrak{p}_Fe_0\oplus
\mathfrak{p}_Fe_{-1}.$$ The filtration
$\{a_n(\Lambda_2)\ |\ n\in \mathbb{Z}\}$ is given by:
\begin{equation}\label{unram_fil_2}
a_{2m}(\Lambda_2)=\varpi^{m}\begin{pmatrix}
    \mathfrak{o}_F&\mathfrak{o}_F&\mathfrak{p}_F^{-1}\\
    \mathfrak{p}_F&\mathfrak{o}_F&\mathfrak{o}_F\\
    \mathfrak{p}_F&\mathfrak{p}_F&\mathfrak{o}_F
\end{pmatrix}\cap \mathfrak{g}\ \text{and}\ 
a_{2m+1}(\Lambda_2)=\varpi^{m}\begin{pmatrix}
  \mathfrak{p}_F&\mathfrak{o}_F&\mathfrak{o}_F\\
  \mathfrak{p}_F&\mathfrak{p}_F&\mathfrak{o}_F\\
  \mathfrak{p}_F^2&\mathfrak{p}_F&\mathfrak{p}_F
\end{pmatrix}\cap \mathfrak{g},\end{equation} for all $m\in \mathbb{Z}$. 
Let $\Lambda_3$ be the lattice sequence
of period $4$ given by
\begin{align*}
  \Lambda_3(-1)=\mathfrak{o}_Fe_1\oplus \mathfrak{o}_Fe_0\oplus
  \mathfrak{o}_Fe_{-1}, &\ \Lambda_3(0)=\mathfrak{o}_Fe_1\oplus
                          \mathfrak{o}_Fe_0\oplus\mathfrak{p}_Fe_{-1},\\
  \Lambda_3(1)=\mathfrak{o}_Fe_1\oplus \mathfrak{p}_Fe_0\oplus
  \mathfrak{p}_Fe_{-1},\ &\ \Lambda_3(2)=\mathfrak{p}_Fe_1\oplus
                           \mathfrak{p}_Fe_0\oplus \mathfrak{p}_Fe_{-1}.
\end{align*}
The filtration $\{a_n(\Lambda_3)\ |\ n\in \mathbb{Z}\}$ on
$\mathfrak{g}$ is given by:
\begin{eqnarray}\label{unram_fil_3}
  a_{4m+r}(\Lambda_3)=
\begin{cases}
  \varpi^{m}\begin{pmatrix}
    \mathfrak{o}_F&\mathfrak{o}_F&\mathfrak{o}_F\\
    \mathfrak{p}_F&\mathfrak{o}_F&\mathfrak{o}_F\\
    \mathfrak{p}_F&\mathfrak{p}_F&\mathfrak{o}_F
\end{pmatrix}\cap \mathfrak{g}\ \text{if}\ r=0,\\
\varpi^{m}\begin{pmatrix}
  \mathfrak{p}_F&\mathfrak{o}_F&\mathfrak{o}_F\\
  \mathfrak{p}_F&\mathfrak{p}_F&\mathfrak{o}_F\\
  \mathfrak{p}_F&\mathfrak{p}_F&\mathfrak{p}_F
\end{pmatrix}\cap \mathfrak{g}\ \text{if}\ r=1,\\
\varpi^{m}\begin{pmatrix}
\mathfrak{p}_F&\mathfrak{p}_F&\mathfrak{o}_F\\
\mathfrak{p}_F&\mathfrak{p}_F&\mathfrak{p}_F\\
\mathfrak{p}_F&\mathfrak{p}_F&\mathfrak{p}_F
\end{pmatrix}\cap \mathfrak{g}\ \text{if}\ r=2,\\
\varpi^{m}\begin{pmatrix}
\mathfrak{p}_F&\mathfrak{p}_F&\mathfrak{p}_F\\
\mathfrak{p}_F&\mathfrak{p}_F&\mathfrak{p}_F\\
\mathfrak{p}^2_F&\mathfrak{p}_F&\mathfrak{p}_F
\end{pmatrix}\cap \mathfrak{g}\ \text{if}\ r=3.
\end{cases}
\end{eqnarray}
Although, there is a lattice sequence, say $\Lambda_4$, with period
$6$, we do not need to write it down explicitly. This corresponds to
type {\bf (A)} strata and in this case all representations are
generic. The filtration $\{U_{\rm der}\cap a_n(\Lambda_1)\ |\
n\in \mathbb{Z}\}$ is given by:
$$U_{\rm der}\cap a_{2m-1}(\Lambda_1)=U_{\rm der}\cap a_{2m}(\Lambda_1)=
U_{\rm der}(m),$$ for $m\in \mathbb{Z}$. The filtration $\{U_{\rm
  der}\cap a_n(\Lambda_2)\ |\ n\in \mathbb{Z}\}$ is given by
$$U_{\rm der}\cap a_{2m}(\Lambda_2)=U_{\rm der}(m-1)\
 \text{and}\ U_{\rm der}\cap a_{2m+1}(\Lambda_2)=
U_{\rm der}(m).$$ The filtration $\Lambda_3$ is given by 
\begin{eqnarray}
U_{\rm der}\cap a_{4m+r}(\Lambda_2)=
\begin{cases}
U_{\rm der}(m)\ \text{if}\ r=0,\\
U_{\rm der}(m)\ \text{if}\ r=1,\\
U_{\rm der}(m)\ \text{if}\ r=2,\\
U_{\rm der}(m+1)\ \text{if}\ r=3. 
\end{cases}
\end{eqnarray}
\subsection{The ramified case:}
Now, assume that $F/F_0$ is a ramified extension and $\Lambda_1$ and
$\Lambda_2$ be the lattice sequence of period $2$ given by
$$\Lambda_1(-1)=\Lambda(0)=\mathfrak{o}_Fe_{1}\oplus   \mathfrak{o}_Fe_0
\oplus \mathfrak{o}_Fe_{-1}.$$ and
$$\Lambda_2(0)=\mathfrak{o}_Fe_{1}\oplus   \mathfrak{o}_Fe_0
\oplus \mathfrak{p}_Fe_{-1}\ \text{and}\
\Lambda_2(1)=\mathfrak{o}_Fe_{1}\oplus \mathfrak{p}_Fe_0 \oplus
\mathfrak{p}_Fe_{-1}$$ The filtration $\{a_n(\Lambda_1)\ |\ n\in
\mathbb{Z}\}$ is similar to the filtration in \eqref{unram_fil_1}.
The flitration $\{a_n(\Lambda_2)\ |\ n\in \mathbb{Z}\}$, in this case,
is similar to the filtration in \eqref{unram_fil_2}. We will not
require to write the filtrations $\{a_n(\Lambda')\ |\ n\in
\mathbb{Z}\}$ for which $P^0(\Lambda')$ is an Iwahori subgroup of $G$.
The filtration $\{U_{\rm der}\cap a_n(\Lambda_1)\ |\
n\in \mathbb{Z}\}$ is given by 
$$U_{\rm der}\cap a_{2m-1}(\Lambda_1)=
U_{\rm der}\cap a_{2m}(\Lambda_1)=U_{\rm der}([m/2]),$$
for all $m\in \mathbb{Z}$. 
The filtration $\{U_{\rm der}\cap a_n(\Lambda_2)\ |\
n\in \mathbb{Z}\}$ is given by 
$$U_{\rm der}\cap a_{2m-1}(\Lambda_2)=
U_{\rm der}\cap a_{2m}(\Lambda_2)=U_{\rm der}([(m-1)/2]),$$ for any
$m\in \mathbb{Z}$.

 \bibliography{../biblio}
\bibliographystyle{abbrv}
\noindent
Santosh Nadimpalli, IMAPP, Radboud Universiteit Nijmegen,
Heyendaalseweg 135, 6525AJ Nijmegen, The Netherlands.
\texttt{nvrnsantosh@gmail.com}, \texttt{Santosh.Nadimpalli@ru.nl}.
\end{document}